%% file: SHVI_Main.tex
\documentclass[11pt,reqno]{article}
\usepackage{epsfig,amssymb,latexsym,amsmath,pifont,multicol,mathtools,verbatim,amsthm,float,lipsum}
\usepackage{caption} 
\usepackage[utf8]{inputenc}
\usepackage[T1]{fontenc}

\usepackage[varg]{txfonts}
\let\mathbb=\varmathbb

\usepackage[sans]{dsfont}

\usepackage[left=2cm, right=2cm, top=2cm, bottom=2cm]{geometry}

\usepackage[%
cal=euler,
bb=fourier,
scr=zapfc,
]
{mathalfa}

\usepackage[font=small,labelfont=bf]{caption}
\usepackage{subfigure}
\usepackage{graphicx}
\graphicspath{{Figures/}} 
\usepackage{acronym}
\usepackage{latexsym}
\usepackage{paralist}
\usepackage{wasysym}
\usepackage{xspace}
\usepackage{framed}
\usepackage{palatino,pxfonts}
\usepackage{authblk}
\usepackage{enumitem}
\usepackage[numbers,sort&compress]{natbib}

\usepackage[dvipsnames,svgnames]{xcolor}
\colorlet{MyBlue}{DodgerBlue!75!Black}
\colorlet{MyGreen}{DarkGreen!95!Black}

\usepackage{hyperref}
\hypersetup{
colorlinks=true,
linktocpage=true,
pdfstartview=FitH,
breaklinks=true,
pdfpagemode=UseNone,
pageanchor=true,
pdfpagemode=UseOutlines,
plainpages=false,
bookmarksnumbered,
bookmarksopen=false,
bookmarksopenlevel=1,
hypertexnames=true,
pdfhighlight=/O,
urlcolor=MyBlue!60!black,linkcolor=MyBlue!70!black,citecolor=DarkGreen!70!black, 
pdftitle={},
pdfauthor={},
pdfsubject={},
pdfkeywords={},
pdfcreator={pdfLaTeX},
pdfproducer={LaTeX with hyperref}
}

\numberwithin{equation}{section}  
\usepackage[sort&compress,capitalize,nameinlink]{cleveref}
\crefname{example}{Ex.}{Exs.}

\crefrangeformat{equation}{\upshape(#3#1#4)\textendash(#5#2#6)}

\usepackage{mathtools}
\usepackage[para,online,flushleft]{threeparttable}
\usepackage{multirow}

\usepackage{physics}
\usepackage{float}
\newcommand{\eps}{\varepsilon}
\DeclareMathOperator*{\argmin}{argmin}
\DeclareMathOperator*{\argmax}{argmax}

\DeclareMathOperator{\zer}{zer}

\DeclareMathOperator{\indicator}{\iota}

\DeclareMathOperator{\dist}{dist}
\DeclareMathOperator{\dif}{d\!}

\DeclareMathOperator{\dom}{dom}

\DeclareMathOperator{\gr}{graph}

\DeclareMathOperator{\supp}{\mathtt{s}}
\DeclareMathOperator{\Var}{\mathsf{Var}}

\newcommand{\ca}{\mathtt{a}}

\DeclareMathOperator{\prox}{\textsf{prox}}

\renewcommand{\iff}{\Leftrightarrow}

\renewcommand{\emptyset}{\varnothing}
\newcommand{\eqdef}{\triangleq}

\newcommand{\scrA}{\mathcal{A}}
\newcommand{\scrB}{\mathcal{B}}
\newcommand{\scrC}{\mathcal{C}}
\newcommand{\scrD}{\mathcal{D}}
\newcommand{\scrE}{\mathcal{E}}
\newcommand{\scrF}{\mathcal{F}}

\newcommand{\scrH}{\mathcal{H}}

\newcommand{\scrK}{\mathcal{K}}
\newcommand{\scrL}{\mathcal{L}}

\newcommand{\scrN}{\mathcal{N}}

\newcommand{\scrS}{\mathcal{S}}

\newcommand{\scrU}{\mathcal{U}}

\newcommand{\scrX}{\mathcal{X}}
\newcommand{\scrY}{\mathcal{Y}}
\newcommand{\scrZ}{\mathcal{Z}}

\renewcommand{\Pr}{\mathbb{P}}
\newcommand{\Ex}{\mathbb{E}}

\newcommand{\filterF}{{\mathbb{F}}}
\newcommand{\filterH}{{\mathbb{H}}}

\newcommand{\1}{\mathbf{1}}

\newcommand{\R}{\mathbb{R}}

\newcommand{\Z}{\mathbb{Z}}
\newcommand{\N}{\mathbb{N}}

\DeclareMathOperator{\NC}{\mathsf{NC}}
\DeclareMathOperator{\TC}{\mathsf{TC}}

\newcommand{\distance}{\mathsf{d}}
\newcommand{\ball}{\mathbb{B}}

\newcommand{\opA}{\mathsf{A}}

\newcommand{\opM}{\mathsf{M}}

\newcommand{\opF}{{\mathsf{F}}}
\newcommand{\opV}{{\mathsf{V}}}

\DeclareMathOperator{\HVI}{HVI}

\DeclareMathOperator{\gap}{\mathsf{Gap}}

\usepackage[boxruled]{algorithm2e}
\theoremstyle{plain}
\newtheorem{theorem}{Theorem}

\newtheorem*{corollary*}{Corollary}
\newtheorem{lemma}[theorem]{Lemma}
\newtheorem{proposition}[theorem]{Proposition}

\theoremstyle{definition}
\newtheorem{definition}[theorem]{Definition}
\newtheorem*{definition*}{Definition}
\newtheorem*{problem*}{Problem}
\newtheorem{assumption}{Assumption}

\newcommand{\close}{\hfill{\footnotesize$\Diamond$}}

\theoremstyle{remark}
\newtheorem{remark}{Remark}
\newtheorem*{remark*}{Remark}
\newtheorem*{notation*}{Notational remark}
\newtheorem{example}{Example}

\numberwithin{theorem}{section}
\numberwithin{remark}{section}
\numberwithin{example}{section}

\DeclarePairedDelimiter{\inner}{\langle}{\rangle}

\title{Stochastic variance reduced extragradient methods for solving hierarchical variational inequalities}
\date{\today}

\author[1]{\small Pavel Dvurechensky}
\author[2]{\small Andrea Ebner} 
\author[2]{\small Johannes Carl Schnebel}
\author[3]{\small Shimrit Shtern}
\author[2]{\small Mathias Staudigl}

\affil[1]{\footnotesize Weierstrass Institute for Applied Analysis and Stochastics, Mohrenstr. 39, 10117 Berlin, Germany\\
(\href{mailto:Pavel.Dvurechensky@wias-berlin.de}{Pavel.Dvurechensky@wias-berlin.de})}
\affil[2]{\footnotesize Mannheim University, Department of Mathematics, B6 26, 68159 Mannheim, Germany\\
(\href{mailto:andrea.ebner@uni-mannheim.de}{andrea.ebner@uni-mannheim.de}, \href{mailto:johannes-carl.schnebel@uni-mannheim.de}{johannes-carl.schnebel@uni-mannheim.de}, \href{mailto:mathias.staudigl@uni-mannheim.de}{m.staudigl@uni-mannheim.de})}
\affil[3]{\footnotesize Faculty of Data and Decision Sciences, Technion - Israel Institute of Technology, Haifa, Israel\\
(\href{mailto:shimrits@technion.ac.il}{shimrits@technion.ac.il})}

\begin{document}

\maketitle

\begin{abstract}
We are concerned with optimization in a broad sense through the lens of solving variational inequalities (VIs) -- a class of problems that are so general that they cover as particular cases minimization of functions, saddle-point (minimax) problems, Nash equilibrium problems, and many others.  
The key challenges in our problem formulation are the two-level hierarchical structure and finite-sum representation of the smooth operators in each level.  For this setting, we are the first to prove convergence rates and complexity statements for variance-reduced stochastic algorithms approaching the solution of hierarchical VIs in Euclidean and Bregman setups. 
\end{abstract}

\section{Introduction}
\label{sec:Intro}
\input{Introduction}

\section{Euclidean Setup}
\label{sec:Euclid}
\input{Euclidian}

\section{Bregman Setup}
\label{sec:Breg}
\input{Bregman}

\section{Numerical Experiments}
\label{sec:Numerics}
\input{Numerics}

\newpage

\paragraph{Acknowledgements.} {This research benefited from the support of the FMJH Program PGMO. MST's research is supported by the Deutsche Forschungsgemeinschaft (DFG) - Projektnummer 556222748 "non-stationary hierarchical minimization". PD's research is supported by the Deutsche Forschungsgemeinschaft (DFG, German Research Foundation) under Germany's Excellence Strategy – The Berlin Mathematics Research Center MATH+ (EXC-2046/1, project ID: 390685689). }

\appendix


\section{Tools}
\label{app:Tools}
\input{App_General}

\section{Facts on variational inequalities}
\label{app:VIs}
\input{App_Fitzpatrick}

\section{Proofs of Section \ref{sec:Euclid}}
\label{app:ProofsEuclid}
\input{Proofs_Euclid}

\section{Proofs of Section \ref{sec:Breg}}
\label{app:ProofsBreg}
\input{Proofs_Breg}


\section{Detailed Numerical Experiments}
\label{app:Num}
\input{app_Numerics}

\bibliographystyle{plain}
\bibliography{PenaltyDynamics}
\end{document}

%% file: Introduction.tex
%
 
 Hierarchical optimization is an increasingly active research area with many applications in machine learning, see, e.g.,  \cite{franceschi18a,ji21c,vicol22a,bao21neurips,chen24a}. At the same time, the algorithmic solution of hierarchical variational inequalities (VIs) is a much less studied field, despite its applications going far beyond ML, including optimal control and mechanics \citep{Barbu84}, network/traffic and economic equilibrium modeling \citep{facchineiPang03,outrata04epec}, and noncooperative game-theoretic equilibrium selection \citep{samadiYousefian23,matsuo25}.
Our main goal is to advance algorithmic and complexity understanding of hierarchical VIs in the big data setting when the involved operators admit a finite sum representation. Particularly, we show that variance reduction techniques also work in this complex context and lead to an algorithm that has improved sample complexity compared to the state-of-the-art.

For the mathematical statement of the problem, let $\scrX$ be a finite dimensional real vector space with inner product $\inner{\cdot,\cdot}$ and corresponding norm $\norm{\cdot}$. We solve the following hierarchical VI
\begin{equation}\label{eq:P}\tag{P}
\text{ Find }x^{\ast}\in\scrS_{2} \text{ s.t. } \inner{\opF_{1}(x^{\ast}),x-x^{\ast}}+g_{1}(x)-g_{1}(x^{\ast})\geq 0 \quad\forall x\in \scrS_{2}\eqdef \zer(\opF_{2}+\partial g_{2}). 
\end{equation}
The data of this problem consist of monotone operators $F_{i}:\scrX\to\scrX,i=1,2$ and proper convex lower semi-continuous functions $g_{i}:\scrX\to(-\infty,\infty]$. Before we proceed, we give several motivating examples. We also refer to Section \ref{sec:Numerics} where we provide numerical tests of our algorithms.


\begin{example}[equilibrium selection] 
An important class of problems falling in the domain of \eqref{eq:P} is the equilibrium selection problem 
$$
\min_{x} f_{1}(x) +g_{1}(x) \quad\text{s.t. }x\in \textnormal{SOL}(\opF_{2},\scrK) 
$$
where $f_{1}:\scrX\to\R$ is Lipschitz smooth and $g_1:\scrX\to(-\infty,\infty]$ is a proper closed convex function, and $\scrK\subseteq\scrX$ a nonempty closed convex set. In this formulation, the composite convex function in the upper level defines a design criterion, used to select among the solution set of a variational inequality problem 
\[
\text{ Find }y\in\scrK\text{ s.t. : } \inner{\opF_{2}(y),x-y}\geq 0\qquad\forall x\in\scrK.
\]
In this case, finite-sum representation on the lower level naturally appears when this VI comes from matrix games or constrained optimization reformulated as a saddle-point problem \citep{alacaoglu2022stochastic,nesterovNemirovski13acta}. 
Stochastic versions of the equilibrium selection problem have only been studied so far in \cite{jalilzadeh2024stochastic}. They develop a stochastic approximation algorithm based on the extragradient method, and show an $O(1/\eps^{4})$ iteration complexity at both levels. This rate is inferior to our scheme, which can achieve $O(1/\eps^{3})$ iteration complexity. 
If $\opF_{2}=\nabla f_{2}$, for a convex smooth and real-valued function $f_{2}$, we recover the simple bilevel optimization problem \citep{Solodov:2007aa,Dempe:2021aa}. As far as we know, the only analysis of stochastic versions of the simple bilevel optimization problem is reported in \cite{cao2023projection}, where a stochastic Frank-Wolfe method is shown to display an $O(1/\eps)$ iteration complexity. 
\end{example}

\begin{example}[Hierarchical games]
Hierarchical games are a useful model template for computational approaches to competitive mechanism design and certain dynamic games. The key goal of mechanism design is to implement a desired Nash equilibrium in a population of agents. Hence, the designers need to anticipate the equilibrium conditions of the lower level problem. For concreteness, we consider the situation in which two upper level players manage the strategies of $N$ lower-level players. The lower level players engage in a Nash game, characterized by local minimisation problems 
$$
 \min_{y^{\nu}\in\scrY^{\nu}}\{h^{\ell}_{\nu}(y^{\nu},y^{-\nu})+\varphi^{\ell}_{\nu}(y^{\nu})\}.
$$
Under standard differentiability and monotonicity assumptions on the data of the game, we can characterize equilibria of this lower level game in terms the variational inequality $\scrS_{2}=\zer(\opF_{2}+\partial g_{2})$, where $\opF_{2}(y)=(\nabla_{y^{1}}h^{\ell}_{1}(y),\ldots,\nabla_{y^{N}}h^{\ell}_{N}(y))$. The problem of the upper level players is to select a Nash equilibrium in the lower level game by itself being engaged in a strategic optimisation problem of the form 
$$
\min_{x^{\mu}\in\scrX^{mu}}\{h^{u}_{\mu}(x^{\mu},x^{-\mu})+\varphi^{u}_{\mu}(x^{\mu})\}\quad \text{s.t.: }(x^{\mu},x^{-\mu})\in\scrS_{2},\mu\in\{1,2\}. 
$$
The strategy $x^{\mu}=(y^{\nu})_{\nu\in\scrN_{\mu}}$ consists of the actions of the adjunct lower level players. Writing the equilibrium conditions for the resulting hierarchical games leads to the problem formulation \eqref{eq:P}. 
\end{example}


\subsection{Related works}

\paragraph{Variational inequalities}
VIs provide a versatile mathematical model for numerically approaching important classes of equilibrium problems in game theory, control, and learning theory. A very popular choice for solving VIs has been the extragradient method (EG) \citep{korpelevich76}. For single-level VIs with Lipschitz operator, the standard convergence results for these algorithms display complexity $O(\eps^{-1})$ for monotone problems, and linear rates of convergence for strongly monotone problems. Complexity statements are usually made in terms of ergodic averages and well-defined gap functions. Generalizations for stochastic problems with infinite expectations were proposed in \cite{juditskyNemirovskiTauvel11,kotsalisLanLi20}. The complexity of such problems contains an extra term $O(\eps^{-2})$ in the monotone setting and $O(\eps^{-1})$ in the strongly monotone setting due to the variance of stochastic realizations of the operator, and this can not be improved without changing the problem class. For the setting of finite expectation, i.e., when the operator is given as a sum of finitely many elements, variance reduction methods allow to improve the rates back to $O(\eps^{-1})$ complexity for monotone problems and linear rate for strongly monotone settings \citep{balamurugan2016stochastic,iusemJofreOliveiraThompson17,cuiShanbhag19,alacaoglu2022stochastic,pmlr-v151-gorbunov22b}. For further references, we refer to a recent survey \cite{beznosikov2023smooth}. 




\paragraph{Hierarchical VIs}
Despite bilevel optimization and equilibrium selection being well-understood (see, e.g., recent survey \cite{beckLjubicSchmidt23}), we are aware of only a few papers on numerical methods for solving hierarchical VIs \cite{van2021regularization,ThongNumerical20,lampariello2022solution}. Among those, only \cite{lampariello2022solution} contains complexity statements. Recently, the papers \cite{samadi_improved_2025,alves2025inertial,dvurechensky2025extragradient} have addressed deterministic hierarchical VIs and proposed generalizations of EG for this setting with the rates $O(1/k^{\delta})$ (lower level) and $O(1/k^{1-\delta})$ (upper level), where $\delta \in (0,1)$ and $k$ is the iteration counter.  \cite{khalafi2025regularized} addresses the full stochastic setting where the operators are given in terms of mathematical expectations. They obtains worse complexities due to the variance in the stochastic oracle.

\subsection{Outline of results and comparison} 
To the best of our knowledge, we propose the first set of variance reduction algorithms for hierarchical VIs in the finite expectation setting. Namely, we propose a simple algorithm for the Euclidean setup and another algorithm for the general Bregman setting, both achieving the same convergence rates $O(1/k^{\delta})$ (lower level) and $O(1/k^{1-2\delta})$ (upper level), i.e., we nearly achieve the deterministic rates showing that in this challenging context of hierarchical VIs variance reduction is a powerful technique. An important ingredient of our analysis is showing that, thanks to a special geometric Attouch-Czarnecki condition \citep{Attouch:2010aa} on the lower-level solution set, the iterates of the algorithm are almost surely bounded. This allows us to drop the assumption of a compact domain of $g_1$ and/or $g_2$ used in previous works on hierarchical VIs.


%% file: Euclidian.tex
%

We derive the main structural results for resolving the hierarchical equilibrium problem, picking the extragradient method, a popular algorithmic template, with numerous applications in AI and ML, particularly in the context of GAN training \citep{chavdarova2019reducing,hsieh2019convergence}. 
\subsection{Preliminaries}  

Let $\scrX$ be a finite dimensional vector space with Euclidean inner produce $\inner{\cdot,\cdot}$ and induced norm $\norm{\cdot}$. The proximal operator is defined as 
$\prox_{g}(x)=\argmin\{g(y)+\frac{1}{2}\norm{y-x}^{2}\}$. We have
\begin{equation}\label{eq:prox}
\bar{z}=\prox_{g}(z)\iff \inner{\bar{z}-z,u-\bar{z}}\geq g(z)-g(\bar{z})\qquad\forall u\in\scrX.
\end{equation}

\begin{assumption}\label{ass:standing} 
The following hypothesis are assumed to hold throughout the paper:
\begin{itemize}
\item[(i)] The solution set $\scrS_{1}$ of \eqref{eq:P} is nonempty. In particular, $\scrS_{2}\eqdef \zer(\opF_{2}+\partial g_{2})\neq\emptyset$.
\item[(ii)] The functions $g_{i}:\scrX\to(-\infty,\infty], i=1,2$ are proper convex and lower semi-continuous.
\item[(iii)] The operators $\opF_{i},i=1,2,$ are monotone and $L_{\opF_{i}}$-Lipschtiz.
\item[(iv)] The operators $\opF_{i}$ admit a finite sum representation $\opF_{i}=\sum_{\ca\in\scrA}\opF_{i}^{\ca}$, where $\scrA$ is a finite index set, and $\opF_{i}^{\ca}:\scrX\to\scrX$ are given mappings.
\end{itemize}
\end{assumption}
We remark that Assumption \ref{ass:standing} implies that $\scrS_{2}$ is a non-empty, closed, and convex set \citep{BauCom16}.
\begin{remark} 
The assumption that both operators $\opF_{1},\opF_{2}$ can be represented as a finite sum with indices coming from a common index set $\scrA$ is without loss of generality. Indeed, if we would have the representation $\opF_{i}=\sum_{\alpha\in\scrA_{i}}\opF_{i}^{\ca}$, then define $\scrA=\scrA_{1}\cup\scrA_{2}$, and declare for example $\ca\in\scrA_{i}\setminus \scrA_{j}$ the operator $\opF_{j}^{\ca}=0$. Alternatively, we could label the mappings with a double index $\ca=(\ca_{1},\ca_{2})\in\scrA_{1}\times\scrA_{2}\equiv\scrA$. 
\end{remark}
We also assume the non-smooth part $g_{1}$ of the upper-level VI to satisfy a finite variation property, where we define the variation  over the set $\scrB\times\scrC\subseteq\scrZ\times\scrZ$ as
$\Var(g_{1}\vert \scrB\times\scrC)\eqdef \sup_{(x,y)\in\scrB\times\scrC}\abs{g_{1}(x)-g_{1}(y)}.$
\begin{assumption}
\label{ass:BVg}
The set $\dom(g_1)\cap \dom(g_2)$ is closed, and for any compact sets $\scrB,\scrC\subset\scrZ$ satisfying $\scrB,\scrC \subset\dom(g_{1})\cap\dom(g_{2})$, 
we have $\Var(g_{1}\vert \scrB\times\scrC)<\infty$.
\end{assumption}

\begin{assumption}
\label{ass:CQ}
The solution set $\scrS_{1}$ of \eqref{eq:P} satisfies
$
\scrS_{1}=\zer(\opF_{1}+\partial g_{1}+\NC_{\scrS_{2}}).
$
\end{assumption}
Similar assumptions have been made in \cite{Bot:2014aa} for solving constrained variational inequalities. It is essentially a constraint qualification assumption, and a rather mild condition. 

Define the mappings $\opV_{k}(x)\eqdef \beta_{k}\opF_{1}(x)+\opF_{2}(x)$ and $G_{k}(x)\eqdef \beta_{k}g_{1}(x)+g_{2}(x)$. The combined operator has again a finite sum structure $\opV_{k}=\sum_{\ca\in\scrA}\opV^{\ca}_{k}$. Moreover, each $\opV_{k}$ is monotone and Lipschitz. Let $(\Omega,\scrF,\Pr)$ be a probability space and $\xi:\Omega\to\Xi$ a random variable. In order to realize our stochastic algorithm, we assume to have access to a stochastic oracle $\opV_{k}^{\xi}$, returning noisy feedback information on the true mapping $\opV_{k}$. 

\begin{assumption}\label{ass:Oracle}
For all $k\in\N$, the operator $\opV_{k}$ admits a stochastic oracle $\opV^{\xi}_{k}$ such that $\Ex[\opV_{k}^{\xi}(x)]=\opV_{k}(x)$ for all $x\in\scrX$. There exists $\scrL_{k}>0$ such that $\sup_{k\geq 0}\scrL_{k}\leq L_{1}$ and 
    \begin{equation}\label{eq:LipV}
    \Ex \left[\, \norm{\opV^{\xi}_k(x) - \opV
    _k^{\xi}(y)}\right] \leq \scrL_k  \norm{x-y}\qquad \forall x,y \in \scrX.
    \end{equation}
  \end{assumption}
\begin{example}
\label{ex:sampling_prob}
Working in a finite sum setting, we can construct different stochastic oracles. Since $\opV_{k}=\sum_{\ca\in\scrA}\opV^{\ca}_{k}$, if each $\opV^{\ca}_{k}$ is $L_{k,\ca}$-Lipschitz, then the triangle inequality gives $L_{\opV_{k}}\leq\sum_{\ca\in\scrA}L_{k,\ca}$. The two simplest stochastic oracles can be defined as follows: 
\begin{enumerate}
    \item $\opV^{\ca}_{k}(x)=\abs{\scrA}\opV^{\ca}_{k}(x)$ and $\xi:\Omega\to\scrA$ is the uniformly distributed random variable with law $Q_{\ca}\eqdef\Pr(\xi=\ca)=\frac{1}{\abs{\scrA}}$. In this case $\scrL_{k}=\sqrt{\abs{\scrA}\sum_{\ca\in\scrA}L_{k,a}^{2}}$. Since $\beta_{k}\leq \beta_{1}$, we can take as upper bound on the Lipschitz modulus $L_{1}=\scrL_{1}$.
    \item Let $\xi:\Omega\to\scrA$ be the random variable with law $Q_{\ca}\eqdef\Pr(\xi=\ca)=\frac{L_{1,\ca}}{\sum_{\ca\in\scrA}L_{1,\ca}}$. Set $\opV^{\xi}_{k}(x)=\frac{1}{q_{a}}\opV^{\ca}_{k}$ on the event $\{\xi=\ca\}$. In this case, we have $\scrL_{k}=\sum_{\ca\in\scrA}L_{1,\ca}=L_{1}.$ 
\end{enumerate}
\end{example}
%


 \paragraph{Gap functions for hierarchical VI's.} Gap functions are a common tool to measure the suboptimality of a given test point when solving a variational inequality. For single-level VIs, a very common formulation of a gap function is given by 
\begin{equation}\label{eq:HFG_def}
 \Theta(x\vert \opF,g,\scrU)=\sup_{y\in\scrU}\left\{H^{(\opF,g)}(x,y)\eqdef\inner{\opF(y),x-y}+g(x)-g(y)\right\},
\end{equation}
where $\scrU$ is a compact subset of $\scrX$, used to deal with the possibility of unboundedness of $\dom(g)$ \citep{solodov2003merit,NesDual07}. Since our aim is to solve hierarchical variational systems, we have to define merit functions for the upper and lower level.

\begin{definition}
The feasibility gap and optimality gap for \eqref{eq:P} are defined as 
\begin{align}
&\Theta_{\rm Feas}(x\vert\scrC)=\sup_{y\in\scrC} H^{(\opF_{2},g_{2})}(x,y) \qquad \text{ and } \qquad \Theta_{\rm Opt}(x\vert\scrC\cap\scrS_{2})=\sup_{y\in\scrC\cap\scrS_{2}} H^{(\opF_{1},g_{1})}(x,y).\label{eq:Opt}
\end{align}
\end{definition}
To obtain two-sided bounds on gap functions, we invoke a sharpness condition of the lower-level solution set of \eqref{eq:P}.
\begin{definition}[Weak Sharpness]
\label{def:WSMain}
Let $\scrS$ be the solution set of $\HVI(\opF,g)$. We say $\scrS$ is $(\kappa,\rho)$-weak sharp with $\kappa>0$ and $\rho\in(1,2)$ if
\begin{equation}\label{eq:WSDef}
(\forall z^{\ast}\in\scrS)(\forall z\in\dom(g)):\quad H^{(\opF,g)}(z,z^{\ast})\geq\kappa\rho^{-1}\dist(z,\scrS)^{\rho}.
\end{equation}
\end{definition}
Weak sharpness is a common assumption in hierarchical optimization \citep{Cabot2005,chen2024penalty,boct2025accelerating}, and hierarchical equilibrium problems alike \citep{samadi_improved_2025}. Similar conditions have been used in an SVRG analysis of extragradient methods for single-level problems in \cite{nan2023convergence}. Combining this geometric condition, with the optimality measures $\Theta_{\rm Opt}$ and $\Theta_{\rm Feas}$, we are able to establish a-priori bounds on the gap functions involved in measuring the quality of test points $x$. The proof is given in Appendix \ref{app:VIs}.
  
\begin{lemma}\label{lem:LB}
Consider problem \eqref{eq:P}. Let Assumption \ref{ass:standing} and \ref{ass:CQ} hold. Let $\scrU_1\subseteq\dom(g_{1})$ be a nonempty compact set with $\scrU_1\cap\scrS_{1}\neq\emptyset$. Then, there exists a constant $B_{\scrU_1}>0$ such that
\begin{equation}\label{eq:LB1}
\Theta_{\rm Opt}(x\vert\scrU_1\cap\scrS_1)\geq-B_{\scrU_1}\dist(x,\scrS_{2}) \quad\forall x\in\scrX.
\end{equation}
Suppose $\scrS_{2}$ is $(\kappa,\rho)$-weakly sharp. Then for all nonempty and compact subsets $\scrU_2\subseteq\dom(g_{2})$ with $\scrU_{2}\cap\scrS_{2}\neq\emptyset$, and all $x\in\dom(g_{2})$, we have 
\begin{equation}\label{eq:WS}
\dist(x,\scrS_{2})\leq\left[\rho\kappa^{-1}\Theta_{\rm Feas}(x\vert \scrU_2\cap \scrS_2)\right]^{1/\rho}.
\end{equation}
\end{lemma}

\subsection{Hierarchical extragradient with variance reduction}

\begin{algorithm}[h]
\SetAlgoLined
\KwData{$\theta\in(0,1]$, probability distribution $Q$, step size $(\tau_{k})_{k},\alpha\in(0,1)$, regularization sequence $(\beta_{k})_{k\in\N}$, $x^{0}=w^{0}$}

\For{$k=0,1,\ldots$}{
 Sample $\xi_{k}\sim Q$.\;
Perform the updates
\begin{align*}
&z_{k}=\alpha x_{k}+(1-\alpha)w_{k}\\ 
&y_{k+1}=\prox_{\tau_{k}G_{k}}(z_{k}-\tau_{k}\opV_{k}(w_{k}))\\
&x_{k+1}=\prox_{\tau_{k}G_{k}}(z_{k}-\tau_{k}(\opV_{k}(w_{k})+\opV^{\xi_{k}}_{k}(y_{k+1})-\opV^{\xi_{k}}_{k}(w_{k}))),\\
&w_{k+1}=\left\{\begin{array}{ll} x_{k+1} & \text{ with probability }\theta \\ 
w_{k} & \text{ with probability }1-\theta.
\end{array}\right.
\end{align*}
}
\caption{Hierarchical extragradient with variance reduction}
\label{alg:Euclid}
\end{algorithm}

The main result of this paper are two complexity statements for the averaged iterates generated by Algorithm \ref{alg:Euclid} in terms of the gap functions \eqref{eq:Opt}. We work in a specific geometric setting in which we assume some structure on the solution set of the lower level problem. Associated to the bifunction $H^{(\opF_{2},g_{2})}$, Appendix \ref{app:VIs} introduces the mapping 
\begin{equation}\label{eq:phimaintext}
\varphi^{(\opF_{2},g_{2})}(x,u)\eqdef \sup_{y\in\dom(g_{2})} \{H^{(\opF_{2},g_{2})}(x,y)+\inner{y,u}\}.
\end{equation}
This function encodes dual properties of the variational inequality, since by \eqref{eq:dualphi} 
\begin{equation}\label{eq:dualphimaintext}
\varphi^{(\opF_{2},g_{2})}(x,u)\leq\sup_{y\in\dom(g_{2})}\{\inner{y,u}-H^{(\opF_{2},g_{2})}(y,x)\}=\left(H^{(\opF_{2},g_{2})}(\bullet,x)\right)^{\ast}(u). 
    \end{equation}
 The following summability condition is essentially due to \cite{Attouch:2010aa}:
\begin{assumption}\label{ass:ACcondition}[Attouch-Czarnecki condition]
The step size sequence $(\tau_{k})_{k\geq 0}$ and the regularization sequence $(\beta_{k})_{k\geq 0}$ satisfy
\begin{equation}\label{eq:AttCza}
(\forall p\in\textnormal{Range}(\NC_{\scrS_{2}})):\;\sum_{k=1}^{\infty}\tau_{k}\left[\sup_{z\in\scrS_{2}}\varphi^{(\opF_{2},g_{2})}(z,\beta_{k}p)-\supp(\beta_{k}p\vert\scrS_{2})\right]<\infty.
\end{equation}
\end{assumption}
    To understand the meaning of Assumption \ref{ass:ACcondition}, it is instructive to specialize our setting to the simple bilevel optimization case. In that situation, the data of the lower level problem are identified with $\opF_{2}=\nabla f_{2}$ and $\scrS_{2}=\argmin_{z}(f_{2}+g_{2})(z)$, and thus 
    \begin{align*}
            H^{(\opF_{2},g_{2})}(y,z)&\leq f_{2}(y)-f_{2}(z)+g_{2}(y)-g_{2}(z)\eqdef \hat{f}_{2}(y)-\hat{f}_{2}(z)\\
       &=(\hat{f}_{2}-\min\hat{f}_{2})(y)-(\hat{f}_{2}-\min\hat{f}_{2})(z).
    \end{align*}
Hence, $\varphi^{(\opF_{2},g_{2})}(z,p)\leq (\hat{f}_{2}-\min\hat{f}_{2})(z)+(\hat{f}_{2}-\min\hat{f}_{2})^{*}(p)$. Since for $z\in\scrS_{2}$ it holds $(\hat{f}_{2}-\min\hat{f}_{2})(z)=0$, this further implies 
$\sup_{z\in\scrS_{2}}\varphi^{(\opF_{2},g_{2})}(z,\beta_{k}p)-\supp(\beta_{k}p\vert\scrS_{2})\leq (\hat{f}_{2}-\min\hat{f}_{2})^{*}(\beta_{k}p)-\supp(\beta_{k}p\vert\scrS_{2}). 
$
In \cite{AttCzarPey11,Peypouquet:2012aa,boct2025accelerating} the summability condition 
$$
\sum_{k=1}^{\infty}\tau_{k}\left[(\hat{f}_{2}-\min\hat{f}_{2})^{*}(\beta_{k}p)-\supp(\beta_{k}p\vert\scrS_{2})\right]<\infty
$$
is imposed. Clearly, this condition is more restrictive than our condition \eqref{eq:AttCza}.

\begin{remark}
\label{rem:Tikhonov}
Assumption \ref{ass:ACcondition} looks quite daunting to verify, but as already observed in \cite{boct2025accelerating}, it fits very natural to the geometric setting of this paper. Indeed, let us assume that $\scrS_{2}$ is $(\kappa,\rho)$-weakly sharp, with $\kappa>0$ and $\rho\in(1,2)$. According to Definition \ref{def:WSMain}, for all $x\in\dom(g_{2})$ and for all $x^{*}\in\scrS_{2}$, we have $H^{(\opF_{2},g_{2})}(x,x^{*})\geq \frac{\kappa}{\rho}\dist(x,\scrS_{2})^{\rho}$. Hence, eq. \eqref{eq:dualphimaintext} yields for all $x\in\scrS_{2}$
\begin{align*}
\varphi^{(\opF_{2},g_{2})}(x,\beta_{k}p^{\ast})-\supp(\beta_{k}p^{\ast}\vert\scrS_{2})&\leq \left(H^{(\opF_{2},g_{2})}(\bullet,x)\right)^{\ast}(\beta_{k}p^{\ast})-\supp(\beta_{k}p^{\ast}\vert\scrS_{2})\leq \kappa^{-\frac{1}{\rho-1}}\left(\frac{\rho-1}{\rho}\right)\beta_{k}^{\frac{\rho}{\rho-1}}\norm{p^{\ast}}^{\frac{\rho}{\rho-1}}. 
\end{align*}
This shows that under the $(\kappa,\rho)$-weak sharpness condition, the summability condition \eqref{eq:AttCza} is satisfied whenever $\sum_{k\geq 1}\tau_{k}\beta^{\frac{\rho}{\rho-1}}_{k}<\infty.$ 
\close
\end{remark}

\subsection{Analysis}

To proceed with the analysis, we define the filtration $\filterH\eqdef (\scrH_{k})_{k\in\N}$ and $\filterF\eqdef(\scrF_{k})_{k\in\N}$ by 
\begin{align*}   
&\scrH_{k}\eqdef\sigma(x^{0},\xi_{0},\ldots,\xi_{k-1},w_{k}),\\
&\scrF_{k}\eqdef \sigma(x^{0},\xi_{0},\ldots,\xi_{k},w_{k}).
\end{align*}
Note that $\scrH_{k}\subset\scrF_{k}$ for all $k\in\N$. Additionally, $y_{k+1}$ is measurable with respect to $\scrH_{k}$, and $x_{k+1}$ is measurable with respect to $\scrF_{k}$.\\
To simplify the notation, we define the random field $\opA_{k+1}:\Omega\to\scrX$ by
\begin{equation}\label{eq:A}
\opA_{k+1}(\omega)\eqdef \opV_{k}(w_{k}(\omega))+\opV_{k}^{\xi_{k}}(y_{k+1}(\omega))-\opV^{\xi_{k}}_{k}(w_{k}(\omega)).
\end{equation}
By Assumption \ref{ass:Oracle}, it follows 
$
\Ex[\opA_{k+1}\vert \scrH_{k}]=\opV_{k}(y_{k+1})\;\Pr-\text{a.s.} 
$
For the iterates $(x_{k})_{k},(y_{k})_{k},(w_{k})_{k}$ of Algorithm \ref{alg:Euclid}, and any $x\in\dom(g_{1})\cap\dom(g_{2})$, we define 
\begin{align}
&\scrE_{k}(x)\eqdef\alpha\norm{x_{k}-x}^{2}+\frac{1-\alpha}{\theta}\norm{w_{k}-x}^{2}, \label{eq:E}\\ 
&\Psi_{k}(x)\eqdef \inner{\opV_{k}(y_{k+1}),y_{k+1}-x}+G_{k}(y_{k+1})-G_{k}(x) \label{eq:Psi}
\end{align}
\begin{lemma}\label{lem:energy}
Let Assumptions \ref{ass:standing}- \ref{ass:Oracle} hold. Given $\alpha\in(0,1],\theta\in (0,1]$ and $\tau_{k}\leq \frac{\sqrt{1-\alpha}}{\scrL_{k}}\gamma$, for $\gamma\in(0,1)$. For any $x\in\dom(g_{1})\cap\dom(g_{2})$, the iterates produced by Algorithm \ref{alg:Euclid} satisfy 
\begin{equation}\label{eq:Energy2}
\begin{split}
\Ex[\scrE_{k+1}(x)\vert\scrH_{k}]+2\tau_{k}\Psi_{k}(x)&\leq \scrE_{k}(x)-\alpha\norm{y_{k+1}-x_{k}}^{2}\\ 
&-(1-\gamma)\left[(1-\alpha)\norm{y_{k+1}-w_{k}}^{2}+\Ex\left(\norm{x_{k+1}-y_{k+1}}^{2}\vert\scrH_{k}\right)\right].
\end{split}
\end{equation}
\end{lemma}
Departing from the energy inequality \eqref{eq:Energy2}, we establish the almost sure boundedness of the sample paths of the stochastic process generated by the Algorithm. The proof is given in Appendix \ref{app:ProofsEuclid}. 
\begin{lemma}\label{lem:boundsequence}
Let Assumptions \ref{ass:standing}- \ref{ass:ACcondition} hold. Let $x^{*}\in\scrS_{1}$ be arbitrary. Then, under the same conditions as in Lemma \ref{lem:energy}, the following statements hold true:
\begin{itemize}
\item[(i)] $\lim_{k\to\infty}\scrE_{k}(x^{*})$  exists $\Pr$-a.s.
\item[(ii)] $\lim_{k\to\infty}\norm{y_{k+1}-w_{k}}^{2}=\lim_{k\to\infty}\norm{y_{k+1}-x_{k+1}}^{2}=0$ $\Pr$-a.s. 
\end{itemize}
In particular, the processes $(x_{k})_{k},(w_{k})_{k}$ and $(y_{k})_{k}$ are almost surely bounded, and 
\begin{equation}\label{eq:aux1}
    \sum_{k=0}^{\infty}\left((1-\alpha)\Ex\left[\norm{y_{k+1}-w_{k}}^{2}\right]+\Ex\left[\norm{x_{k+1}-y_{k+1}}^{2}\right]\right)\leq\frac{1}{1-\gamma}\left(\scrE_{0}(x^{*})+2\sum_{k=0}^{\infty}\tau_{k}h_{k}\right),
\end{equation}
where $h_{k}\eqdef \sup_{x\in\scrS_{2}}\varphi^{(\opF_{2},g_{2})}(x,\beta_{k}p^{*})-\supp(\beta_{k}p^{*}\vert\scrS_{2})$ with a certain $p^\ast \in \NC_{\scrS_{2}}(x^\ast)$.
\end{lemma}
We now come to our main theorem for the Euclidean case. It contains explicit complexity statements of expected feasibility and optimality gaps evaluated on the ergodic sequence generated by Algorithm \ref{alg:Euclid}. The proof of this theorem is rather technical and presented in Appendix \ref{app:ProofsEuclid}. 
\begin{theorem}\label{th:mainMonotone}
    Let Assumptions \ref{ass:standing}-\ref{ass:ACcondition} hold. Given $\theta\in(0,1],\alpha=1-\theta$, and $\tau_{k}\leq \frac{\sqrt{\theta}}{\scrL_{k}}\gamma$ for $\gamma\in(0,1)$. Let $\scrU_1, \scrU_2$ as in Lemma \ref{lem:LB}, let $x^\ast \in \scrS_1$, and define $p^\ast, h_k$ as in Lemma \ref{lem:boundsequence}. Further, define
$\bar{y}^{K}\eqdef \frac{\sum_{k=0}^{K-1}\tau_{k}y_{k+1}}{T_{k}}$ and $T_{K}\eqdef \sum_{k=0}^{K-1}\tau_{k}$. 
Then, there are constants $C_{\scrU_1}, C_{\scrU_2}, B_{\scrU_1} > 0$ such that we have 
\begin{equation}\label{eqthm:feas}
        \begin{aligned}
          \Ex[\Theta_{\rm Feas}(\bar{y}^{K}\vert\scrU_{2})]&\leq \frac{7}{4T_{K}}\max_{x\in\scrU_{2}}\norm{x-x_{0}}^{2}+\frac{7}{2T_{K}(1-\gamma)}\sum_{k=0}^{\infty}\tau_{k}h_{k}\\
          &+C_{\scrU_{2}}\frac{\sum_{k=0}^{K-1}\tau_{k}\beta_{k}}{T_{K}}+\frac{7}{2T_{K}(1-\gamma)}\norm{x^{*}-x_{0}}^{2}\eqdef W_{\rm Feas}(K,x_{0},,x^{*},\scrU_{2}),
\end{aligned}
\end{equation}
\begin{equation}\label{eqthm:opt}
    \begin{aligned}
    -B_{\scrU_1}\dist(\bar{y}^{K},\scrS_{2})&\leq \Ex \left[\Theta_{\rm Opt}(\bar{y}^{K}\vert\scrU_1\cap\scrS_{2})\right] \\
    &\quad 
    \leq  \frac{C_{\scrU_1}}{2T_K \beta_K}+\frac{3}{4T_{K}}\max_{x\in \scrU_1 \cap \scrS_2} \norm{x-x_0}^2+\frac{7}{2(1-\gamma)T_{K}\beta^{2}_{K}}\norm{x^{*}-x_{0}}^{2}\\
    &+\frac{7}{2(1-\gamma)T_{K}\beta^{2}_{K}}\sum_{k=0}^{\infty}\tau_{k}h_{k}\eqdef W_{\rm Opt}(K,x_{0},x^{*},\scrU_{1}).
    \end{aligned}
\end{equation}
 If the lower level solution set $\scrS_{2}$ is $(\kappa,\rho)$-weakly sharp, then we also have the lower bound 
\begin{equation}\label{eqthm:optweaklysharp}
        -B_{\scrU_{1}}\left[\frac{\rho}{\kappa}W_{\rm Feas}(K,x_{0},\scrU_{2})\right]^{1/\rho}\leq \Ex \left[\Theta_{\rm Opt}(\bar{y}^{K}\vert\scrU_1\cap\scrS_{2})\right]
\end{equation}
\end{theorem}

To obtain more comprehensive rates on the feasibility and optimality gap, we next consider the specific Tikhonov sequence 
\begin{equation}\label{eq:betapower}
\beta_{k}=\frac{a}{(k+b)^{\delta}}\quad a,b>0,k\geq 0,\delta\in(0,1/2).
\end{equation}
The weak sharpness of $\scrS_{2}$ implies $h_{k}\leq \frac{\rho-1}{\rho}\alpha^{-\frac{1}{\rho-1}}\beta_{k}^{\frac{\rho}{\rho-1}}\norm{p^{*}}^{\frac{\rho}{\rho-1}}$ (cf. Remark \ref{rem:Tikhonov}). Let $\rho^{*}$ be the conjugate parameter $\frac{1}{\rho^{*}}=1-\frac{1}{\rho}$ and assume that $\delta>\frac{1}{\rho^{*}}$. Setting $C_{\rho}=\norm{p^{*}}^{1/\rho^{*}}\alpha^{-1/(\rho-1)}\frac{1}{\rho^{*}}$, and adopting the polynomial Tikhonov sequence \eqref{eq:betapower} yields  
\begin{align}
    \sum_{k=0}^{K}h_{k}&\leq C_{\rho}\sum_{k= 0}^{K}a^{\rho^{*}}(k+b)^{-\delta\rho^{*}}\leq C_{\rho}\sum_{k= 0}^{\infty}a^{\rho^{*}}(k+b)^{-\delta\rho^{*}} \leq C_{\rho}\left[a^{\rho^{*}}b^{-\rho^{*}\delta}+\frac{a^{\rho^{*}}b^{1-\rho^{*}\delta}}{\delta\rho^{*}-1}\right]\eqdef\bar{h}_{\rho}. \label{eq:barh}
\end{align}
Inserting this bound in the complexity estimates derived in Theorem \ref{th:mainMonotone} yields the following comprehensive set of complextity guarantees.
\begin{proposition}\label{prop:ratesEucliddelta}
Assume that $\scrS_{2}$ is $(\kappa,\rho)$-weakly sharp. Let Algorithm \ref{alg:Euclid} use constant step size $\tau_{k}=\bar{\tau}\leq \frac{\sqrt{\theta}\gamma}{L_{1}}$ and regularization sequence $(\beta_{k})_{k\geq 0}$ of the form \eqref{eq:betapower} with $a,b>0$ and $\delta>1-1/\rho=\frac{1}{\rho^{*}}$. Then, we have 
\begin{equation}
W_{\rm Feas}(K,x_{0},x^{*},\scrU_{2})\leq O((K+b)^{-\delta}), \quad \text{ and } \quad W_{\rm Opt}(K,x_{0},x^{*},\scrU_{1})\leq O((K+b)^{-(1-2\delta)}). 
\end{equation}
\end{proposition}


\begin{remark} 
Consider the bound on $W_{\rm Feas}$, where we substitute $\bar{\tau}= \frac{\sqrt{\theta}\gamma}{L_{1}}$. Omitting numerical constants, we obtain the bound
\begin{align*}
W_{\rm Feas}(K,x_{0},x^{*},\scrU_{2})
&\leq \frac{L_1\max_{x\in\scrU_{2}}\norm{x-x_{0}}^{2}+L_1\norm{x^{*}-x_{0}}^{2}}{\gamma(1-\gamma)\sqrt{\theta}K}+\frac{\bar{h}_{\rho}+ C_{\scrU_{2}} }{(1-\gamma)K^{\delta}},
\end{align*}
where we've used $\gamma,\theta\in (0,1)$ and hence $1\leq 1/(1-\gamma)$ and $1\leq 1/(\gamma \sqrt{\theta})$, as well as $1/K\leq 1/K^{\delta}$ for $K\geq 1$. Thus, to attain an $\varepsilon$-optimal solution in terms of the feasibility gap, it is sufficient to take $K=O\left(\max\left\{\frac{1}{\sqrt{\theta}\varepsilon},\frac{1}{\varepsilon^{1/\delta}}\right\}\right)$ iterations. Each iteration on average requires $\theta |\scrA|+2$ evaluations of $\opV^{\ca}_{k}$. Thus, to reach accuracy $\varepsilon$ on average the algorithm needs 
$O\left(\max\left\{\frac{1}{\sqrt{\theta}\varepsilon},\frac{1}{\varepsilon^{1/\delta}}\right\}\right) \cdot (\theta |\scrA|+2)$ evaluations of the stochastic oracle. Optimizing this estimate with respect to $\theta$ gives $\theta=2/|\scrA|$, and the total complexity becomes 
$O\left(\max\left\{\frac{\sqrt{|\scrA|}}{\varepsilon},\frac{1}{\varepsilon^{1/\delta}}\right\}\right)$. 
Note that existing (deterministic) algorithms would require $|\scrA|$ evaluations of $\opV^{\ca}_{k}$ per iteration, and the total complexity is $|\scrA|\max\{1/\varepsilon,1/\varepsilon^{1/\delta}\}$ evaluations of $\opV^{\ca}_{k}$. This demonstrates the potential effectiveness of variance reduction. We also note that each term in the maximum has a natural interpretation. The first term corresponds to the standard single-level setting and is the same as in \cite{alacaoglu2022stochastic}; The second term stems from the hierarchical tructure of our problem and indicates the intrinsic complexity of the bilevel problem.
\end{remark}

%% file: Bregman.tex
%

In this section we give a non-euclidean extension of  Algorithm \ref{alg:Euclid}. Our algorithmic design follows closely \cite{alacaoglu2022stochastic}, but the analysis differs significantly due to the hierarchical nature of problem \eqref{eq:P}.

Let $\scrX$ be a finite-dimensional vector space with norm $\norm{\cdot}$, with dual space $\scrX^{*}$. Its dual norm is defined as $\norm{v}_{\ast}\eqdef\sup\{\inner{v,x}\vert\; \norm{x}\leq 1\}$. The monotone operator $\opF_{i}:\scrX\to\scrX^{*}$ is assumed to be $L_{\opF_{i}}$-Lipschitz, meaning 
$
\norm{\opF_{i}(x)-\opF_{i}(x')}_{\ast}\leq L_{\opF_{i}}\norm{x-x'}$ for all $x,x'\in\scrX.
$
Let $\distance(\bullet)\in\Gamma_{0}(\scrX)$ be a distance-generating function: $\distance$ is 1-strongly convex and continuous under a norm $\norm{\cdot}$. We follow the standard convention to assume that subdifferential of $\distance$ admits a continuous selection, which means that there exists a continuous function $\distance'$ such that $\distance'(x)\in\partial \distance(x)$ for all $x\in \dom(\partial \distance)$ (see e.g. \cite{FOMSurvey}). We define the Bregman divergence with respect to $\distance$ as $D(x,z)\eqdef \distance(x)-\distance(z)-\inner{\distance'(z),x-z}$.

Algorithm~\ref{alg:MirrorProx} defines three random sequences $(x^s_{k},y^s_{k},w^s)$ via a double loop scheme characterized by epochs $s=0,1,\ldots,S-1$ and inner loop iterations $k=0,1,\ldots,K-1$. Both $S$ and $K$ are fixed integers. The outer loop is characterized by step sizes $\tau_{s}$ and regularization parameters $\beta_{s}$. They give rise to the combined operator 
$
\opV_{s}(x)\eqdef \beta_{s}\opF_{1}(x)+\opF_{2}(x) \text{ and }G_{s}(x)=\beta_{s}g_{1}(x)+g_{2}(x).
$
As in the euclidean case, we assume that have access to an efficient stochastic oracle for the operator $\opV_{s}$. 
\begin{assumption}
\label{ass:Oracle_Bregman}
For all $s=0,1,\ldots,S-1$, we have access to a stochastic oracle $\opV^{\xi}_{s}$ such that for all $x,y\in \dom(g_{1})\cap\dom(g_{2})$ there exists a distribution $Q_{x,y}$ for which $\opV_{s}(x)=\Ex_{\xi\sim Q_{x,y}}[\opV_{s}^{\xi}(x))]$ and $\Ex_{\xi\sim Q_{x,y}}[\norm{\opV^{\xi}_{s}(x)-\opV_{s}^{\xi}(y)}^{2}_{*}]\leq L^{2}_{s}\norm{x-y}^{2}$.
\end{assumption}
 
\begin{algorithm}[h]
\SetAlgoLined
\KwData{ Step size $(\tau_{s})_{s},\alpha\in(0,1)$, $K>0$, regularization sequence $(\beta_{s})_{s}$, $x_j^{-1} = x_0^0 = w^0 ,\ \forall j \in [K]$.} 

\For{$s = 0,1,\ldots$}{
\For{$k=0,1,\ldots ,K-1$}{
%
Compute
\begin{align}
&y_{k+1}^s = \argmin_{z\in\scrZ} \left\{\tau_{s}G_{s}(z)+\tau_{s}\inner{\opV_{s}(w^s),z} + \alpha D(z,x_{k}^s) + (1-\alpha)D(z,\bar{w}^s)\right\} \label{eq:mpupd1} \\
&\text{Fix distribution } Q_{y_{k+1}^s,w^s} \text{ and sample } \xi_k^s \text{ according to it.}\nonumber \\
&\text{Define the stochastic oracle }  \opA_{k+1}^s = \opV_{s}(w^s)+\opV^{\xi_{k}^s}_{s}(y_{k+1}^s)-\opV^{\xi_{k}^s}_{s}(w^s) \notag\\
&x_{k+1}^s = \argmin_{z\in\scrZ} \left\{\tau_{s}G_{s}(z)+ \tau_{s}\inner{\opA_{k+1}^s,z} + \alpha D(z,x^s_{k}) + (1-\alpha) D(z,\bar{w}^s)\right\} \label{eq:mpupd2}
\end{align}
}
Update
$w^{s+1} = \frac{1}{K} \sum_{k=1}^K x^s_k,\;\nabla \distance(\bar{w}^{s+1}) = \frac{1}{K} \sum_{k=1}^{K} \nabla \distance(x_k^s),\; x_0^{s+1} = x_K^s$
}
\caption{Hierarchical mirror prox with variance reduction}
\label{alg:MirrorProx}
\end{algorithm}

\subsection{Analysis}
Algorithm \ref{alg:MirrorProx} and its analysis are deeply inspired by \cite{alacaoglu2022stochastic}. The hierarchical nature of our equilibrium problem, however, leads to substantial technical challenges, which require a deeper Lyapunov analysis of the stochastic process produced by the Algorithm. While the technical analysis is presented in Appendix \ref{app:ProofsBreg}, the main steps and results shall be spelled out here. 

The first important step in making progress in understanding the complexity of the scheme is a refined Lyapunov bound in the spirit of Lemma \ref{lem:energy}. The main energy terms appearing in this inequality are 
$\scrE^{s}(x)\eqdef\alpha D(x,x_0^{s})+(1-\alpha) \sum_{j=1}^{K}D(x,x_j^{s-1})$ and $h(u)\eqdef \sup_{x\in\scrS_{2}}\varphi^{(\opF_{2},g_{2})}(x,u)-\supp(u\vert\scrS_{2}).$ 
$\scrE^{s}(x)$ already appeared in \cite{alacaoglu2022stochastic}, while the second one reflects the hierarchical nature of our problem. In order to deduce from the dissipativity properties proved in terms of this function, we need a slightly refined estimate on the evolution of the sequence $(\scrE^s)_{s}$, reading as follows:
\begin{lemma}
\label{lem:bounded_Bregman}
Let Assumptions \ref{ass:standing}-\ref{ass:CQ} and \ref{ass:Oracle_Bregman} hold. Consider Algorithm \ref{alg:MirrorProx}, run with $\tau_{s}\leq \frac{\sqrt{1-\alpha}}{2L_{s}}.$ Let $x^\ast \in \scrS_1$ with corresponding $p^\ast \in \NC_{\scrS_{2}}(x^\ast)$ and assume that $\sum_{s=0}^{\infty} \tau_{s}h(\beta_{s}p^{*})<\infty$. Then 
\begin{enumerate} [label=(\roman*)]
    \item $\lim_{s\to\infty}\scrE^{s}(x^{*})$ exists $\Pr$-a.s. for all $x^{*}\in\scrS_{1}.$ 
    \item $\sum_{s=0}^{\infty}\sum_{k=0}^{K-1}\norm{y^{s}_{k+1}-w^{s}}^2<\infty$ and $\sum_{s=0}^{\infty}\sum_{k=0}^{K-1}\norm{x^{s}_{k+1}-y^{s}_{k+1}}^{2}<\infty$ $\Pr$-a.s.
    \item \label{eq:sum_iterates_Bregman}
    $\displaystyle
    \sum_{s=0}^{S-1}\sum_{k=0}^{K-1}
    \Big((1-\alpha)\Ex[\norm{y^{s}_{k+1}-w^{s}}^{2}]
    + \Ex[\norm{x^{s}_{k+1}-y^{s}_{k+1}}^{2}]\Big)
    \leq
    4\scrE^{0}(x^{*})+4K\sum_{s=0}^{S-1}\tau_{s}h(\beta_{s}p^{*}).$
\end{enumerate}
\end{lemma}
Our main result on the finite time iteration complexity of Algorithm \ref{alg:MirrorProx} is the following.

\begin{theorem} Let conditions of Lemma \ref{lem:bounded_Bregman} hold, $\scrU_1, \scrU_2$ be subsets, chosen as in Lemma \ref{lem:LB}. Fix $x^{*}\in\scrS_{1}$ and corresponding $p^{*}\in\NC_{\scrS_{2}}(x^{*})$. Set 
$T_S \eqdef \sum_{s=0}^{S-1} \tau_s, \, \hat{x}^S \eqdef \frac{1}{K T_S} \sum_{s=0}^{S-1} \tau_s \sum_{k=0}^{K-1} y_{k+1}^s$, 
$R^2(x^{*},x_0^0)\eqdef\sup_{x \in \scrU_2}\scrE^{0}(x)+\sup_{x \in \scrU_2}D(x,x_0^0)+\scrE^{0}(x^{*})$.
Then, there are constants $C_{\scrU_1}, C_{\scrU_2}, B_{\scrU_1} > 0$ s.t.
\begin{align}
\label{eq:Bregman_theta_feas}
     \Ex[\Theta_{\rm Feas}(\hat{x}^S\vert\scrU_{2})]&\leq \frac{R^2(x^{*},x_0^0)}{KT_{S}}+\frac{\sum_{s=0}^{S-1}\tau_sh(\beta_{s}p^{*})}{2T_{S}}+\frac{\sum_{s=0}^{S-1}\beta_{s}\tau_{s}}{T_{S}}C_{\scrU_{2}}\eqdef W^B_{\rm Feas}(K,S,x_{0}^0,x^{*},\scrU_{2})
\\
-B_{\scrU_1}\dist(\hat{x}^S,\scrS_{2})& \leq \Ex[\Theta_{\rm Opt}(\hat{x}^S\vert\scrU_{1}\cap \scrS_2)] \leq\frac{C_{\scrU_1}}{\beta_SK T_S} + \frac{\sup_{x \in \scrU_{1}\cap \scrS_2} D(x,x_0^0)}{K T_S}+  \frac{ \scrE^{0}(x^{*})}{2 K\beta_S^2 T_S} \nonumber \\ 
     & \hspace{10em}+\frac{\sum_{s=0}^{S-1}\tau_{s}h(\beta_{s}p^{*})}{2 \beta_S^2 T_S} \eqdef W^B_{\rm Opt}(K,S,x_{0}^0,x^{*},\scrU_{1}) \label{eq:Bregman_theta_opt}. 
\end{align}
 If the lower level solution set $\scrS_{2}$ is $(\kappa,\rho)$-weakly sharp, then we also have the lower bound 
\begin{equation}\label{eqthm:optweaklysharp}
        -B_{\scrU_{1}}\left[\rho\kappa^{-1}W^B_{\rm Feas}(K,S,x_{0}^0,x^{*},\scrU_{2})\right]^{1/\rho}\leq \Ex \left[\Theta_{\rm Opt}(\hat{x}^S\vert\scrU_1\cap\scrS_{2})\right].
\end{equation}
\end{theorem}
To obtain more comprehensive rates on the feasibility and optimality gap, we consider the specific sequence $\tau_s=\bar{\tau}\leq \frac{\sqrt{1-\alpha}}{2L_{0}}, \, \beta_s=\frac{1}{(K(s+1))^\delta}$ with $\delta \in (0,1/2)$. Similarly to Theorem \ref{th:mainMonotone}, by choosing $\delta \rho^{*}>1$, we have 
$\sum_{s=0}^{S-1}h(\beta_{s}p^{*}) \leq \tilde{h}_{\rho}/K$ and $\sum_{s=0}^{S-1}\beta_s \leq \frac{K^{-\delta}}{1-\delta} \left(S^{1-\delta} + \delta\right)$. Thus,
\begin{align}
&\begin{aligned}
     \Ex[\Theta_{\rm Feas}(\hat{x}^S\vert\scrU_{2})] 
      \leq \frac{\sup_{x \in \scrU_2}\scrE^{0}(x)+\sup_{x \in \scrU_2}D(x,x_0^0)+\scrE^{0}(x^{*})}{\bar{\tau}KS} + \frac{4C_{\scrU_2}}{(KS)^{\delta}} +\frac{\tilde{h}_{\rho}}{2KS},
\end{aligned}\\
&\begin{aligned}
    \Ex[\Theta_{\rm Opt}(\hat{x}^S\vert\scrU_{1}\cap \scrS_2)] 
      \leq\frac{C_{\scrU_1}}{\bar{\tau}(KS)^{1-\delta}}     
      + \frac{\sup_{x \in \scrU_{1}\cap \scrS_2} D(x,x_0)}{\bar{\tau}KS}+  \frac{\scrE^{0}(x^{*})}{2\bar{\tau}  (KS)^{1-2\delta}}     
      +\frac{\tilde{h}_{\rho} }{2 (KS)^{1-2\delta}}.
\end{aligned}
\end{align}

%% file: Numerics.tex
%

In this section, we briefly present two numerical examples to illustrate our theoretical contribution, with further details given in Appendix~\ref{app:Num}.
In both examples, we consider a matrix game 
\begin{equation}\label{eq:equilibrium} \max_{y\in \Delta_m}\min_{x\in \Delta_n}x^\top My,\end{equation}
where for any $l\in\Z$, $\Delta_l$ is the unit simplex in dimension $l$.

\paragraph{Equilibrium Selection.}\label{sec:Numerics_selection}

The equilibrium selection problem seeks an equilibrium of \eqref{eq:equilibrium} that minimizes a certain objective $f(z)$ with $z=(x,y)$. Our choices of matrix $M$ and $f$, as well as the full details of the implementation and setup for this example, are given in Appendix~\ref{app:Num_Eq_Selection}. Figure~\ref{fig:eq_select_comparison}  summarizes the algorithms' performance in terms of the feasibility gap and distance from the optimal solution in the upper level. The solid lines indicate the performance of the ergodic average, while the dashed lines indicate the performance of the ``last iterate'' $w$. We note that the convergence of the last iterate is an interesting open problem; see \cite{azizian2021last} for a single-level analysis.
Observe that under this choice of parameters, the performances of Algorithm~\ref{alg:Euclid} and Algorithm~\ref{alg:MirrorProx} with $\ell_2$-norm are similar, with a slight advantage to Algorithm~\ref{alg:MirrorProx}, both outperforming the deterministic EG and Algorithm~\ref{alg:MirrorProx} with $\ell_1$-norm. For Algorithm~\ref{alg:MirrorProx} with $\ell_1$-norm, the ergodic average shows inferior performance to that of the EG. Surprisingly, all stochastic methods exhibit almost linear convergence of ``last iterate'' $w$. These results correspond with the results reported by \cite[Appendix E]{malitsky2020golden} for the Bregman case.
\begin{figure}
    \includegraphics[scale=0.45]{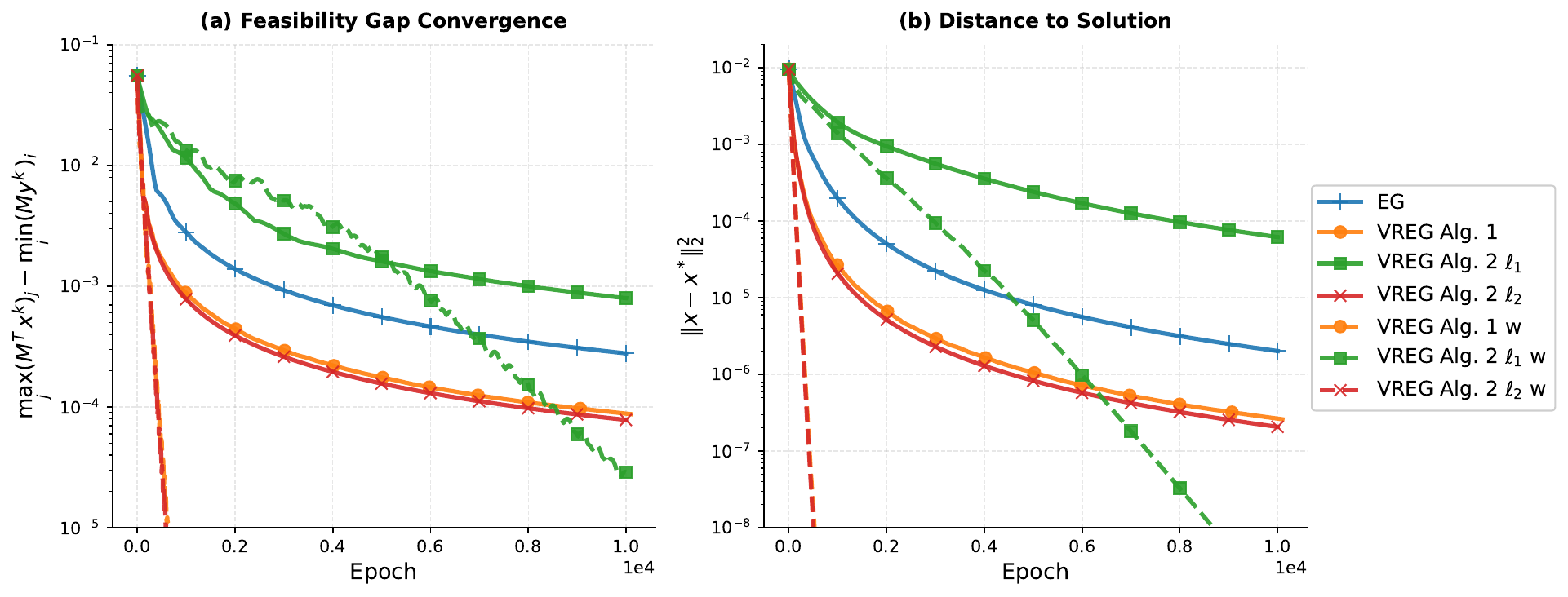}
    \caption{Performance in the Equilibrium Selection problem}
        \label{fig:eq_select_comparison}
\end{figure}

\paragraph{Linearly Constrained Equilibrium.}\label{sec:Numerics_constrained}

In this setting, we seek an equilibrium of \eqref{eq:equilibrium} that also satisfies 
a set of linear equalities given by $ Bx+Cy=d,$
 for some $B\in \R^{l\times n}$, $C\in\R^{l\times m}$ and $d\in \R^l$, with the linear constraint equivalently formulated as a minimization problem on the lower level. Our choices of the parameters, as well as the full details of the implementation and setup for this example, are given in Appendix~\ref{app:Num_Eq_lin_const}. Note that Algorithm~\ref{alg:Euclid} is not applicable since it requires, in this case, iteration-dependent randomization probabilities (cf. Example \ref{ex:sampling_prob}, item 2).

Figure~\ref{fig:eq_lin_const_comparison} compares the performance of EG and Algorithm~\ref{alg:MirrorProx} with $\ell_2$-norm in terms of optimality and feasibility gaps of their ergodic averages for $\delta \in \{0.01,0.4\}$. Algorithm~\ref{alg:MirrorProx} with $\ell_1$-norm is omitted from the plots as it underperformed relative to EG. Notably, Algorithm~\ref{alg:MirrorProx} outperforms EG for both parameter values, and decreasing $\delta$ from $0.4$ to $0.01$ yields a significant improvement in optimality while incurring only a negligible cost in feasibility, highlighting the sensitivity of the optimality-feasibility trade-off to this parameter. Moreover, for $\delta = 0.4$, the theoretically superior optimality gap convergence rate of EG is mitigated by its high iteration complexity, resulting in performance similar to Algorithm~\ref{alg:MirrorProx}. We additionally show performance of all Algorithms for $\delta=0.01$ in terms of both ergodic-average and "last-iterate" $w$ in Figure \ref{fig:minmax_all} in Appendix \ref{app:Num_Eq_lin_const}.

\begin{figure}
\includegraphics[scale=0.45]{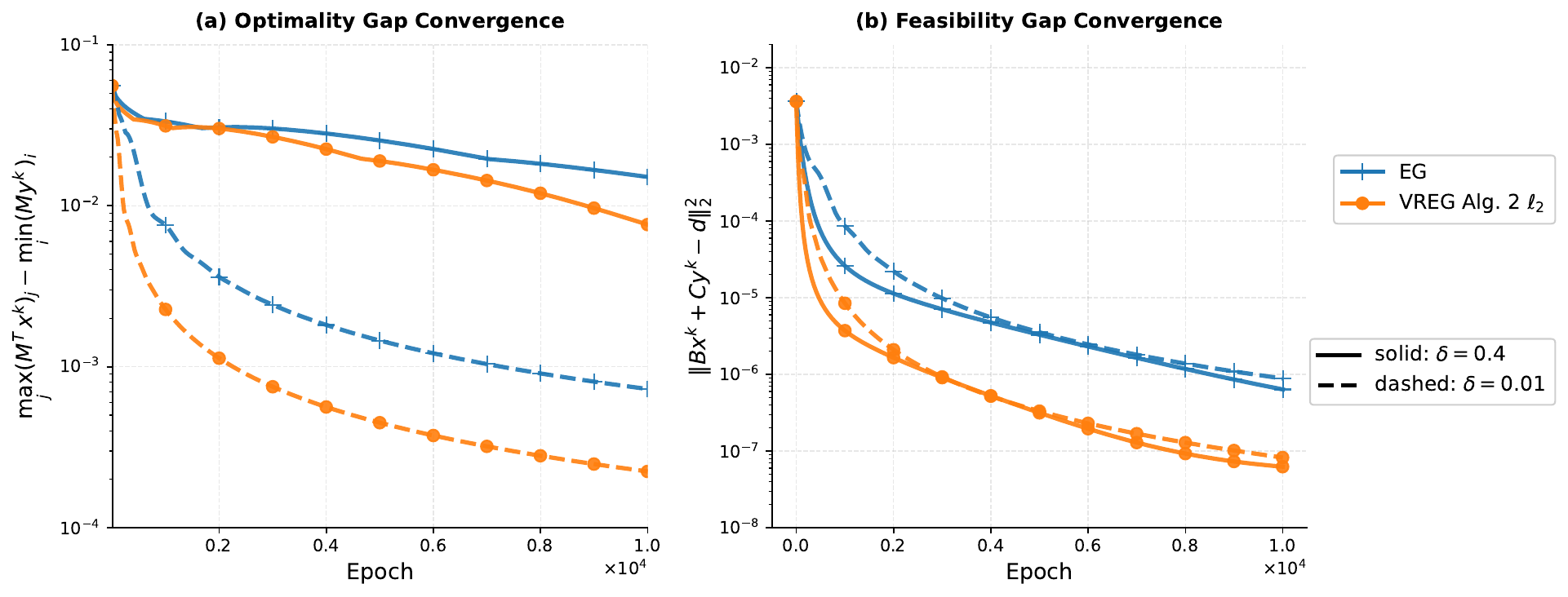}
\caption{Performance in the Linearly Constrained Equilibrium problem}
\label{fig:eq_lin_const_comparison}
\end{figure}

%% file: App_General.tex
%

\begin{lemma}[\cite{Robbins:1971aa}, Theorem 1]
\label{lem:RS}
Let $\filterF=(\scrF_{k})_{k\in\N}$ be a filtration. Let $(\scrE_{k})_{k}\in\ell_{+}(\filterF),(\gamma_{k})_{k}\in\ell^{1}_{+}(\filterF),(D_{k})_{k}\in\ell_{+}(\filterF)$ and $(\eta_{k})_{k}\in\ell^{1}_{+}(\filterF)$ be such that 
$$
\Ex(\scrE_{k+1}\vert\scrF_{k})+D_{k}\leq(1+\gamma_{k})\scrE_{k}+\eta_{k}\quad\forall k\in\N,\Pr-\text{a.s.} 
$$
Then $(D_{k})_{k\in\N}\in\ell^{1}_{+}(\filterF)$ and $(\scrE_{k})_{k\in\N}$ converges $\Pr$-a.s. to a $[0,\infty)$-valued random variable. 
\end{lemma}

We repeatedly employ the following classical Lemma, due to \cite{RobustSA}.
\begin{lemma}\label{lem:nemirovski}
Let $\filterF=(\scrF_{k})_{k\in\N}$ be a filtration and $(Z_{k})_{k\in\N}$ an $\filterF$-adapted stochastic process with values in $\scrX$ and $\Ex[Z_{k+1}\vert\scrF_{k}]=0$. Then, for any $K\geq 1,x_{0}\in\scrX$ and any compact set $\scrC\subset\scrX$, 
$$
\Ex\left[\max_{x\in\scrC}\sum_{k=0}^{K-1}\inner{Z_{k+1},x}\right]\leq\max_{x\in\scrC}\frac{1}{2}\norm{x_{0}-x}^{2}+\frac{1}{2}\sum_{k=0}^{K-1}\Ex[\norm{Z_{k+1}}^{2}]. 
$$
\end{lemma}

In order to prove the convergence rate in the Bregman case, we need the following version of Lemma \ref{lem:nemirovski}, due to \cite{alacaoglu2022stochastic}. 
\begin{lemma}
\label{lem:nemirovski_Bregman}
    Let $\filterF=(\scrF_{k}^{s})_{s\geq 0, k\in[0,K-1]}$ be a filtration and $(Z_{k}^{s})_{s\geq 0, k\in[0,K-1]}$ an $\filterF$-adapted stochastic process with values in $\scrX$ and $\Ex[Z_{k+1}^{s}\vert\scrF_{k}^{s}]=0$. Then, for any $S\geq 0,x_{0}\in\scrX$ and any compact set $\scrC\subset\scrX$, 
$$
\Ex\left[\max_{x\in\scrC}\sum_{s=0}^{S-1}\sum_{k=0}^{K-1}\inner{Z_{k+1}^{s},x}\right]\leq\max_{x\in\scrC}D(x,x_{0})+\frac{1}{2}\sum_{s=0}^{S-1}\sum_{k=0}^{K-1}\Ex[\norm{Z_{k+1}^{s}}^{2}_{*}]. 
$$
\end{lemma}
The next Lemma is a classical results for Bregman-based algorithms. 

\begin{lemma}\label{lem:3point}
   Let $\distance$ be a distance generating function, inducing the Bregman divergence $D$. We have
    \begin{equation}
       \begin{aligned}
            \inner{\nabla \distance(x) - \nabla \distance(y),z-x} = D(z,y) - D(z,x) - D(x,y)\quad \forall x,y,z \in \scrX
       \end{aligned}
    \end{equation}
\end{lemma}
Also we have for the update steps in the algorithm the following prox-like inequality:
\begin{lemma}\label{lem:bregmanproxineq}
    Let $ g$ be a proper convex and lower semi-continuous function and 
    \begin{align*}
        x^+ = \argmin_z \left\{g(z) + \inner{v,z} + \alpha D(z,x_1)+ (1-\alpha) D(z,x_2)\right\}.
    \end{align*}
    Then for any $ x \in \scrX$ the following inequality holds:  
    \begin{align*}
        g(x) - g(x^+) + \inner{v,x-x^+} \geq &\ D(x,x^+) + \alpha (D(x^+,x_1) - D(x,x_1))\\
        &\quad+(1-\alpha)(D(x^+,x_2) + (1-\alpha)D(x,x_2))
    \end{align*}
\end{lemma}
\begin{proof}
    By the optimality condition we have for $ x^+$
    \begin{align*}
        0 \in \partial g(x^+) + v + \alpha (\nabla \distance(x^+) - \nabla \distance(x_1)) + (1-\alpha) (\nabla \distance(x^+) - \nabla \distance(x_2)).
    \end{align*}
By the subgradient inequality for $ g$, we get
\begin{align*}
    g(x) - g(x^+) \geq \inner{v + \alpha(\nabla \distance(x^+) - \nabla \distance(x_1)) + (1-\alpha)(\nabla \distance(x^+) - \nabla \distance(x_2)), x^+ - x}.
\end{align*}
The result follows then from applying the three point identity Lemma \ref{lem:3point} twice to the gradient terms.
\end{proof}

%% file: App_Fitzpatrick.tex
%
\subsection{Fitzpatrick and Gap functions} 
Let $\opM:\scrZ\to 2^{\scrZ}$ be a maximally monotone operator. The \emph{Fitzpatrick function} \citep{fitzpatrick1988representing,BauCom16} $\scrF_{\opM}:\scrZ\to(-\infty,+\infty]$,  associated with the operator $\opM$, is defined as 
$$
\scrF_{\opM}(x,u)\eqdef \sup_{(y,v)\in\gr(\opM)}\{\inner{x,v}+\inner{y,u}-\inner{y,v}\}.
$$
Following \cite{Borwein:2016aa}, we define the gap function $\gap_{\opM}(x)\eqdef \scrF_{\opM}(x,0)$. $\gap_{\opM}$ is convex, and in fact the smallest translation invariant gap function associated with the monotone operator $\opM$ \cite[Theorem 3.1]{Borwein:2016aa}. Importantly, this gives the properties $\gap_{\opM}(x)\geq 0$ and $\gap_{\opM}(x)=0$ if and only if $x\in\zer(\opM)$. To make this concept concrete, observe that if $\opM=\opF+\NC_{\scrC}$, then the above definition of the gap function reduces to the well-known Auslender dual gap function \citep{FacPan03,auslender1976optimisation}
$$
\gap_{\opF+\NC_{\scrC}}(x)=\sup_{y\in\scrC}\inner{\opF(y),x-y}.
$$
If $\opM=\opF+\partial g$ for a function $g\in\Gamma_{0}(\scrZ)$, we easily obtain 
$$
\gap_{\opF+\partial g}(x)\leq\sup_{y\in\dom(g)}\inner{\opF(y),x-y}+g(x)-g(y). 
$$
To the data $(\opF,g)$, we thus associate the bifunction $H^{(\opF,g)}:\scrZ\times\scrZ\to[-\infty,\infty]$ defined as 
$$
H^{(\opF,g)}(x,y)\eqdef \inner{\opF(y),x-y}+g(x)-g(y).
$$
If $\opF:\scrZ\to\scrZ$ is monotone and continuous on $\dom(g)$, it is easy to verify that $H^{(\opF,g)}(x,y)\leq -H^{(\opF,g)}(y,x)$ for all $(x,y)\in \dom(g)\times\dom(g)$ (i.e. $H^{(\opF,g)}$ is a monotone bifunction \cite{iusem2011maximal}). In the structured setting $\HVI(\opF,g)$, we obtain the following bounds on the Fitzpatrick function, showing its close connection to the bifunction $H^{(\opF,g)}$ and the Auslender dual gap function. First, the convex subgradient inequality yields the relation
\begin{align*}
\scrF_{\opF+\partial g}(x,u)&=\sup_{y\in\dom(\partial g),\xi\in\partial g(y)}\{\inner{x-y,\opF(y)}+\inner{\xi,x-y}+\inner{y,u}\}\\
&\leq \sup_{y\in\dom(g)}\{\inner{\opF(y),x-y}+g(x)-g(y)+\inner{y,u}\}\\
&=\sup_{y\in\dom(g)} \{H^{(\opF,g)}(x,y)+\inner{y,u}\} \eqdef\varphi^{(\opF,g)}(x,u).
\end{align*}
The function $\varphi^{(\opF,g)}(x,u)$ is thus seen as the Fitzpatrick transform of $H^{(\opF,g)}$. Second, since the function $x\mapsto H^{(\opF,g)}(x,y)$ is convex, we see 
\begin{equation}\label{eq:dualphi}
\varphi^{(\opF,g)}(x,u)\leq\sup_{y\in\dom(g)}\{\inner{y,u}-H^{(\opF,g)}(y,x)\}=\left(H^{(\opF,g)}(\bullet,x)\right)^{\ast}(u). 
    \end{equation}
In particular, 
\begin{align}
&\varphi^{(\opF,g)}(x,u)\geq\inner{x,u} \qquad\forall x\in\dom(g), \label{eq:FPlower}\\ 
&\scrF_{\opF+\partial g}(x,0)\leq \sup_{y\in\dom(g)}H^{(\opF,g)}(x,y)=\varphi^{(\opF,g)}(x,0). \label{eq:FPgap}
\end{align}
Following this notation, we define for a given a subset $\scrC\subseteq\scrZ$, the restricted dual gap function as 
\begin{equation}\label{eq:gap_Def}
    \Theta(x\vert\opF,g,\scrC)\eqdef \sup_{y\in\scrC}H^{(\opF,g)}(x,y).
\end{equation}
The role of such localized gap functions is made clear in the next lemma, whose proof can be found in \cite{dvurechensky2025extragradient}.
\begin{lemma}\label{lem:gap}
    Let $\scrC\subset\dom(g)$  be a nonempty compact convex set. Consider problem $\HVI(\opF,g)$ with $\opF:\scrZ\to\scrZ$ monotone and Lipschitz continuous. The function $x\mapsto \Theta(x\vert\opF,g,\scrC)$ is well-defined and convex on $\scrZ$. For any $x\in\scrC$ we have $\Theta(x\vert \opF,g,\scrC)\geq 0.$ Moreover, if $x\in\scrC$ is a solution to $\HVI(\opF,g)$, then $\Theta(x\vert\opF,g,\scrC)=0$. Conversely, if $\Theta(x\vert\opF,g,\scrC)=0$ for some $x\in\scrC$ for which there exists an $\eps>0$ such that $\ball(x,\eps)\cap\scrC=\ball(x,\eps)\cap\dom(g)$, then $x$ is a solution of $\HVI(\opF,g)$.   
\end{lemma}

\subsection{Sharpness and error-bound property of hierarchical VI's}
Our geometric framework is phrased in terms of weak-sharpness of the lower-level solution set $\scrS_{2}$. Originally formulated for optimization problems in \cite{burke1993weak}, weak-sharpness in the context of variational inequalities has been defined in \cite{patriksson1993unified}. Subsequently, implications in terms of error bounds of primal and dual gap functions have been stated in \cite{marcotte1998weak,liu2016characterization}. The following definition of weak sharpness is from \cite{Huang:2018aa}. 

\begin{definition}
Let $\opF:\scrZ\to\scrZ$ be continuous and monotone over $\dom(g)\subset\scrZ$. Assume $\dom(g)$ is closed. The solution set $\scrS$ of $\HVI(\opF,g)$ is \emph{weakly sharp} if there exists $\tau>0$ such that 
\[
(\forall z^{\ast}\in\scrS):\quad \tau\ball(0,1)\subseteq\opF(z^{\ast})+\partial g(z^{\ast})+[\TC_{\dom(g)}(z^{\ast})\cap\NC_{\scrS}(z^{\ast})]^{\circ} 
\]
\end{definition}

Based on \cite{Huang:2018aa}, we can give the following characterization of weak sharpness in terms of an error bound involving the dual gap function of $\HVI(\opF,g)$. A proof can be found in \cite{dvurechensky2025extragradient}.
\begin{proposition}\label{prop:EB1}
Consider problem $\HVI(\opF,g)$ with $\dom(g)$ a closed, convex, and nonempty subset of $\scrZ$. If $\scrS=\zer(\opF+\partial g)$ is weakly sharp, then 
$$
(\forall z\in\dom(g)):\quad \Theta(z\vert\opF,g,\dom(g))\geq \tau\dist(z,\scrS).
$$
\end{proposition}

This Proposition shows that weak sharpness implies that the gap function satisfies an error bound property \cite{pang1997error}. Motivated by this fact, we propose the following definition:
\begin{definition}[Weak Sharpness]
\label{def:WS}
Let $\scrS$ be the nonempty solution set of $\HVI(\opF,g)$. We say $\scrS$ is $(\kappa,\rho)$-weak sharp with $\kappa>0$ and $\rho\geq 1$ if
\begin{equation}\label{eq:WSDefAppendix}
(\forall z^{\ast}\in\scrS)(\forall z\in \dom(g)):\quad H^{(\opF,g)}(z,z^{\ast})\geq\kappa\rho^{-1}\dist(z,\scrS)^{\rho}.
\end{equation}
\end{definition}
\begin{remark}
     An important class of examples arises when $\rho=1$, notably in monotone linear complementarity problems in finite dimensions under nondegeneracy conditions \cite{burke1993weak,pang1997error}. In the case where the $\HVI(\opF,g)$ reduces to a convex optimization problem, weak-sharpness implies an Hölderian errror bound, an assumption already imposed by \cite{Cabot2005} in the context of hierarchical minimization. Specifically, let us assume that $\opF=0$ and $g\in\Gamma_{0}(\scrZ)$. Then $\scrS=\argmin g$, and acordingly, $H^{(0,g)}(z,z^{*})=g(z)-\min g$ for $z^{*}\in\scrS$. Hence, \eqref{eq:WSDef} implies 
    \[
    \frac{\kappa}{\rho}\dist(u,\argmin g)^{\rho}\leq g(u)-\min g\qquad \forall u\in\scrZ.
    \]
If $\rho=1$, this is the weak-sharpness condition of \cite{burke1993weak}. The case $\rho=2$ corresponds to the "quadratic growth" condition of \cite{drusvyatskiy2018error}.\close
\end{remark}

Weak sharpness is a useful condition for obtaining a priori bounds the two gap functions introduced in the main body of the paper. We now give a proof of Lemma \ref{lem:LB} announced in the main text, which we restate for the readers' convenience together with a proof.
\begin{lemma}
Consider problem \eqref{eq:P}. Let Assumption \ref{ass:standing} and \ref{ass:CQ} hold. Let $\scrU_1\subseteq\dom(g_{1})$ be a nonempty compact set with $\scrU_1\cap\scrS_{1}\neq\emptyset$. Then, there exists a constant $B_{\scrU_1}>0$ such that
\begin{equation}\label{eq:LB1Appendix}
\Theta_{\rm Opt}(z\vert\scrU_1\cap\scrS_1)\geq-B_{\scrU_1}\dist(x,\scrS_{2}), \quad\forall x\in\scrX.
\end{equation}
Suppose $\scrS_{2}$ is $(\kappa,\rho)$-weakly sharp. Then for all nonempty and compact subsets $\scrU_2\subseteq\dom(g_{2})$ with $\scrU_{2}\cap\scrS_{2}$ and all $x\in\dom(g_{2})$, we have 
\begin{equation}\label{eq:WSAppendix}
\dist(z,\scrS_{2})\leq\left[\frac{\rho}{\kappa}\Theta_{\rm Feas}(x\vert \scrU_2\cap \scrS_2)\right]^{1/\rho}.
\end{equation}
\end{lemma}

\begin{proof}
Let $x^{*}\in\scrS_{1}=\zer(\opF_{1}+\partial g_{1}+\NC_{\scrS_{2}})\subset\scrS_{2}$. Then, there exists $p^{*}\in\NC_{\scrS_{2}}(x^{*})$ such that $-\opF_{1}(x^{*})-p^{*}\in\partial g_{1}(x^{*})$. By the convex sugradient inequality, this implies 
\begin{equation}\label{eq:b1}
\inner{\opF_{1}(z^{*}),x-x^{*}}+g_{1}(x)-g_{1}(x^{*})\geq \inner{-p^{*},x-x^{*}}\qquad\forall x\in\scrX.
\end{equation}
Take $x\in \scrX$ and let $\hat{x}=\Pi_{\scrS_{2}}(x)$ be the orthogonal projector of $x$ onto $\scrS_{2}$. Thus, $\inner{p^{*},\hat{x}-x^*}\leq 0$, resulting in
\begin{equation*}
\inner{p^{*},\hat{x}-x}\leq\inner{p^{*},x^{*}-x}, 
\end{equation*}
which implies when combined with \eqref{eq:b1} 
\begin{align*}
\inner{\opF_{1}(x^{*}),x-x^{*}}+&g_{1}(x)-g_{1}(x^{*})\geq \inner{p^{*},\hat{x}-x}\geq -\norm{p^{*}}\cdot\norm{\hat{x}-x}=-\norm{p^{*}}\dist(x,\scrS_{2}).
\end{align*}
Hence, for all compact $\scrU_1\subset\dom(g_{1})$ with $x^*\in\scrU_1\cap\scrS_{1}\neq\emptyset$, we conclude 
$\Theta_{\rm Opt}(x\vert\scrU_1\cap\scrS_1)\geq-B_{\scrU_1}\dist(x,\scrS_{2})$, where $B_{\scrU_1}=\norm{p^{*}}$. 

To show \eqref{eq:WSAppendix}, we can directly use Definition \ref{def:WS}, to conclude
\begin{align*}
\Theta_{\rm Feas}(x\vert\scrU_2\cap\scrS_2)=\sup_{x^{*}\in\scrU_2\cap\scrS_2}H^{(\opF_{2},g_{2})}(x,x^{*})\geq \frac{\kappa}{\rho}\dist(x,\scrS_{2})^{\rho}.
\end{align*}
 \end{proof}

%% file: Proofs_Euclid.tex
%

\subsection{Proof of Lemma \ref{lem:energy}}
By the prox-inequality, we have for all $x\in\dom(g_{1})\cap\dom(g_{2})$ 
\begin{align*}
\inner{y_{k+1}-z_{k},y_{k+1}-x}\leq\inner{\tau_{k}\opV_{k}(w_{k}),x-y_{k+1}}+\tau_{k}(G_{k}(x)-G_{k}(y_{k+1}),\\ 
\inner{x_{k+1}-z_{k},x_{k+1}-x}\leq\inner{\tau_{k}\opA_{k+1},x-x_{k+1}}+\tau_{k}(G_{k}(x)-G_{k}(x_{k+1}).
\end{align*}
Injecting $x=x_{k+1}$ into the first inequality, we have 
$$
\inner{y_{k+1}-z_{k},y_{k+1}-x_{k+1}}\leq\inner{\tau_{k}\opV_{k}(w_{k}),x_{k+1}-y_{k+1}}+\tau_{k}(G_{k}(x_{k+1})-G_{k}(y_{k+1}) 
$$
Using the Pythagorean three-point identity yields 
\begin{align*}
\norm{x_{k+1}-x}^{2}&=\norm{z_{k}-x}^{2}-\norm{x_{k+1}-z_{k}}^{2}+2\inner{x_{k+1}-z_{k},x_{k+1}-x}\\ 
&=\alpha\norm{x_{k}-x}^{2}+(1-\alpha)\norm{w_{k}-x}^{2}+\alpha(1-\alpha)\norm{w_{k}-x_{k}}^{2}\\
&-\norm{x_{k+1}-z_{k}}^{2}+2\inner{x_{k+1}-z_{k},x_{k+1}-x},
\end{align*}
and 
\begin{align*}
\norm{x_{k+1}-z_{k}}^{2}&=\norm{x_{k+1}-y_{k+1}}^{2}+\norm{y_{k+1}-z_{k}}^{2}+2\inner{x_{k+1}-y_{k+1},y_{k+1}-z_{k}}\\
&=\norm{x_{k+1}-y_{k+1}}^{2}+\alpha\norm{y_{k+1}-x_{k}}^{2}+(1-\alpha)\norm{y_{k+1}-w_{k}}^{2}\\
&+\alpha(1-\alpha)\norm{x_{k}-w_{k}}^{2}+2\inner{x_{k+1}-y_{k+1},y_{k+1}-z_{k}}
\end{align*}

From these expressions, we directly arrive at 
\begin{equation}\label{eq:firstenergy}
    \begin{aligned}
        \norm{x_{k+1}-x}^{2}&\leq \alpha\norm{x_{k}-x}^{2}+(1-\alpha)\norm{w_{k}-x}^{2}-\alpha\norm{y_{k+1}-x_{k}}^{2}-(1-\alpha)\norm{y_{k+1}-w_{k}}^{2}\\
&+2\tau_{k}\inner{\opV_{k}(w_{k}),x_{k+1}-y_{k+1}}+2\tau_{k}[G_{k}(x)-G_{k}(y_{k+1})]\\
&+2\tau_{k}\inner{\opA_{k+1},x-x_{k+1}}-\norm{x_{k+1}-y_{k+1}}^{2}.
    \end{aligned}
\end{equation}
By definition of the random operator $\opA_{k+1}$, we see 
\begin{align*}
\inner{\opV_{k}(w_{k}),x_{k+1}-y_{k+1}}+&\inner{\opA_{k+1},x-x_{k+1}}=\inner{\opA_{k+1},x-y_{k+1}}\\
&+\inner{\opV^{\xi_{k}}_{k}(w_{k})-\opV_{k}^{\xi_{k}}(y_{k+1}),x_{k+1}-y_{k+1}},
\end{align*}
which allows us to continue with previous thread as 
\begin{equation}\label{eq:Energy1}
\begin{split}
\norm{x_{k+1}-x}^{2}&\leq \alpha\norm{x_{k}-x}^{2}+(1-\alpha)\norm{w_{k}-x}^{2}-\alpha\norm{y_{k+1}-x_{k}}^{2}-(1-\alpha)\norm{y_{k+1}-w_{k}}^{2}\\
&+2\tau_{k}\inner{\opV^{\xi_{k}}_{k}(w_{k})-\opV^{\xi_{k}}_{k}(y_{k+1}),x_{k+1}-y_{k+1}}+2\tau_{k}[G_{k}(x)-G_{k}(y_{k+1})]\\
&+2\tau_{k}\inner{\opA_{k+1},x-y_{k+1}}-\norm{x_{k+1}-y_{k+1}}^{2}.
\end{split}
\end{equation}
By Fenchel-Young inequality, we obtain for every $\gamma>0$,
\begin{equation}\label{eq:FYV}
2\tau_{k}\inner{\opV^{\xi_{k}}_{k}(w_{k})-\opV^{\xi_{k}}_{k}(y_{k+1}),x_{k+1}-y_{k+1}}\leq\frac{\tau^{2}_{k}}{\gamma}\norm{\opV^{\xi_{k}}_{k}(w_{k})-\opV^{\xi_{k}}_{k}(y_{k+1})}^{2}+\gamma\norm{x_{k+1}-y_{k+1}}^{2}. 
\end{equation}
Applying the conditional expectations $\Ex[\cdot\vert\scrH_{k}]$ and Assumption \ref{ass:Oracle}, we obtain 
\begin{align*}
\Ex[\norm{x_{k+1}-x}^{2}\vert\scrH_{k}]&\leq  \alpha\norm{x_{k}-x}^{2}+(1-\alpha)\norm{w_{k}-x}^{2}-\alpha\norm{y_{k+1}-x_{k}}^{2}\\
&+\left(\frac{\scrL^{2}_{k}\tau^{2}_{k}}{\gamma}-(1-\alpha)\right)\norm{y_{k+1}-w_{k}}^{2}-(1-\gamma)\Ex[\norm{x_{k+1}-y_{k+1}}^{2}\vert\scrH_{k}]\\
&+2\tau_{k}\inner{\opV_{k}(y_{k+1}),x-y_{k+1}}+2\tau_{k}[G_{k}(x)-G_{k}(y_{k+1})].
\end{align*}
Next, using the definition of $w_{k+1}$ gives 
$$
\frac{1-\alpha}{\theta}\Ex[\norm{w_{k+1}-x}^{2}\vert\scrH_{k}]=(1-\alpha)\Ex[\norm{x_{k+1}-x}^{2}\vert\scrH_{k}]+(1-\alpha)(1/\theta-1)\norm{w_{k}-x}^{2}.
$$
Substituting this into the previous display, we continue
\begin{align*}
\frac{1-\alpha}{\theta}\Ex[\norm{w_{k+1}-x}^{2}\vert\scrH_{k}]&+\alpha\Ex[\norm{x_{k+1}-x}^{2}\vert\scrH_{k}]\leq \alpha\norm{x_{k}-x}^{2}+\frac{1-\alpha}{\theta}\norm{w_{k}-x}^{2}-\alpha\norm{y_{k+1}-x_{k}}^{2}\\ 
&+\left(\frac{\scrL^{2}_{k}\tau^{2}_{k}}{\gamma}-(1-\alpha)\right)\norm{y_{k+1}-w_{k}}^{2}-(1-\gamma)\Ex[\norm{x_{k+1}-y_{k+1}}^{2}\vert\scrH_{k}]\\
&+2\tau_{k}\inner{\opV_{k}(y_{k+1}),x-y_{k+1}}+2\tau_{k}[G_{k}(x)-G_{k}(y_{k+1})].
\end{align*}
Using the step size rule $\tau_{k}\leq \frac{\sqrt{1-\alpha}}{\scrL_{k}}\gamma$, we finally obtain the bound 
\begin{align*}
\frac{1-\alpha}{\theta}\Ex[\norm{w_{k+1}-x}^{2}\vert\scrH_{k}]&+\alpha\Ex[\norm{x_{k+1}-x}^{2}\vert\scrH_{k}]\leq \alpha\norm{x_{k}-x}^{2}+\frac{1-\alpha}{\theta}\norm{w_{k}-x}^{2}-\alpha\norm{y_{k+1}-x_{k}}^{2}\\ 
&-(1-\gamma)\left((1-\alpha)\norm{y_{k+1}-w_{k}}^{2}+\Ex[\norm{x_{k+1}-y_{k+1}}^{2}\vert\scrH_{k}]\right)\\
&+2\tau_{k}\inner{\opV_{k}(y_{k+1}),x-y_{k+1}}+2\tau_{k}[G_{k}(x)-G_{k}(y_{k+1})].
\end{align*}
In terms of the definitions in eqs. \eqref{eq:E} and \eqref{eq:Psi}, we finally arrive at eq. \eqref{eq:Energy2}.

\subsection{Proof of Lemma \ref{lem:boundsequence}}
 For $x^{*}\in\scrS_{1}\subset\scrS_{2}$ there exists $p^{*}\in\NC_{\scrS_{2}}(x^{*})$ with $ -\opF_1(x^\ast) - p^\ast \in  \partial g_1(x^\ast)$. Using the monotonicity of $ \opF_1$ and the convex subgradient inequality for $ g_1$, we obtain 
\begin{align*}
\inner{\opF_{1}(y_{k+1}),y_{k+1}-x^{*}}+&g_{1}(y_{k+1})-g_{1}(x^{*})+\inner{p^{*},y_{k+1}-x^{*}}\geq \\
&\inner{\opF_{1}(x^{*}),y_{k+1}-x^{*}}+g_{1}(y_{k+1})-g_{1}(x^{*})+\inner{p^{*},y_{k+1}-x^{*}}\geq 0,
\end{align*}
From \eqref{eq:Energy2}, it follows 
\begin{align*}
\Ex[\scrE_{k+1}(x^{*})\vert\scrH_{k}]+&2\tau_{k}\left(\inner{\opF_{2}(y_{k+1}),y_{k+1}-x^{*}}+g_{2}(y_{k+1})-g_{2}(x^{*})-\inner{\beta_{k}p^{*},y_{k+1}-x^{*}}\right)\\
&\leq \scrE_{k}(x^{*})-(1-\gamma)\left((1-\alpha)\norm{y_{k+1}-w_{k}}^{2}+\Ex[\norm{x_{k+1}-y_{k+1}}^{2}\vert\scrH_{k}]\right).
\end{align*}
We continue 
\begin{align*}
&-\left(\inner{\opF_{2}(y_{k+1}),y_{k+1}-x^{*}}+g_{2}(y_{k+1})-g_{2}(x^{*})-\inner{\beta_{k}p^{*},y_{k+1}-x^{*}}\right)\\ 
&=\inner{\beta_{k}p^{*},y_{k+1}}+\inner{\opF_{2}(y_{k+1}),x^{*}-y_{k+1}}+g_{2}(x^{*})-g_{2}(y_{k+1})-\inner{\beta_{k}p^{*},x^{*}}\\ 
&=\inner{\beta_{k}p^{*},y_{k+1}}+H^{(\opF_{2},g_{2})}(x^{*},y_{k+1})-\supp(\beta_{k}p^{*}\vert\scrS_{2})\\
&\leq \sup_{y\in\dom(g_{2})}\left(\inner{\beta_{k}p^{*},y}+H^{(\opF_{2},g_{2})}(x^{*},y)\right)-\supp(\beta_{k}p^{*}\vert\scrS_{2})\\
&=\varphi^{(\opF_{2},g_{2})}(x^{*},\beta_{k}p^{*})-\supp(\beta_{k}p^{*}\vert\scrS_{2})\\
&\leq\sup_{x\in\scrS_{2}}\varphi^{(\opF_{2},g_{2})}(x,\beta_{k}p^{*})-\supp(\beta_{k}p^{*}\vert\scrS_{2}).
\end{align*}
Hence, 
\begin{equation}\label{eq:energy3}
\begin{split}
\Ex[\scrE_{k+1}(x^{*})\vert\scrH_{k}]&\leq \scrE_{k}(x^{*})-(1-\gamma)\left((1-\alpha)\norm{y_{k+1}-w_{k}}^{2}+\Ex[\norm{x_{k+1}-y_{k+1}}^{2}\vert\scrH_{k}]\right)\\
&+2\tau_{k}\left(\sup_{x\in\scrS_{2}}\varphi^{(\opF_{2},g_{2})}(x,\beta_{k}p^{*})-\supp(\beta_{k}p^{*}\vert\scrS_{2})\right).
\end{split}
\end{equation}
To simplify the notation, we define
\begin{align*}
    D_{k}&\eqdef (1-\gamma)\left((1-\alpha)\norm{y_{k+1}-w_{k}}^{2}+\Ex\left[\norm{x_{k+1}-y_{k+1}}^{2}\vert\scrH_{k}\right)\right],\\
h_{k}&\eqdef \sup_{x\in\scrS_{2}}\varphi^{(\opF_{2},g_{2})}(x,\beta_{k}p^{*})-\supp(\beta_{k}p^{*}\vert\scrS_{2}),
\end{align*}
so that 
\begin{equation}\label{eq:Supermartingale}
    \Ex[\scrE_{k+1}(x^{*})\vert\scrH_{k}]\leq\scrE_{k}(x^{*})-D_{k}+2\tau_{k}h_{k}\qquad\Pr-\text{a.s.}
\end{equation}
Together with Assumption \ref{ass:ACcondition} the Robbins-Siegmund Lemma (Lemma \ref{lem:RS}) allows us to deduce that $(D_{k})_{k}\in \ell^{1}_{+}(\filterF)$. Hence, $\lim_{k\to\infty}\norm{y_{k+1}-w_{k}}=\lim_{k\to\infty}\norm{x_{k+1}-y_{k+1}}=0$.

Taking expectation in eq. \eqref{eq:Supermartingale}, and telescoping this expression, we immediately arrive at the estimate \eqref{eq:aux1}.

\subsection{Full proof of Theorem \ref{th:mainMonotone}} 
The lengthy and rather technical proof of Theorem \ref{th:mainMonotone} is organised in various steps. First, we establish the upper bound on the feasibility gap $\Theta_{\rm Feas}$. In the next step use insights obtained from that part in order to prove upper bounds on the optimality gap $\Theta_{\rm Opt}$. Since the latter gap function is a-priori signless, we need to establish an informative lower bound as well. This is the last step in the proof in order to deduce \eqref{eqthm:opt}. Exploiting then the geometric setting provided by weak sharpness of the lower level solution set, allows us to prove two-sided bounds on the optimality gap, which in the end allow us to sandwich this optimality measure. 

\subsubsection{Establishing the rates of the feasibility gap}
To obtain rates in terms of these restricted merit functions, we depart from \eqref{eq:Energy1}:
\begin{align*}
\norm{x_{k+1}-x}^{2}&\leq\alpha\norm{x_{k}-x}^{2}+(1-\alpha)\norm{w_{k}-x}^{2}-\alpha\norm{y_{k+1}-x_k}^{2}-(1-\alpha)\norm{y_{k+1}-w_{k}}^{2}\\ 
&-\norm{x_{k+1}-y_{k+1}}^{2} +2\tau_{k}(G_{k}(x)-G_{k}(y_{k+1}))+2\tau_{k}\inner{\opV_{k}(y_{k+1}),x-y_{k+1}}\\
&+2\tau_{k}\inner{\opA_{k+1}-\opV_{k}(y_{k+1}),x-y_{k+1}}+2\tau_{k}\inner{\opV^{\xi_{k}}_{k}(w_{k})-\opV^{\xi_{k}}_{k}(y_{k+1}),x_{k+1}-y_{k+1}}.
\end{align*}
Calling 
$$
M_{1}^{k}(x)\eqdef \inner{\opA_{k+1}-\opV_{k}(y_{k+1}),x-y_{k+1}},\; M_{2}^{k}\eqdef \inner{\opV_{k}^{\xi_{k}}(w_{k})-\opV_{k}^{\xi_{k}}(y_{k+1}),x_{k+1}-y_{k+1}},
$$
and set $\alpha=1-\theta$, we obtain 
\begin{align}
\norm{x_{k+1}-x}^{2}&\leq (1-\theta)\norm{x_{k}-x}^{2}+\theta\norm{w_{k}-x}^{2}-(1-\theta)\norm{y_{k+1}-x_k}^{2}-\theta\norm{y_{k+1}-w_{k}}^{2}\nonumber\\
&-\norm{x_{k+1}-y_{k+1}}^{2}-2\tau_{k}\Psi_{k}(x)+2\tau_{k}(M_{1}^{k}(x)+M_{2}^{k}).  \label{eq:important_recursion_for_SM}
\end{align}
Note that $\Ex[M_{1}^{k}(x)\vert\scrH_{k}]=0$ a.s.

With the specific choice $\alpha=1-\theta$, we obtain $\scrE_{k}(x)=(1-\theta)\norm{x_{k}-x}^{2}+\norm{w_{k}-x}^{2}$. Via some simple algebra we thus arrive at 
\begin{align*}
\scrE_{k+1}(x)+2\tau_{k}\Psi_{k}(x)&\leq \scrE_{k}(x)+\norm{w_{k+1}-x}^{2}-(1-\theta)\norm{w_{k}-x}^{2}-\theta\norm{x_{k+1}-x}^{2}\\
&-(1-\theta)\norm{y_{k+1}-x_k}^{2}+2\tau_{k}(M_{1}^{k}(x)+M_{2}^{k})-\norm{x_{k+1}-y_{k+1}}^{2}-\theta\norm{y_{k+1}-w_{k}}^{2}
\end{align*}
Calling
\begin{align}
R^{k}_{1}(x)&\eqdef \norm{w_{k+1}-x}^{2}-(1-\theta)\norm{w_{k}-x}^{2}-\theta\norm{x_{k+1}-x}^{2} \label{eq:R_1k_def}\\
&= 2\inner{\theta x_{k+1}+  (1-\theta)w_{k} - w_{k+1},x} - \theta \norm{x_{k+1}}^2 - (1-\theta) \norm{w_{k}}^2 + \norm{w_{k+1}}^2  \nonumber\\
R_{2}^{k}&\eqdef \norm{x_{k+1}-y_{k+1}}^{2}+\theta\norm{y_{k+1}-w_{k}}^{2}, \label{eq:R_2k_def}
\end{align}
the above simplifies to 
\begin{equation}\label{eq:rate1}
\scrE_{k+1}(x)+2\tau_{k}\Psi_{k}(x)\leq \scrE_{k}(x)+R_{1}^{k}(x)+2\tau_{k}(M_{1}^{k}(x)+M_{2}^{k})-R_{2}^{k}.
\end{equation}
Note that $\Ex[R_{1}^{k}(x)\vert\scrF_{k}]=0$, and the tower property gives $\Ex[R_{1}^{k}(x)\vert\scrH_{k}]=\Ex[\Ex(R_{1}^{k}(x)\vert\scrF_{k})\vert\scrH_{k}]=0$. 

We define the variation of a function $g\in\Gamma_{0}(\scrX)$ over bounded sets $\scrU_{1}\times\scrU_{2}\subset \dom(g)\times\dom(g)$, as
$$
\Var(g\vert\scrU_{1}\times\scrU_{2})=\sup_{(x,y)\in\scrU_{1}\times\scrU_{2}}\abs{g(x)-g(y)}.
$$
%
If Assumption \ref{ass:ACcondition}  holds, thanks to Lemma \ref{lem:boundsequence}, there exists a deterministic constant $C_{r}>0$ such that $x_{k},y_{k}\in\ball(x_{0},C_{r})$ for all $k\geq 1$, $\Pr$-a.s. If $\dom(g_{1})\cap\dom(g_{2})$ is compact, such a ball also exists by construction of the algorithm, since then the iterates are confined to stay in the compact set $\dom(g_{1})\cap\dom(g_{2})$ by construction. Hence, there  exists a measurable set $\Omega_{0}\subseteq\Omega$ with $\Pr(\Omega_{0})=1$ such that 
\begin{align*}
\inner{\opF_{1}(y_{k+1}(\omega)),x^{*}-y_{k+1}(\omega)}&+g_{1}(x^{*})-g_{1}(y_{k+1}(\omega))\\
&\leq \norm{\opF_{1}(y_{k+1}(\omega))}\cdot\norm{y_{k+1}(\omega)-x^{*}}+g_{1}(x^{*})-g_{1}(y_{k+1}(\omega)) \\
&\leq \norm{\opF_{1}(y_{k+1}(\omega))}\cdot\left(\norm{x^{*}}+\norm{x_{0}}+\norm{x_{0}-y_{k+1}(\omega)}\right)+g_{1}(x^{*})-g_{1}(y_{k+1}(\omega))\\
&\leq \left(\norm{\opF_{1}(x_{0})}+\norm{\opF_{1}(x_{0})-\opF_{1}(y_{k+1}(\omega))}\right)\cdot\left(\norm{y_{k+1}(\omega)-x_{0}}+\norm{x_{0}}+\norm{x^{*}}\right)\\
&+\abs{g_{1}(x^{*})-g_{1}(y_{k+1}(\omega))}\\
&\leq \left(\norm{\opF_{1}(x_{0})}+L_{\opF_{1}}\cdot\norm{x_{0}-y_{k+1}(\omega)}\right)\cdot\left(\norm{y_{k+1}(\omega)-x_{0}}+\norm{x_{0}}+\norm{x^{*}}\right)\\
&+\abs{g_{1}(x^{*})-g_{1}(y_{k+1}(\omega))}\\
&\leq \left(\norm{\opF_{1}(x_{0})}+L_{\opF_{1}}C_{r}\right)\cdot\left(C_{r}+\norm{x_{0}}+\max_{x\in\scrU_{2}}\norm{x}\right)+\scrD(\scrU_{2},C_{r}),
\end{align*}
where $\scrD(\scrU_{2},C_{r})\eqdef\Var(g_{1}\vert \scrU_{2}\times\scrB_{r})$, with $\scrB_{r}\eqdef\ball(x_{0},C_{r})\cap\dom(g_{1})\cap\dom(g_{2})$. Hence, for $x^{*}\in\scrU_{2}\subset\dom(g_{1})\cap\dom(g_{2})$ compact convex, there exists an almost surely bounded random variable $C_{\scrU_{2}}\in L^{1}_{\Pr}(\Omega;\R_{+})$ for which
\begin{equation}
\inner{\opF_{1}(y_{k+1}(\omega)),x^{*}-y_{k+1}(\omega)}+g_{1}(x^{*})-g_{1}(y_{k+1}(\omega))\leq C_{\scrU_{2}}(\omega)\quad \forall \omega\in\Omega_{0}.
\end{equation}
Using this bound, we can continue by noting that
\begin{align*}
\Psi_{k}(x,\omega)+\beta_{k}C_{\scrU_{2}}(\omega)&\geq \Psi_{k}(x,\omega)+\beta_{k}\left(\inner{\opF_{1}(y_{k+1}(\omega)),x-y_{k+1}(\omega)}+g_{1}(x)-g_{1}(y_{k+1}(\omega))\right) \\
&\geq \inner{\opF_{2}(y_{k+1}(\omega)),y_{k+1}(\omega)-x}+g_{2}(y_{k+1}(\omega))-g_{2}(x)\\
&\geq \inner{\opF_{2}(x),y_{k+1}(\omega)-x}+g_{2}(y_{k+1}(\omega))-g_{2}(x). 
\end{align*}
Therefore, eq. \eqref{eq:rate1} delivers 
\begin{align*}
2\tau_{k}\left(\inner{\opF_{2}(x),y_{k+1}(\omega)-x}+g_{2}(y_{k+1}(\omega))-g_{2}(x)\right)&\leq \scrE_{k}(x,\omega)-\scrE_{k+1}(x,\omega)+2\tau_{k}\beta_{k}C_{\scrU_{2}}(\omega)-R_{2}^{k}(\omega)\\
&+R_{1}^{k}(x,\omega)+2\tau_{k}(M_{1}^{k}(x,\omega)+M_{2}^{k}(\omega)). 
\end{align*}
Define the ergodic averages
$$
\bar{y}_{K}=\frac{\sum_{k=0}^{K-1}\tau_{k}y_{k+1}}{T_{k}},\; T_{K}\eqdef \sum_{k=0}^{K-1}\tau_{k}. 
$$
Summing from $k=0,\ldots,K-1$, we arrive at 
\begin{align*}
2T_{K}&\left(\inner{\opF_{2}(x),\bar{y}_{K}(\omega)-x}+g_{2}(\bar{y}_{K}(\omega))-g_{2}(x)\right)\leq \scrE_{0}(x,\omega)+2\sum_{k=0}^{K-1}\tau_{k}\beta_{k}C_{\scrU_{2}}(\omega)\\
&+\sum_{k=0}^{K-1}(R^{k}_{1}(x,\omega)+2\tau_{k}M_{1}^{k}(x,\omega))+2\sum_{k=0}^{K-1}\tau_{k}M_{2}^{k}(\omega))-\sum_{k=0}^{K-1}R_{2}^{k}(\omega).
\end{align*}
Hence, dividing both sides by $T_{K}$, taking the maximum over all points $x\in\scrU_{2}$ and then applying the expectation operator on both sides, we are left with the bound
\begin{align*}
\Ex[\Theta_{\rm Feas}(\bar{y}_{K}\vert\scrU_{2})]&\leq \frac{1}{2T_{K}}\max_{x\in\scrU_{2}}\scrE_{0}(x)+\Ex[C_{\scrU_{2}}]\frac{\sum_{k=0}^{K-1}\tau_{k}\beta_{k}}{T_{K}}\\
&+\frac{1}{2T_{K}}\Ex\left[\max_{x\in\scrU_{2}}\sum_{k=0}^{K-1}R_{1}^{k}(x)\right]+\frac{1}{2T_{K}}\Ex\left[\max_{x\in\scrU_{2}}\sum_{k=0}^{K-1}2\tau_{k}M_{1}^{k}(x)\right]\\
&+\frac{1}{2T_{k}}\Ex\left[\sum_{k=0}^{K-1}(2\tau_{k}M_{2}^{k}-R_{2}^{k})\right]. 
\end{align*}
We estimate the last term by using \eqref{eq:LipV} and \eqref{eq:FYV} to obtain
\begin{equation}\label{eq:M2R2bound}
    \begin{aligned}
        \Ex[2\tau_{k}M_{2}^{k}-R_{2}^{k}\vert\scrH_{k}]&\leq \frac{\tau^{2}_{k}\scrL^{2}_{k}}{\gamma}\norm{y_{k+1}-w_k}^{2}+\gamma\Ex[\norm{x_{k+1}-y_{k+1}}^{2}\vert\scrH_{k}]\\
&-\Ex[\norm{x_{k+1}-y_{k+1}}^{2}\vert\scrH_{k}]-\theta\norm{y_{k+1}-w_{k}}^{2}\\
&\leq -\theta(1-\gamma)\norm{y_{k+1}-w_{k}}^{2}-(1-\gamma)\Ex[\norm{x_{k+1}-y_{k+1}}^{2}\vert\scrH_{k}]\\
&\leq 0. 
    \end{aligned}
\end{equation}

Hence, we can drop the last term, simplifying the penultimate display to 

\begin{equation}\label{eq:boundwithR1M1}
    \begin{aligned}
        \Ex[\Theta_{\rm Feas}(\bar{y}_{K}\vert\scrU_{2})]&\leq \max_{x\in\scrU_{2}} \frac{1}{2T_{K}}\scrE_{0}(x)+\Ex[C_{\scrU_{2}}]\frac{\sum_{k=0}^{K-1}\tau_{k}\beta_{k}}{T_{K}}\\
&+\frac{1}{2T_{K}}\Ex\left[\max_{x\in\scrU_{2}}\sum_{k=0}^{K-1}R_{1}^{k}(x)\right]+\frac{1}{2T_{K}}\Ex\left[\max_{x\in\scrU_{2}}\sum_{k=0}^{K-1}2\tau_{k}M_{1}^{k}(x)\right].
    \end{aligned}
\end{equation}

For the last term in \eqref{eq:boundwithR1M1}, we use Lemma \ref{lem:nemirovski} with $Z_{k+1} = 2\tau_k(A_{k+1} - \opV_{k}(y_{k+1}))$ and filtration $\scrH_k$. Then we have 
   \begin{align}
        \Ex\left[\max_{x\in\scrU_{2}}\sum_{k=0}^{K-1}2\tau_{k}M_{1}^{k}(x)\right] & =   \Ex\left[\max_{x\in\scrU_{2}}\sum_{k=0}^{K-1} \inner{Z_{k+1},x}\right] -  \Ex\left[\sum_{k=0}^{K-1} \inner{Z_{k+1},y_{k+1}}\right] \nonumber\\
        &=\Ex\left[\max_{x\in\scrU_{2}}\sum_{k=0}^{K-1} \inner{Z_{k+1},x}\right]\nonumber\\
        &\leq\max_{x\in\scrU_{2}}\frac{1}{2}\norm{x_{0}-x}^{2}+\frac{1}{2}\sum_{k=0}^{K-1}\Ex[\norm{Z_{k+1}}^{2}]\nonumber\\
        &\leq \max_{x\in\scrU_{2}}\frac{1}{2}\norm{x_{0}-x}^{2}+2\sum_{k=0}^{K-1} \tau_k^2 \scrL_k^2 \Ex\left[\norm{y_{k+1}-w_{k}}^2\right]. \label{eq:M1bound}
   \end{align}
where we have used in the second equality that $ y_{k+1} $ is $ \scrH_k$-measurable and $ \Ex[Z_{k+1} \vert \scrH_k] = 0$. For the last inequality we have used the fact that $\Ex[\norm{X-\Ex(X)}^{2}]\leq\Ex[\norm{X}^{2}]$, implying 
\begin{align*}
         \Ex\left[\norm{Z_{k+1}}^{2}\right] & =  \tau_k^2 \Ex\left[\norm{(\opV_k^{\xi_k}(y_{k+1})-\opV^{\xi_k}_k(w_{k})) - (\opV_k(y_{k+1})+ \opV_k(w_{k}))}^{2}\right]\\
         &\leq  \tau_k^2 \Ex\left[\norm{\opV_k^{\xi_k}(y_{k+1})-\opV^{\xi_k}_k(w_{k}) }^2\right]\\
         &\leq \tau_k^2 \scrL_k^2 \Ex \left[\norm{y_{k+1}-w_{k}}\right].
   \end{align*}
For the remaining term, we proceed in a similar way, to obtain  
\begin{equation}\label{eq:R1bound}
   \begin{aligned}
        \Ex\left[\max_{x\in\scrU_{2}}\sum_{k=0}^{K-1}R_{1}^{k}(x)\right] &= \Ex\left[\max_{x\in\scrU_{2}}\sum_{k=0}^{K-1} \left[\norm{w_{k+1}-x}^{2}-(1-\theta)\norm{w_{k}-x}^{2}-\theta\norm{x_{k+1}-x}^{2}  \right]\right]\\
        &= \Ex\left[\max_{x\in\scrU_{2}}\sum_{k=0}^{K-1} \left[ 2\inner{\theta x_{k+1}+  (1-\theta)w_{k} - w_{k+1},x} - \theta \norm{x_{k+1}}^2 - (1-\theta) \norm{w_{k}}^2 + \norm{w_{k+1}}^2 \right]\right]\\
        &=2\Ex\left[\max_{x\in\scrU_{2}}\sum_{k=0}^{K-1} \inner{\theta x_{k+1}+  (1-\theta)w_{k} - w_{k+1},x}  \right]\\
        &\leq \max_{x\in\scrU_2} \norm{x-x_0}^2 + \sum_{k=0}^{K-1} \Ex\left[\norm{\theta x_{k+1}+  (1-\theta)w_{k} - w_{k+1}}^2\right]\\
        &= \max_{x\in\scrU_2} \norm{x-x_0}^2 + \theta(1-\theta)\sum_{k=0}^{K-1} \Ex\left[\norm{x_{k+1}-w_{k}}^2\right].
   \end{aligned}
\end{equation}
The first equality uses the definition of the process $R_{1}^{k}(x)$. The second equality is derived from the definition of $w_{k+1}$. The first inequality is an application of Lemma \ref{lem:nemirovski}, with filtration $\scrH_{k}$ and process $Z_{k+1}=\theta x_{k+1}+  (1-\theta)w_{k} - w_{k+1}.$ The last equality is obtained from the following direct calculation
   \begin{align*}
        \Ex\left[\norm{\theta x_{k+1}+  (1-\theta)w_{k} - w_{k+1}}^2\right] &=\Ex[\norm{\Ex(w_{k+1}\vert\scrH_{k})-w_{k+1}}^{2}]\\  
        &= \Ex \left[\theta\norm{x_{k+1}}^2 + (1-\theta)\norm{w_{k}}^2\ - \norm{\theta x_{k+1}+(1-\theta)w_{k}}^2 \right]\\
        &=\theta(1-\theta)\Ex\left[ \norm{x_{k+1} - w_{k}}^2\right].
   \end{align*}
Plugging in \eqref{eq:M1bound} and \eqref{eq:R1bound} into \eqref{eq:boundwithR1M1}, we obtain
\begin{equation}
   \begin{aligned}
      \Ex[\Theta_{\rm Feas}(\bar{y}_{K}\vert\scrU_{2})]&\leq \frac{1}{2T_{K}}\max_{x \in \scrU_2} \scrE_{0}(x)+\Ex[C_{\scrU_{2}}]\frac{\sum_{k=0}^{K-1}\tau_{k}\beta_{k}}{T_{K}}\\
      &+\frac{1}{2T_{K}}\left(\max_{x\in\scrU_2} \norm{x-x_0}^2 + \theta(1-\theta)\sum_{k=0}^{K-1} \Ex [\norm{x_{k+1}-w_{k}}^2]\right)\\
      &\quad +\frac{1}{2T_{K}}\left(\frac{1}{2}\max_{x\in\scrU_{2}}\norm{x_{0}-x}^{2}+2\sum_{k=0}^{K-1} \tau_k^2 \scrL_k^2 \Ex \left[\norm{y_{k+1}-w_{k}}^2\right]\right).
   \end{aligned}
\end{equation}
To estimate the terms under the sum, we use the step size condition $\tau_{k}\scrL_{k}\leq \sqrt{\theta}\gamma\leq\sqrt{\theta}$, in order to arrive at
\begin{align*}
           & \Ex \left[\sum_{k=0}^{K-1} \left(\theta(1-\theta)\norm{x_{k+1}-w_{k}}^2+2\tau_k^2 \scrL_k^2 \norm{y_{k+1}-w_{k}}^{2}\right)\right]\\
       \leq & \Ex \left[\sum_{k=0}^{K-1} \left(\theta(1-\theta)\norm{x_{k+1}-w_{k}}^2+ 2\theta \norm{y_{k+1}-w_{k}}^{2}\right)\right]\\
       \leq & \theta\Ex \left[\sum_{k=0}^{K-1} \left(\norm{x_{k+1}-w_{k}}^2+ 2 \norm{y_{k+1}-w_{k}}^{2}\right)\right].
  \end{align*}
In combination with \eqref{eq:aux1}, we continue with the estimate
\begin{align*}
    \theta\sum_{k=0}^{K-1}&\Ex\left[\norm{x_{k+1}-w_{k}}^{2}+2\norm{y_{k+1}-w_{k}}^{2}\right]\leq \theta \sum_{k=0}^{K-1}\Ex\left[(2+\sqrt{2})\norm{x_{k+1}-y_{k+1}}^{2}+(2+\sqrt{2})\norm{y_{k+1}-w_{k}}^{2}\right] \\
    &\leq\frac{2+\sqrt{2}}{1-\gamma}\left(\scrE_{0}(x^{*})+2\sum_{k=0}^{\infty}\tau_{k}h_{k}\right) 
\end{align*}
where the inequality in the last line uses eq. \eqref{eq:aux1}. 
Moreover, $\scrE_{0}(x)=(2-\theta)\norm{x-x_{0}}^{2}\leq2\norm{x-x_{0}}^{2}.$ Combining this with the bounds established, this yields 
\begin{align*}
          \Ex[\Theta_{\rm Feas}(\bar{y}_{K}\vert\scrU_{2})]&\leq \frac{7}{4T_{K}}\max_{x\in\scrU_{2}}\norm{x-x^{0}}^{2}+\frac{7}{2T_{K}(1-\gamma)}\sum_{k=0}^{\infty}\tau_{k}h_{k}\\
          &+\Ex[C_{\scrU_{2}}]\frac{\sum_{k=0}^{K-1}\tau_{k}\beta_{k}}{T_{K}}+\frac{7}{2T_{K}(1-\gamma)}\norm{x^{*}-x_{0}}^{2}.
          %
\end{align*}
\paragraph{Establishing the rate on the optimality gap}
Let $\scrU_{1}\subset\dom(g_{1})\cap\dom(g_{2})$ be a compact set with $\scrU_{1}\cap\scrS_{1}\neq\emptyset.$ For a point $x\in\scrS_{2}$, the monotonicity of $\opF_{1}$ and $\opF_{2}$ gives
$$
\Psi_{k}(x)\geq\beta_{k}\left(\inner{\opF_{1}(x),y_{k+1}-x}+g_{1}(y_{k+1})-g_{1}(x)\right)=\beta_{k}H^{(\opF_{1},g_{1})}(y_{k+1},x).
$$
Starting from \eqref{eq:rate1}, which we divide by $\beta_k$, and using $\Psi_{k}(x)\geq\beta_{k} H^{(\opF_{1},g_{1})}(y_{k+1},x)$, we obtain
\begin{align}
    2\tau_{k}H^{(\opF_{1},g_{1})}(y_{k+1},x)\leq \frac{1}{\beta_{k}}\scrE_{k}(x)-\frac{1}{\beta_{k}}\scrE_{k+1}(x) + \frac{1}{\beta_k}(R_{1}^{k}(x)+2\tau_{k}(M_{1}^{k}(x)+M_{2}^{k})-R_{2}^{k}).
\end{align}
We sum over $ k=0,\dots ,K-1$ and divide by $T_K$ to obtain
\begin{equation}
   \begin{aligned}
        \frac{2}{T_K} &\sum_{k=0}^{K-1} \tau_k \left(\inner{F_1(x),y_{k+1}-x} + g_1(y_{k+1}) - g_1(x)\right) \leq \frac{1}{T_K \beta_1} \scrE_1(x) + \frac{1}{T_K}(\frac{1}{\beta_2}-\frac{1}{\beta_1})\scrE_2(x)\\
        &+\dots+ \frac{1}{T_K}(\frac{1}{\beta_K}-\frac{1}{\beta_{K-1}})\scrE_K(x)+ \sum_{k=0}^{K-1} \frac{1}{T_K\beta_k}(R_{1}^{k}(x)+2\tau_{k}(M_{1}^{k}(x)+M_{2}^{k})-R_{2}^{k}).
  \end{aligned}
\end{equation}
We know that $x_k,w_k$ are bounded for $\omega \in \Omega_{0}\subseteq\Omega$. Hence on $\Omega_{0}$, the sequence $(\scrE_{k}(x))_{k\in\N}$ is uniformly bounded by some positive constant $\bar{E}(x)$. Using this bound, $ \beta_{k+1} \leq \beta_k$ and the convexity of the LHS in the $y_{k+1}$ argument, we obtain
\begin{equation}\label{eq:optgapwitherror}
    \begin{aligned}
                 \inner{F_1(x),\bar{y}^{K}-x} + g_1(\bar{y}^{K}) - g_1(x) &\leq \frac{\bar{E}(x)}{2T_K \beta_K}+ \sum_{k=0}^{K-1} \frac{1}{2T_K\beta_k}(R_{1}^{k}(x)\\
                 &+2\tau_{k}(M_{1}^{k}(x)+M_{2}^{k})-R_{2}^{k}), \quad \forall x \in \scrS_2.
    \end{aligned}
\end{equation}
Hence, taking first the supremum over $x \in \scrU_1 \cap \scrS_2$ and then expectations on both sides of eq. \eqref{eq:optgapwitherror}, we arrive at 
\begin{equation}\label{eq:optboundwitherror}
    \begin{aligned}
        &\Ex \left[\Theta_{\rm Opt}(\bar{y}^{K}\vert\scrU_1\cap\scrS_{2})\right] \leq \frac{C_{\scrU_1}}{2T_K \beta_K}+\frac{1}{2T_{K}}\Ex\left[\max_{x\in\scrU_1 \cap \scrS_2}\sum_{k=0}^{K-1}\frac{1}{\beta_k}R_{1}^{k}(x)\right]\\
       & +\frac{1}{2T_{K}}\Ex\left[\max_{x\in\scrU_1 \cap \scrS_2}\sum_{k=0}^{K-1}2\frac{\tau_k}{\beta_k}M_{1}^{k}(x)\right] +\frac{1}{2T_{K}}\Ex\left[\sum_{k=0}^{K-1}\frac{1}{\beta_k}(2\tau_{k}M_{2}^{k}-R_{2}^{k})\right],
    \end{aligned}
\end{equation}
where $C_{\scrU_1} \eqdef \sup_{x\in\scrU_{1}\cap\scrS_{2}}\bar{E}(x)$.
We have to deal with the sums on the RHS in the same manner as above. From \eqref{eq:M2R2bound} we have $\Ex[2\tau_{k}M_{2}^{k}-R_{2}^{k}\vert\scrH_{k}] \leq 0$ and we can drop the last term. Furthermore for the $M_1$-term we use lemma \ref{lem:nemirovski} with $Z_{k+1} = 2\frac{\tau_k}{\beta_k}(A_{k+1} - V(y_{k+1}))$ and filtration $\scrH_k$. Then we have 
\begin{equation}\label{eq:optM1bound}
   \begin{aligned}
        \Ex\left[\max_{x\in\scrU_1 \cap \scrS_2}\sum_{k=0}^{K-1} 2\frac{\tau_k}{\beta_k}M_{1}^{k}(x)\right] & =   \Ex\left[\max_{x\in\scrU_1 \cap \scrS_2}\sum_{k=0}^{K-1} \inner{Z_{k+1},x}\right] -  \Ex\left[\sum_{k=0}^{K-1} \inner{Z_{k+1},y_{k+1}}\right]\\
        &=\Ex\left[\max_{x\in\scrU_1 \cap \scrS_2}\sum_{k=0}^{K-1} \inner{Z_{k+1},x}\right]\\
        &\leq \max_{x\in\scrU_1 \cap \scrS_2}\frac{1}{2}\norm{x_{0}-x}^{2}+\frac{2}{ \beta_K^2}\sum_{k=0}^{K-1} \tau_k^2 \scrL_k^2 \Ex\left[\norm{y_{k+1}-w_{k}}^2\right]. 
   \end{aligned}
\end{equation}
where we have used in the second equality that $ y_{k+1} $ is $ \scrH_k$-m.b. and $ \Ex[Z_{k+1} \vert \scrH_k] = 0$. For the last inequality we have used
\begin{equation}
   \begin{aligned}
         \Ex\left[\norm{Z_{k+1}}^{2}\right] & =  4\frac{\tau^2_k}{\beta_k^2} \Ex\left[\norm{(\opV_k^{\xi_k}(y_{k+1})-\opV^{\xi_k}_k(w_{k})) - (\opV_k(y_{k+1})- \opV_k(w_{k}))}^{2}\right]\\
         &= 4\frac{\tau^2_k}{\beta_k^2}  \Ex \left[\Ex\left[\norm{(\opV_k^{\xi_k}(y_{k+1})-\opV^{\xi_k}_k(w_{k})) - (\opV_k(y_{k+1})-\opV_k(w_{k}))}^{2}\vert \scrH_k\right]\right]\\
         &= 4\frac{\tau^2_k}{\beta_k^2}  \Ex\left[ \Ex\left[\norm{\opV_k^{\xi_k}(y_{k+1})-\opV^{\xi_k}_k(w_{k}) }^2\vert \scrH_k\right]\right]\\
         &\leq 4\frac{\tau^2_k}{\beta_k^2}  \scrL_k^2 \Ex \left[\norm{y_{k+1}-w_{k}}^2\right]\\
   \end{aligned}
\end{equation}
where we use the tower-property of the conditional variance $ \Ex \norm{X - \Ex X}^2 \leq \Ex \norm{X}^2$ and Assumption \ref{ass:Oracle} . Lastly, for the remaining term we have
\begin{equation}\label{eq:optR1bound}
   \begin{aligned}
        &\Ex\left[\max_{x\in\scrU_1 \cap \scrS_2}\sum_{k=0}^{K-1}\frac{1}{\beta_k}R_{1}^{k}(x)\right] = \Ex\left[\max_{x\in\scrU_1 \cap \scrS_2}\sum_{k=0}^{K-1} \frac{1}{\beta_k}\left[\norm{w_{k+1}-x}^{2}-(1-\theta)\norm{w_{k}-x}^{2}-\theta\norm{x_{k+1}-x}^{2}  \right]\right]\\
        &= \Ex\left[\max_{x\in\scrU_1 \cap \scrS_2}\sum_{k=0}^{K-1} \frac{1}{\beta_k}\left[ 2\inner{\theta x_{k+1}+  (1-\theta)w_{k} - w_{k+1},x} - \theta \norm{x_{k+1}}^2 - (1-\theta) \norm{w_{k}}^2 + \norm{w_{k+1}}^2 \right]\right]\\
        &=2\Ex\left[\max_{x\in\scrU_1 \cap \scrS_2}\sum_{k=0}^{K-1} \frac{1}{\beta_k}\inner{\theta x_{k+1}+  (1-\theta)w_{k} - w_{k+1},x}  \right]\\
        &\leq \max_{x\in\scrU_1 \cap \scrS_2} \norm{x-x^0}^2 + \sum_{k=0}^{K-1} \frac{1}{\beta_k^2} \Ex\left[\norm{\theta x_{k+1}+  (1-\theta)w_{k} - w_{k+1}}^2\right]\\
        &= \max_{x\in\scrU_1 \cap \scrS_2} \norm{x-x^0}^2 + \theta(1-\theta)\sum_{k=0}^{K-1} \frac{1}{\beta_k^2}\Ex\left[\norm{x_{k+1}-w_{k}}^2\right]\\
        &\leq \max_{x\in\scrU_1 \cap \scrS_2} \norm{x-x^0}^2 + \frac{\theta(1-\theta)}{\beta_K^2}\sum_{k=0}^{K-1} \Ex\left[\norm{x_{k+1}-w_{k}}^2\right]
   \end{aligned}
\end{equation}
where we use $ \Ex[\Ex[(1-\theta) \norm{w_{k}}^2 + \norm{w_{k+1}}^2 \vert \scrH_k]]=0$ and 
\begin{equation}
   \begin{aligned}
        \Ex \norm{\theta x_{k+1}+  (1-\theta)w_{k} - w_{k+1}}^2 &=\Ex \left[\Ex\left[ \norm{\Ex\left[w_{k+1}\vert \scrH_k\right] - w_{k+1}}^2\vert \scrH_k\right]\right]\\
        &= \Ex \left[\Ex\left[\norm{w_{k+1}}^2\vert \scrH_k\right] - \norm{\Ex\left[w_{k+1}\right]}^2\right]\\
        &= \Ex \left[\theta\norm{x_{k+1}}^2 + (1-\theta)\norm{w_{k}}^2\ - \norm{\theta x_{k+1}+(1-\theta)w_{k}}^2 \right]\\
        &=\theta(1-\theta)\Ex \norm{x_{k+1} - w_{k}}^2,
   \end{aligned}
\end{equation}
where we use $ \Ex \norm{X - \Ex X}^2 = \Ex \norm{X}^2 - \norm{\Ex X}^2.$ Plugging in \eqref{eq:optM1bound} and \eqref{eq:optR1bound} into \eqref{eq:optboundwitherror} gives us
\begin{equation}
    \begin{aligned}
         \Ex \left[\Theta_{\rm Opt}(\bar{y}^{K}\vert\scrU_1\cap\scrS_{2})\right] &\leq \frac{C_{\scrU_{1}}}{2T_K \beta_K}+\frac{1}{2T_{K}}\left( \max_{x\in\scrU_1 \cap \scrS_2} \norm{x-x^0}^2 + \frac{\theta(1-\theta)}{\beta_K^2}\sum_{k=0}^{K-1} \Ex\left[ \norm{x_{k+1}-w_{k}}^2\right]\right)\\
         &\quad +\frac{1}{2T_{K}}\left(\max_{x\in\scrU_1 \cap \scrS_2}\frac{1}{2}\norm{x_{0}-x}^{2}+\frac{2}{ \beta_K^2}\sum_{k=0}^{K-1} \tau_k^2 \scrL_k^2 \Ex\left[\norm{y_{k+1}-w_{k}}^2\right]\right).
    \end{aligned}
\end{equation}
Finally, using $2\tau_k^2 \scrL_k^2\leq \theta$, by eq. \eqref{eq:aux1} we have
\begin{align*}
     \Ex \left[\sum_{k=1}^{K-1} \left(\theta(1-\theta)\norm{x_{k+1}-w_{k}}^2+2\tau_k^2 \scrL_k^2 \norm{y_{k+1}-w_{k}}^{2}\right)\right] \leq \frac{3.5}{1-\gamma}\left(\scrE_{0}(x^{*})+2\sum_{k=0}^{\infty}\tau_{k}h_{k}\right).
\end{align*}
Together this yields 
\begin{align*}
    \Ex \left[\Theta_{\rm Opt}(\bar{y}^{K}\vert\scrU_1\cap\scrS_{2})\right] &\leq \frac{C_{\scrU_1}}{2T_K \beta_K}+\frac{3}{4T_{K}}\max_{x\in \scrU_1 \cap \scrS_2} \norm{x-x_0}^2+\frac{7}{2(1-\gamma)T_{K}\beta^{2}_{K}}\norm{x^{*}-x_{0}}^{2}\\
    &+\frac{7}{2(1-\gamma)T_{K}\beta^{2}_{K}}\sum_{k=0}^{\infty}\tau_{k}h_{k}.
\end{align*}
We obtain the lower bound in \eqref{eqthm:opt} immediately as a consequence of Lemma \ref{lem:LB}: 
\begin{align*}
   - B_{\scrU_1}\dist(\bar{y}^{K},\scrS_{2})\stackrel{\eqref{eq:LB1}}{\leq}\Theta_{\rm Opt}(\bar{y}^{K}\vert\scrU_{1}\cap\scrS_{1}) \leq \Theta_{\rm Opt}(\bar{y}^{K}\vert\scrU_{1}\cap\scrS_{2})
\end{align*}

\subsection{Proof of Proposition \ref{prop:ratesEucliddelta}}
This is essentially a straightforward computation, given the expressions derived in Theorem \ref{th:mainMonotone} and \eqref{eq:barh}. 

We first give the detailed derivation of \eqref{eq:barh}. \begin{align*}
    \sum_{k=0}^{K}h_{k}&\leq C_{\rho}\sum_{k= 0}^{K}a^{\rho^{*}}(k+b)^{-\delta\rho^{*}}=a^{\rho^{*}}b^{-\rho^{*}\delta}+a^{\rho^{*}}\sum_{k=1}^{K}(k+b)^{-\rho^{*}\delta}\\
    &\leq C_{\rho}\left[a^{\rho^{*}}b^{-\rho^{*}\delta}+a^{\rho^{*}}\int_{0}^{K}(t+b)^{-\delta\rho^{*}}\dif t\right]\\
&=C_{\rho}\left[a^{\rho^{*}}b^{-\rho^{*}\delta}+a^{\rho^{*}}\left(\frac{1}{1-\rho^{*}\delta}(K+b)^{1-\rho^{*}\delta}-\frac{1}{1-\rho^{*}\delta}b^{1-\delta\rho^{*}}\right)\right]
\end{align*}
If $\delta>\frac{1}{\rho^{*}}=1-1/\rho$, we can pass to the limit $K\to\infty$, to obtain 
\begin{equation}
\sum_{k=0}^{\infty}h_{k}\leq C_{\rho}\left[a^{\rho^{*}}b^{-\rho^{*}\delta}+\frac{a^{\rho^{*}}b^{1-\rho^{*}\delta}}{\delta\rho^{*}-1}\right]\eqdef\bar{h}_{\rho}.
\end{equation}

Combining this inequality with Theorem \ref{th:mainMonotone}, we obtain 
\begin{align*}
W_{\rm Feas}(K,x_{0},x^{*},\scrU_{2}) &\leq \frac{7}{4\bar{\tau}K}\max_{x\in\scrU_{2}}\norm{x-x_{0}}^{2}+\frac{7}{2K(1-\gamma)}\bar{h}_{\rho}\\
& + C_{\scrU_{2}} \frac{ab^{-\delta}+(K-1+b)^{1-\delta}}{K(1-\delta)}+\frac{7}{2\bar{\tau}K(1-\gamma)}\norm{x^{*}-x_{0}}^{2}\\
&\leq O(K^{-1})+O((K+b)^{-\delta}).
\end{align*}
In the same way, we see
\begin{align*}
W_{\rm Opt}(K,x_{0},x^{*},\scrU_{1})&\leq \frac{C_{\scrU_{1}}}{2a\bar{\tau}}K^{-(1-\delta)}+\frac{3}{4\bar{\tau}a^{2}}(K+b)^{-(1-2\delta)}\max_{x\in\scrS_{1}\cap\scrU_{1}}\norm{x-x_{0}}^{2}\\
&+\frac{7}{2\bar{\tau}(1-\gamma)a^{2}}(K+b)^{-(1-2\delta)}\norm{x^{*}-x_{0}}^{2}+\frac{7\bar{h}_{\rho}}{2(1-\gamma)a^{2}}(K+b)^{-(1-2\delta)}\\
&\leq O(K^{-(1-\delta)})+O((K+b)^{-(1-2\delta)}).
\end{align*}
Collecting the leading order terms yields the result.

%% file: Proofs_Breg.tex
%

In the proofs on the trajectory generated by Algorithm \ref{alg:MirrorProx}, we have to keep track of two indices. We thus define $\scrH_{s,k}\eqdef\sigma\left(y^{0}_{1},\ldots,y^{0}_{K},\ldots,y^{s}_{1},\ldots,y^{s}_{k+1}\right)$. 
\subsection{Energy estimates of the hierarchical mirror prox algorithm with variance reduction}

We start by applying the three point identity Lemma \ref{lem:3point} to the update steps \eqref{eq:mpupd1} with $ x = x_{k+1}^s$ 

\begin{equation}\label{eq:proxyupd}
   \begin{aligned}
        \tau_s &\left(G_s(x_{k+1}^s) - G_s (y^s_{k+1}) + \inner{\opV_{s}(w^s), x_{k+1}^s- y^s_{k+1}}\right) \\
        &\geq  D(x_{k+1}^s,y_{k+1}^s) + \alpha \left(D(y_{k+1}^s,x_k^s) - D(x_{k+1}^s,x_k^s)\right)+(1-\alpha) \left(D(y_{k+1}^s, \bar{w}^s) - D(x_{k+1}^s,\bar{w}^s)\right).
   \end{aligned}
\end{equation}
Similar for the update step \eqref{eq:mpupd2} with general $ x \in \scrX$:
\begin{equation}\label{eq:proxxupd}
   \begin{aligned}
        \tau_s\left(G_s(x) - G_s(x_{k+1}^s) + \inner{\opA_{k+1}^s, x_{k+1}^s - y_{k+1}^s}\right) &\geq  D(x,x_{k+1}^s) + \alpha \left(D(x_{k+1}^s,x_k^s) - D(x,x_k^s)\right)\\
        & \quad +(1-\alpha) \left(D(x_{k+1}^s, \bar{w}^s) - D(x,\bar{w}^s)\right).
   \end{aligned}
\end{equation}
Now we sum \eqref{eq:proxxupd} and \eqref{eq:proxyupd}.
\begin{equation}\label{eq:sumproxineq}
   \begin{aligned}
        \tau_s&\left(G_s(x) - G_s(y_{k+1}^s) + \inner{\opA_{k+1}^s, x - x_{k+1}^s} +\inner{\opV_s(w^s), x_{k+1}^s- y^s_{k+1}}\right) \\
        &\geq  D(x,x_{k+1}^s) + D(x_{k+1}^s,y_{k+1}^s)+\alpha \left(D(y_{k+1}^s,x_k^s) - D(x,x_k^s)\right)\\
        &+(1-\alpha) \left(D(y_{k+1}^s, \bar{w}^s) - D(x,\bar{w}^s)\right).
   \end{aligned}
\end{equation}
By the definition of $ \nabla \distance(\bar{w}^s)$, we have
\begin{equation}\label{eq:dualgrad}
   \begin{aligned}
        D(u,\bar{w}^s) - D(v,\bar{w}^s) &= 
        \distance(u) - \distance(v) + \inner{\nabla \distance (\bar{w}^s), v-u}\\
        &= \distance(u) - \distance(v) + \inner{\frac{1}{K}\sum_{j=1}^{K} \nabla \distance(x_j^{s-1}),v-u}\\
        &= \frac{1}{K}\sum_{j=1}^{K} \left[\distance(u) - \distance(v) - \inner{ \nabla \distance(x_j^{s-1}),v-u}\right]\\
        &= \frac{1}{K}\sum_{j=1}^{K} \left[D(u, x_j^{s-1}) - D(v,x_j^{s-1})\right].
   \end{aligned}
\end{equation}
Plugging \eqref{eq:dualgrad} into \eqref{eq:sumproxineq}, and adding a zero in the inner product terms, we obtain
\begin{align}\label{eq:Bregindermetiade}
            &\tau_s\left(G_s(x) - G_s(y_{k+1}^s) + \inner{\opA_{k+1}^s, x - y_{k+1}^s} +\inner{\opV_{s}^{\xi_k^s}(y_{k+1}^s)-V_s^{\xi_k^s}(w^s), y^s_{k+1}-x_{k+1}^s}\right) \\
        & \geq  D(x,x_{k+1}^s) + D(x_{k+1}^s,y_{k+1}^s)+\alpha \left(D(y_{k+1}^s,x_k^s) - D(x,x_k^s)\right)+\frac{1-\alpha}{K}\sum_{j=1}^{K} \left[D(y_{k+1}^s, x_j^{s-1}) - D(x,x_j^{s-1})\right].
   \end{align}
We now use the strong-convexity of the Bregman divergence and Jensen's inequality to estimate
\begin{align}
        &\frac{1-\alpha}{K}\sum_{j=1}^{K} D(y_{k+1}^s, x_j^{s-1}) \geq \frac{1-\alpha}{K}\sum_{j=1}^{K} \frac{1}{2} \norm{y_{k+1}^s-x_j^{s-1}}^2 \geq \frac{1-\alpha}{2} \norm{y_{k+1}^s - w^s}^2,\label{eq:Dstronglyconvex1}\\
        &D(x_{k+1}^s,y_{k+1}^s) \geq \frac{1}{2}\norm{x_{k+1}^s - y_{k+1}^s}^2.\label{eq:Dstronglyconvex2}
   \end{align}
Define $\scrE^{0}(x)\eqdef (\alpha+(1-\alpha)K)D(x,x_{0})$, and
\begin{align}
&\scrE^{s}(x)\eqdef\alpha D(x,x_0^{s})+(1-\alpha) \sum_{j=1}^{K}D(x,x_j^{s-1}) \label{eq:E},\\ 
&\Psi_{k}^s(x)\eqdef \inner{\opV_{s}(y^s_{k+1}),y^s_{k+1}-x}+G_{s}(y^s_{k+1})-G_{s}(x). \label{eq:Psi}
\end{align}
and
\begin{align}
    &M_{1}(x,s,k)\eqdef \inner{\opA^s_{k+1}-\opV_{s}(y^s_{k+1}),x-y_{k+1}^s},\label{eq:M1}\\
    &M_{2}(s,k)\eqdef \tau_{s}\inner{\opV_s^{\xi_k^s}(y_{k+1}^s)-\opV_s^{\xi_k^s}(w^s), y^s_{k+1}-x_{k+1}^s}\notag\\
    & -\frac{1}{2}\norm{x^{s}_{k+1}-y^{s}_{k+1}}^{2}-\frac{1-\alpha}{2}\norm{y^{s}_{k+1}-w^{s}}^{2}.\label{eq:M2}
\end{align}
We have $\Ex[M_{1}(x,s,k)\vert\scrH_{k,s}]=0$, and from Fenchel-Young 
\begin{align*}
M_{2}(s,k)&\leq \tau^{2}_{s}\norm{\opV^{\xi_{k}^{s}}_{s}(w^{s})-\opV^{\xi^{s}_{k}}_{s}(y^{s}_{k+1})}_{*}^2+\frac{1}{4}\norm{y^{s}_{k+1}-x^{s}_{k+1}}^{2} \\
& -\frac{1-\alpha}{2}\norm{y^{s}_{k+1}-w^{s}}^{2}-\frac{1}{2}\norm{y^{s}_{k+1}-x^{s}_{k+1}}^{2}. 
\end{align*}
Applying Assumption \ref{ass:Oracle_Bregman}, choosing $\tau_s$ s.t. $\tau_s^{2}L^{2}_{s}\leq (1-\alpha)/4$, and taking expectations on both sides yields thus 
\begin{equation}\label{eq:boundM2}
    \Ex[M_{2}(s,k)\vert\scrH_{s,k}]\leq -\frac{1}{4}\norm{y^{s}_{k+1}-x^{s}_{k+1}}^{2}-\frac{1-\alpha}{4}\norm{y^{s}_{k+1}-w^{s}}^{2}.
    \end{equation}

We continue our estimation, by plugging in \eqref{eq:Dstronglyconvex1} and \eqref{eq:Dstronglyconvex2} into \eqref{eq:Bregindermetiade} and use the definitions \eqref{eq:Psi} and \eqref{eq:M1}, \eqref{eq:M2} to obtain
\begin{equation}\label{eq:togetherbound}
   \begin{aligned}
        &\tau_s \Psi^s_k(x) \leq \tau_{s}\Psi^{s}_{k}(x)+ \alpha D(y_{k+1}^s,x_k^s)  \\
        & \leq \tau_{s} M_1(x,s,k)+ M_2(s,k)  + \alpha D(x,x_k^s) - D(x,x_{k+1}^s) +\frac{1-\alpha}{K} \sum_{j=1}^{K}D(x,x_j^{s-1})\\
        &=\tau_{s} M_1(x,s,k)+ M_2(s,k)  + \alpha (D(x,x_k^s) - D(x,x_{k+1}^s))\\
        &-(1-\alpha)D(x,x^{s}_{k+1})- \frac{1-\alpha}{K} \sum_{j=1}^{K}D(x,x_j^{s-1})\\
   \end{aligned}
\end{equation}
Summing over $k=0,1,\ldots,K-1$, and using the definitions \eqref{eq:E}, \eqref{eq:Psi}, we obtain
\begin{equation}
\label{eq:togetherbound_1}
      \sum_{k=0}^{K-1}\tau_s \Psi^s_k(x)+\scrE^{s+1}(x)\leq\scrE^{s}(x)+\sum_{k=0}^{K-1} \tau_s M_1(x,s,k) + \sum_{k=0}^{K-1}M_2(s,k). 
\end{equation}

We have 
\begin{align}
\Psi^{s}_{k}(x)&=\beta_{s}\left(\inner{\opF_{1}(y^{s}_{k+1}),y^{s}_{k+1}-x}+g_{1}(y^{s}_{k+1})-g_{1}(x)\right)\nonumber \\
&+\inner{\opF_{2}(y^{s}_{k+1}),y_{k+1}^{s}-x}+g_{2}(y^{s}_{k+1})-g_{2}(x)\nonumber  \\
&=-\beta_{s}H^{(\opF_{1},g_{1})}(x,y^{s}_{k+1})-H^{(\opF_{2},g_{2})}(x,y^{s}_{k+1}) \label{eq:psi_H}
\end{align}
Choosing $x=x^{*}\in\scrS_{1}$, there exist for $p^{*}\in\NC_{\scrS_{2}}(x^{*})$, for which 
$$
\sum_{k=0}^{K-1}\Psi^{s}_{k}(x^{*})\geq \beta_{s}K\inner{p^{*},x^{*}-\bar{y}^{s}}-\sum_{k=0}^{K-1}H^{(\opF_{2},g_{2})}(x^{*},y^{s}_{k+1}). 
$$
Hence, 
\begin{align*}
\scrE^{s+1}(x^{*})&
\leq\scrE^{s}(x^{*})+\tau_{s}\left(\sum_{k=0}^{K-1}H^{(\opF_{2},g_{2})}(x^{*},y^{s}_{k+1})+K\inner{\beta_{s}p^{*},\bar{y}^{s}-x^{*}}\right)+\sum_{k=0}^{K-1}\tau_s M_1(x,s,k) + \sum_{k=0}^{K-1}M_2(s,k)\\
&= \scrE^{s}(x^{*})+\tau_{s}\sum_{k=0}^{K-1}\left(\inner{\beta_{s}p^{*},y^{s}_{k+1}}+H^{(\opF_{2},g_{2})}(x^{*},y^{s}_{k+1})-\supp(\beta_{s}p^{*}\vert\scrS_{2})\right)+\sum_{k=0}^{K-1}\tau_s M_1(x,s,k) + \sum_{k=0}^{K-1}M_2(s,k)\\
&\leq \scrE^{s}(x^{*})+\tau_{s}\sum_{k=0}^{K-1}\left(\varphi^{(\opF_{2},g_{2})}(x^{*},\beta_{s}p^{*})-\supp(\beta_{s}p^{*}\vert\scrS_{2})\right)+\sum_{k=0}^{K-1}\tau_s M_1(x,s,k) + \sum_{k=0}^{K-1}M_2(s,k)\\
&\leq \scrE^{s}(x^{*})+\tau_{s}Kh(\beta_{s}p^{*})+\sum_{k=0}^{K-1}\tau_s M_1(x,s,k) + \sum_{k=0}^{K-1}M_2(s,k),
\end{align*}
where the last inequality uses the definition $h(u)\eqdef \sup_{x\in\scrS_{2}}\varphi^{(\opF_{2},g_{2})}(x,u)-\supp(u\vert\scrS_{2}).$ Combining with \eqref{eq:boundM2}, we thus yield the estimate 
$$
\Ex[\scrE^{s+1}(x^{*})]\leq\Ex[\scrE^{s}(x^{*})]+\tau_{s}K h(\beta_{s}p^{*})-\frac{1-\alpha}{4}\sum_{k=0}^{K-1}\Ex[\norm{y^{s}_{k+1}-w^{s}}^{2}]-\frac{1}{4}\sum_{k=0}^{K-1}\Ex[\norm{x^{s}_{k+1}-y^{s}_{k+1}}^{2}].
$$
We readily deduce from this Lyapunov inequality the statements made in Lemma \ref{lem:bounded_Bregman}.

\subsection{Complexity analysis of the hierarchical mirror prox with variance reduction algorithm}

Lemma \ref{lem:bounded_Bregman} shows that the sample paths generated by Algorithm \ref{alg:MirrorProx} are a.s. bounded. Thus, there exist a measurable set $\Omega_{0}\subseteq\Omega$ with $\Pr(\Omega_{0})=1$ such that all the iterates $y_k^s(\omega),x_k^s(\omega),w^s(\omega)$ are bounded for all $\omega \in \Omega_0$. In what follows we restrict ourselves to this set $\Omega_0$.

Combining \eqref{eq:togetherbound_1} and \eqref{eq:psi_H}, we obtain, for any $x \in \scrZ$,
\begin{equation}
\label{eq:Bregman_rates_proof_1}
   \begin{aligned}
      -\beta_{s}& \tau_s \sum_{k=0}^{K-1} H^{(\opF_{1},g_{1})}(x,y^{s}_{k+1})-\tau_s \sum_{k=0}^{K-1}H^{(\opF_{2},g_{2})}(x,y^{s}_{k+1}) 
      \leq\scrE^{s}(x)-\scrE^{s+1}(x)\\
      &+\sum_{k=0}^{K-1} \tau_s M_1(x,s,k) + \sum_{k=0}^{K-1}M_2(s,k). 
  \end{aligned}
\end{equation}
Since the sequence $y^{s}_{k+1}$ is a.s. bounded, we have that there exists a constant $C_x$ s.t. $H^{(\opF_{1},g_{1})}(x,y^{s}_{k+1}) \leq C_x$. 
By monotonicity of $\opF_2$, we have $-H^{(\opF_{2},g_{2})}(x,y^{s}_{k+1}) \geq H^{(\opF_{2},g_{2})}(y^{s}_{k+1},x)$. 
Combining this with the above inequality, we obtain
\begin{equation}
\label{eq:Bregman_rates_proof_2}
      \tau_s \sum_{k=1}^{K-1}H^{(\opF_{2},g_{2})}(y^{s}_{k+1},x) 
      \leq\scrE^{s}(x)-\scrE^{s+1}(x)+\beta_s\tau_s K C_x+\sum_{k=0}^{K-1} \tau_s M_1(x,s,k) + \sum_{k=0}^{K-1}M_2(s,k). 
\end{equation}
Summing these inequalities, defining $\hat{x}^S=\frac{1}{K T_S}\sum_{s=0}^{S-1}\tau_s\sum_{k=0}^{K-1}y^{s}_{k+1}$, where $T_S \eqdef \sum_{s=0}^{S-1}\tau_s$, and using the convexity of $g_2$, we obtain
\begin{equation}
\label{eq:Bregman_rates_proof_3}
      H^{(\opF_{2},g_{2})}(\hat{x}^S,x) 
      \leq\frac{\scrE^{0}(x)}{K T_S}+\frac{\sum_{s=0}^{S-1}\beta_s\tau_s K C_x}{K T_S}+ \frac{1}{K T_S}\sum_{s=0}^{S-1}\left(\sum_{k=0}^{K-1} \tau_s M_1(x,s,k) + \sum_{k=0}^{K-1}M_2(s,k)\right). 
\end{equation}
Since for any compact $\scrU_2$, $\sup_{x\in \scrU_2}C_x=C_{\scrU_2} < + \infty$, taking supremum in the above inequality and then expectation, we obtain
\begin{equation}
\label{eq:Bregman_rates_proof_4}
     \Ex[\Theta_{\rm Feas}(\hat{x}^S\vert\scrU_{2})] 
      \leq\frac{\sup_{x \in \scrU_2}\scrE^{0}(x)}{K T_S}+\frac{\sum_{s=0}^{S-1}\beta_s\tau_s  C_{\scrU_2}}{T_S}+ \frac{1}{K T_S}\Ex\left[\sup_{x \in \scrU_2}\sum_{s=0}^{S-1}\left(\sum_{k=0}^{K-1} \tau_s  M_1(x,s,k) \right)\right]. 
\end{equation}
The last term can be estimated in the same way as in the Euclidean setting. Specifically, we use Lemma \ref{lem:nemirovski_Bregman} with $Z_{k+1}^{s} = \tau_s(\opA^s_{k+1}-\opV_{\kappa}(y^s_{k+1}))$ together with the filtration $\scrH_{s,k}$. It then follows 
\begin{align}
    &\Ex \left[\sup_{x \in \scrU_2}\sum_{s=0}^{S-1}\sum_{k=0}^{K-1} \tau_s  M_1(x,s,k)\right] = \Ex\left[\sup_{x \in \scrU_2}\sum_{s=0}^{S-1}\sum_{k=0}^{K-1}  \tau_s\inner{\opA^s_{k+1}-\opV_{\kappa}(y^s_{k+1}),x-y_{k+1}^s}\right] \\
    &= \Ex\left[ \sup_{x \in \scrU_2}\sum_{s=0}^{S-1}\sum_{k=0}^{K-1}  \tau_s\inner{\opA^s_{k+1}-\opV_{\kappa}(y^s_{k+1}),x}\right] \\
    &\leq \sup_{x \in \scrU_2} D(x,x_0^0)+ \frac{1}{2}\sum_{s=0}^{S-1}\sum_{k=0}^{K-1}\Ex[\norm{Z_{k+1}^{s}}^{2}_{*}] \leq \sup_{x \in \scrU_2} D(x,x_0^0)+\frac{1}{2}\sum_{s=0}^{S-1}\sum_{k=0}^{K-1} \tau_s^2 L_s^2  \norm{y^{s}_{k+1}-w^{s}}^{2},
\end{align}
where we have used in the second equality that $ y_{k+1}^s $ is $ \scrH_{s,k}$-measurable and $ \Ex[Z_{k+1}^s \vert \scrH_{s,k}] = 0$. For the last inequality we have used the fact that $\Ex[\norm{X-\Ex(X)}^{2}]\leq\Ex[\norm{X}^{2}]$, implying 
\begin{align*}
         \Ex\left[\norm{Z_{k+1}^s}^{2}\right] & =  \tau_s^2 \Ex\left[\norm{(\opV_s^{\xi_k^s}(y_{k+1}^s)-\opV^{\xi_k^s}_s(w^s)) - (\opV_s(y_{k+1}^s)+ \opV_s(w^s))}^{2}\right]\\
         &\leq  \tau_s^2\Ex\left[\norm{\opV_s^{\xi_k^s}(y_{k+1}^s)-\opV^{\xi_k^s}_s(w^s)}^{2}\right]\\
         &\leq \tau_s^2 L_s^2  \Ex \left[\norm{y_{k+1}^s-w^s}^2\right].
\end{align*}
We further use the stepsize assumption $\tau_s^2L_s^2\leq (1-\alpha)/4$ and Lemma~\ref{lem:bounded_Bregman}\ref{eq:sum_iterates_Bregman} to obtain 
\begin{align*}
    &\Ex\left[\sup_{x \in \scrU_2}\sum_{s=0}^{S-1}\sum_{k=0}^{K-1} \tau_s  M_1(x,s,k)\right]  \leq \sup_{x \in \scrU_2} D(x,x_0^0)+\frac{1}{8} \left(4\scrE^{0}(x^{*})+4K\sum_{s=0}^{S-1}\tau_{s}h(\beta_{s}p^{*})\right)\\
    &=\sup_{x \in \scrU_2} D(x,x_0^0)+\frac{1}{2} \left(\scrE^{0}(x^{*})+K\sum_{s=0}^{S-1}\tau_{s}h(\beta_{s}p^{*})\right),
\end{align*}
which gives us the final bound
   \begin{align*}
     \Ex[\Theta_{\rm Feas}(\hat{x}^S\vert\scrU_{2})] 
      \leq&\frac{\sup_{x \in \scrU_2}\scrE^{0}(x)+\sup_{x \in \scrU_2}D(x,x_0^0)+\scrE^{0}(x^{*})/2}{K T_S}+\frac{\sum_{s=0}^{S-1}\beta_s\tau_s  C_{\scrU_2}}{T_S}\\
      &+ \frac{\sum_{s=0}^{S-1}\tau_{s}h(\beta_{s}p^{*})}{2 T_S}. 
   \end{align*}

To estimate the optimality gap, we divide \eqref{eq:Bregman_rates_proof_1} by $\beta_s$ and take $x=x_2^*\in \scrS_2$, which implies that $0\leq -H^{(\opF_{2},g_{2})}(x_2^*,y^{s}_{k+1})$ for any $k,s$. By monotonicity of $\opF_1$, we have $-H^{(\opF_{1},g_{1})}(x,y^{s}_{k+1}) \geq H^{(\opF_{1},g_{1})}(y^{s}_{k+1},x)$.  Thus, we get
\begin{equation}
\label{eq:Bregman_rates_proof_5}
      \tau_s \sum_{k=1}^{K-1} H^{(\opF_{1},g_{1})}(y^{s}_{k+1},x_2^*)
      \leq\frac{1}{\beta_s}\scrE^{s}(x_2^*)-\frac{1}{\beta_s}\scrE^{s+1}(x_2^*)+\frac{\tau_s}{\beta_s}\sum_{k=0}^{K-1}  M_1(x,s,k) + \frac{1}{\beta_s}\sum_{k=0}^{K-1}M_2(s,k). 
\end{equation}
For any fixed $x$, we have that $\scrE^{s}(x)$ is a.s. bounded by some constant $\tilde{C}_x$. Using this, summing these inequalities, and using the convexity of $g_2$, we obtain, for any $x_2^*\in \scrS_2$
\begin{equation}
\label{eq:Bregman_rates_proof_6}
      H^{(\opF_{1},g_{1})}(\hat{x}^s,x_2^*)
      \leq\frac{\tilde{C}_{x_2^*}}{\beta_SK T_S}+\frac{1}{K T_S}  \left(\sum_{s=0}^{S-1} \frac{\tau_s}{\beta_s}\sum_{k=0}^{K-1}  M_1(x,s,k) + \sum_{s=0}^{S-1}\frac{1}{\beta_s}\sum_{k=0}^{K-1}M_2(s,k)\right). 
\end{equation}
Since for any compact $\scrU_1$, $\sup_{x\in \scrU_1\cap \scrS_2}\tilde{C}_x=C_{\scrU_1} < + \infty$, taking supremum in the above inequality and then expectation, we obtain
\begin{equation}
\label{eq:Bregman_rates_proof_7}
     \Ex[\Theta_{\rm Opt}(\hat{x}^S\vert\scrU_{1}\cap \scrS_2)] 
      \leq\frac{C_{\scrU_1}}{\beta_SK T_S}+ \frac{1}{K T_S}\Ex\left[\sup_{x \in \scrU_{1}\cap \scrS_2}\sum_{s=0}^{S-1}\frac{\tau_s}{\beta_s}\left(\sum_{k=0}^{K-1}   M_1(x,s,k) \right)\right]. 
\end{equation}
The last term can be estimated in the same way as above, by using Lemma \ref{lem:nemirovski_Bregman} with $Z_{k+1}^{s} = \frac{\tau_s}{\beta_s}(\opA^s_{k+1}-\opV_{\kappa}(y^s_{k+1}))$ 
\begin{align*}
    &\Ex\left[\sup_{x \in \scrU_{1}\cap \scrS_2}\sum_{s=0}^{S-1}\sum_{k=0}^{K-1} \frac{\tau_s}{\beta_s}  M_1(x,s,k)\right] \leq \sup_{x \in \scrU_{1}\cap \scrS_2} D(x,x_0^0)+\frac{1}{2}\sum_{s=0}^{S-1}\sum_{k=0}^{K-1} \frac{\tau_s^2 L_s^2}{\beta_s^2}  \norm{y^{s}_{k+1}-w^{s}}^{2}\\
    &\leq \sup_{x \in \scrU_{1}\cap \scrS_2} D(x,x_0^0) + \frac{1}{2\beta_S^2} \left(\scrE^{0}(x^{*})+K\sum_{s=0}^{S-1}\tau_{s}h(\beta_{s}p^{*})\right),
\end{align*}
where we used that $\beta_s$ is decreasing, the stepsize assumption $\tau_s^2L_s^2\leq (1-\alpha)/4$ and Lemma~\ref{lem:bounded_Bregman}\ref{eq:sum_iterates_Bregman}. This gives us finally
\begin{equation}
\label{eq:Bregman_theta_opt}
   \begin{aligned}
     \Ex[\Theta_{\rm Opt}(\hat{x}^S\vert\scrU_{1}\cap \scrS_2)] 
      &\leq\frac{C_{\scrU_1}}{\beta_SK T_S} + \frac{\sup_{x \in \scrU_{1}\cap \scrS_2} D(x,x_0^0)}{K T_S}\\
      &+  \frac{\scrE^{0}(x^{*})}{2 K\beta_S^2 T_S}+\frac{\sum_{s=0}^{S-1}\tau_{s}h(\beta_{s}p^{*})}{2 \beta_S^2 T_S}. 
   \end{aligned}
\end{equation}

Let us choose $\tau_s=\bar{\tau}\leq \frac{\sqrt{1-\alpha}}{2L_{1}}$ and $\beta_s=\frac{1}{(K(s+1))^\delta}$ with $\delta \in (0,1/2)$. Then, we have (cf. the derivation after \eqref{eq:betapower}) $h(\beta_{s}p^{*})\leq \frac{\rho-1}{\rho}\alpha^{-\frac{1}{\rho-1}}\beta_{s}^{\frac{\rho}{\rho-1}}\norm{p^{*}}^{\frac{\rho}{\rho-1}}=C_{\rho} K^{-\delta \rho^{*}} (s+1)^{-\delta \rho^{*}}\leq C_{\rho} K^{-1} (s+1)^{-\delta \rho^{*}}$, where we used that $\delta \rho^{*}>1$. Hence, $\sum_{s=0}^{S-1}h(\beta_{s}p^{*}) \leq \tilde{h}_{\rho}/K$, where $\tilde{h}_{\rho}$ is obtained from $\bar{h}_{\rho}$ defined in \eqref{eq:barh} with the particular choice $a=1,b=1$. We also have $\sum_{s=0}^{S-1}\beta_s=\sum_{s=0}^{S-1}(Ks)^{-\delta}\leq \frac{K^{-\delta}}{1-\delta} \left(S^{1-\delta} + \delta\right)$.

Denoting 
$R^2(x^{*},x_0^0)\eqdef\sup_{x \in \scrU_2}\scrE^{0}(x)+\sup_{x \in \scrU_2}D(x,x_0^0)+\scrE^{0}(x^{*})$, we obtain from \eqref{eq:Bregman_theta_feas} the following bound on the feasibility gap
\begin{equation}
\begin{aligned}
\label{eq:Bregman_theta_feas_expl}
     &\Ex[\Theta_{\rm Feas}(\hat{x}^S\vert\scrU_{2})] 
      \leq\frac{R^2}{K T_S}+\frac{\sum_{s=0}^{S-1}\beta_s\tau_s  C_{\scrU_2}}{T_S}+ \frac{\sum_{s=0}^{S-1}\tau_{s}h(\beta_{s}p^{*})}{2 T_S} \\
      &\leq \frac{R^2}{\bar{\tau}KS}+\frac{\frac{K^{-\delta}}{1-\delta} \left(S^{1-\delta}+ \delta\right)}{S}C_{\scrU_2}+\frac{\tilde{h}_{\rho}/K}{2S}\\
      &\leq \frac{R^2}{\bar{\tau}KS} + \frac{4C_{\scrU_2}}{(KS)^{\delta}} +\frac{\tilde{h}_{\rho}}{2KS} . 
\end{aligned}
\end{equation}
For the refined version of the optimality gap, we use \eqref{eq:Bregman_theta_opt}
\begin{align*}
    &\Ex[\Theta_{\rm Opt}(\hat{x}^S\vert\scrU_{1}\cap \scrS_2)] 
      \leq\frac{C_{\scrU_1}}{\beta_SK T_S} + \frac{\sup_{x \in \scrU_{1}\cap \scrS_2} D(x,x_0^0)}{K T_S}
      +  \frac{\scrE^{0}(x^{*})}{2 K\beta_S^2 T_S}+\frac{\sum_{s=0}^{S-1}\tau_{s}h(\beta_{s}p^{*})}{2 \beta_S^2 T_S}\\
      &\leq \frac{C_{\scrU_1}}{\bar{\tau}KS\cdot (KS)^{-\delta}}     
      + \frac{\sup_{x \in \scrU_{1}\cap \scrS_2} D(x,x_0^0)}{\bar{\tau}KS}+  \frac{\scrE^{0}(x^{*})}{2\bar{\tau} KS \cdot (KS)^{-2\delta}}     
      +\frac{\tilde{h}_{\rho}/K}{2 S \cdot (KS)^{-2\delta}}\\
      &= \frac{C_{\scrU_1}}{\bar{\tau}(KS)^{1-\delta}}     
      + \frac{\sup_{x \in \scrU_{1}\cap \scrS_2} D(x,x_0^0)}{\bar{\tau}KS}+  \frac{\scrE^{0}(x^{*})}{2\bar{\tau}  (KS)^{1-2\delta}}     
      +\frac{\tilde{h}_{\rho} }{2 (KS)^{1-2\delta}}.
\end{align*}

%% file: App_Numerics.tex
%

In this section, we will explore two numerical examples to illustrate our theoretical results.
In both examples, we use the a matrix game, with matrix $M\in\R^{n\times m}$
$$\max_{y\in \Delta_m}\min_{x\in \Delta_n}x^\top My,$$
where for any $l\in\Z$, $\Delta_l$ is the unit simplex in dimension $l$.
For some, predefined $\nu\in\Z$ we take
\begin{align}\label{eq:M_choice}
M = I_{\nu} \otimes U 
,\quad U = \begin{pmatrix}
    1&-1\\
    -1&1
\end{pmatrix},
\end{align}
where $\otimes$ denotes the Kronecker product, and $U$ is the matrix associated with classical 'matching pennies' game.
The set of optimal strategies is $x^\ast = \frac{1}{2}\mu \otimes (1,1)^T$, and  $y^\ast=\frac{1}{2}\eta \otimes (1,1)^T$
where $(\mu_s)_{s\in [\nu]}, (\eta_s)_{s\in[\nu]} \in \Delta_\nu$.
For $z=(x,y)$ the optimal solution of this matrix game is given by a solution to the variational inequality with data 
$\tilde{g}(z)=\indicator_{\Delta_n}(x)+ \indicator_{\Delta_m}(y)$ and
$$\tilde{F}(z)=\begin{bmatrix}
    My\\
    -M^Tx
\end{bmatrix}=\sum_{j\in[m]} \begin{bmatrix} 
M^jy_j\\
\boldsymbol{0}_n\\
\end{bmatrix}+\sum_{i\in[n]} \begin{bmatrix} 
\boldsymbol{0}_n\\
M_ix_i
\end{bmatrix}.$$
$M^j$ is the $j$th column of $M$, and $M_i$ is the $i$th row of $M$ as a column vector.

\subsection{Equilibrium Selection}\label{app:Num_Eq_Selection}

The equilibrium selection problem finds an equilibrium of a game, among all possible equilibrium points, which minimizes a certain objective $f(z)$. Considering the objective function $f(z)=\frac{1}{2}\norm{z}^2$, the equilibrium selection problem is given by
\begin{align}
\min_{z\in \R^{n+m}}& \norm{z}^2\label{eq:eq_selection}\\
\text{s.t. }& \begin{pmatrix}x\\y\end{pmatrix}\in \begin{pmatrix}\argmin_{\tilde{x}\in\Delta_n}\tilde{x}^\top M y\\
\argmax_{\tilde{y}\in\Delta_m}x^\top M \tilde{y}
\end{pmatrix}.\nonumber
\end{align}
This problem can be cast as a hierarchical VI of form~\eqref{eq:P} with data $g_1(z) = 0 ,\; g_2(z) = \indicator_{\Delta_n}(x) + \indicator_{\Delta_m}(y)$, and 
\begin{align*}
    F_1(z)=\sum_{j\in[m]} \begin{bmatrix} \boldsymbol{0}_n\\e_j^my_j\end{bmatrix}+\sum_{i\in[n]}\begin{bmatrix} e_i^nx_i\\\boldsymbol{0}_m\end{bmatrix}
,\quad F_2=\tilde{F},
\end{align*}
where $e^n_i$ is the $i$th vector in the standard basis of $\R^n$.
Therefore, at each iteration $k$ of the algorithms $G_k(z) = g_2(z)=\indicator_{\Delta^n}(x) + \indicator_{\Delta^m}(y)$.


Since both $\opF_1$ and $\opF_2$ can be written as a finite sum of $n+m$ components, we set the random variable $\xi$ to be the set of components chosen at each iteration. 
Specifically, $\xi=(i,j)$ with probability $Q((i,j))=c_ir_j$, where $(c_i)_{i\in [n]} \in\Delta_n$ and $(r_j)_{j\in[m]}\in \Delta_m$, and we sample the random operators
via so called "importance-sampling' as in \cite{alacaoglu2022stochastic}: 
\begin{align*}
    \opF_1^\xi(z) = \begin{pmatrix}
        \frac{1}{c_i} e_i^n x_i\\
        \frac{1}{r_j} e_j^m y_j
    \end{pmatrix},\ \opF_2^\xi(z) = \begin{pmatrix}
        \frac{1}{r_j} M^jy_j\\
        -\frac{1}{c_i} M_ix_i.
    \end{pmatrix},
\end{align*}
where $D_i$ and $D^j$ denote the $i$-th row respective $j$-th column of a matrix $D$. One evaluation of $V_k(z)=F_2(z)+\beta_kF_1(z)$ has therefore complexity $\mathcal{O}(2nm)$ and of $V_k^\xi(z)=F_2^{\xi}(z)+\beta_kF_1^{\xi}(z)$ is $\mathcal{O}(n+m)$.

In the following, we present the operators and definition of $Q$ for each of the algorithms, as well as the setup and results of the computational study for this example.   
\subsubsection{Setting for Algorithm~\ref{alg:Euclid}}
In each iteration of Algorithm~\ref{alg:Euclid}, we are required to compute 
\begin{align}\label{eq:prox_euclid}\prox_{G_k}(z) = \left(\Pi_{\Delta_n}(x), \Pi_{\Delta_m}(y)\right),\end{align} where $\Pi_{\scrC}$ denotes the euclidean orthogonal projection onto set $\scrC$. We assume that we can compute $\prox_{G_k}$ in $\tilde{\mathcal{O}}(n+m)$ operations and therefore ignore it in computing the overall complexity of the Algorithm. The total computational cost of one iteration of Algorithm \ref{alg:Euclid} is therefore on average $\mathcal{O}(2\theta nm + n+m)$.

We define the sampling probabilities to be
\begin{align*}
r_j = \frac{\norm{M^j}^2}{\norm{M}_{F}^2},\quad c_i = \frac{\norm{M_i^2}}{\norm{M}_{F}^2}.
\end{align*} We can obtain a Lipschitz constant $\scrL_k$ satisfying Assumption~\ref{ass:Oracle} via
\begin{align*}
\Ex_{\xi\sim Q} \left[\norm{V_k^\xi(z)}^2_2\right]  &= \Ex \norm{ \beta_k \begin{pmatrix}
        \frac{1}{c_i} e_i^n x_i\\
        \frac{1}{r_j} e_j^m y_j
    \end{pmatrix}+\begin{pmatrix}
        \frac{1}{r_j} M^jy_j\\
        -\frac{1}{c_i} M_ix_i
    \end{pmatrix} }_2^2 \leq 2 \beta_k^2\Ex \norm{ \begin{pmatrix}
        \frac{1}{c_i} e_i^n x_i\\
        \frac{1}{r_j} e_j^m y_j
    \end{pmatrix}  }_2^2 + 2\Ex \norm{ \begin{pmatrix}
        \frac{1}{r_j} M^jy_j\\
        -\frac{1}{c_i} M_ix_i
    \end{pmatrix} }_2^2  \\
    &= 2  \beta_k^2 \left(\underset{i \sim c}{\Ex} \norm{\frac{1}{c_i}e^n_ix_i}_2^2 + \underset{j \sim r}{\Ex} \norm{\frac{1}{r_j}e^m_jy_j}_2^2\right) + 2\norm{M}_F^2 \norm{z}_2^2 \\
    &=2  \beta_k^2\left( \sum_{i\in[n]} c_i \norm{\frac{1}{c_i}e^n_ix_i}_2^2 + \sum_{j\in[m]} \norm{\frac{1}{r_j}e^m_jy_j}_2^2\right) +2\norm{M}_F^2 \norm{z}_2^2\\
    &=2 \beta_k^2\left( \sum_{i\in[n]} \frac{1}{c_i}(x_i)^2 + \sum_{j\in[m]} \frac{1}{r_j}(y_j)^2 \right)+\norm{M}_F^2 \norm{z}_2^2\\
    &\leq2 \left( \beta_k^2 \max_{i,j} \left\{\frac{1}{c_i},\frac{1}{r_j}\right\}+\norm{M}_F^2 \right) \norm{z}_2^2\\
\end{align*}
where we have bounded $\Ex\norm{F_2^\xi(z)}_2^2 \leq \norm{M}_F^2 \norm{z}_2^2$ as in \cite{alacaoglu2022stochastic}. Thus, we can conclude,
by the linearity of $V_k^{\xi}$ and
 Jensen's inequality and 
\begin{align}\label{eq:Eqlb_const_compute_Lk}
   \Ex\left[\norm{V_k^{\xi}(u)-V_k^{\xi}(v)}\right]\leq  &\sqrt{\Ex\left[\norm{V_k^{\xi}(u)-V_k^{\xi}(v)}^2_2\right]}=\sqrt{\Ex\left[\norm{V_k^{\xi}(u-v)}^2_2\right]}\leq \scrL_k\norm{u-v}
   \end{align}
with
$$\scrL_k = \sqrt{2(\beta_k^2{\max_{i,j} \left\{\frac{1}{c_i},\frac{1}{r_j}\right\}}^2+\norm{M}_F^2)}.$$

\subsubsection{Setting for Algorithm~\ref{alg:MirrorProx}.}
For this setting we look at two possible geometries. 

\paragraph{$\ell_1$-norm.}
  For $z = (x,y) \in \R^{n+m}$,  the first geometry, is associated with the norm \begin{equation}\label{eq:l1_norm} \norm{z} = \sqrt{\norm{x}_1^2 + \norm{y}_1^2},\end{equation} with the corresponding dual norm $\norm{z^\ast}_\ast = \sqrt{\norm{x}_{\infty}^2 + \norm{y}_{\infty}^2}$. In this geometry, we use the distance generating function as the sum of negative entropy functions over $x$ and $y$, i.e., $$\distance(z)=\sum_{i=1}^{n+m} z_i\log(z_i)=\sum_{i\in[n]} x_i\log(x_i)+\sum_{j\in[m]} y_j\log(y_j),$$ which is strongly convex with respect to the defined norm over the simplex.  
Thus, at each outer iteration $s$ of Algorithm \ref{alg:MirrorProx}, and for some $u\in \R^{2v} ,\;w,v\in \Delta_n\times \Delta_m$ the following operator must be computed 
\begin{align}\label{eq:bregman_prox} T(u,v,w)&=\argmin\{\inner{u,z}+G_s(z)+\alpha D(z,v)+(1-\alpha)D(z,w)\}\\
&=\argmin_{z\in \Delta_n\times\Delta_m} \inner{u,z}+\alpha \sum_{i=1}^{n+m} z_i\log \frac{z_i}{v_i}+(1-\alpha) \sum_{i=1}^{n+m} z_i\log \frac{z_i}{w_i}.\nonumber
\end{align}
Thus, due to the separability of the operator in each index, for $z=T(u,v,w)$ we obtain
$$z_i=\tilde{z}_i/\pi_i,\;\tilde{z}_i=v_i^{\alpha}w_i^{1-\alpha}\exp(-u_i),\quad i \in [n+m],$$
where $\pi_i$ is a normalization constant with $\pi_1=\ldots=\pi_n=\sum_{i\in[n]}{\tilde{z}_i}$ and $\pi_{n+1}=\ldots=\pi_{n+m}=\sum_{i=n+1}^{n+m}\tilde{z_i}$ ensuring that the components $x$ and $y$ sum to one. 

 
For a matrix $D$, denote $\norm{D}_{\max} = \max_{i,j}|D_{ij}|$.
 Distribution $Q$ is now updated at  each iteration, and is denoted by $Q_{u,v}$ as it depends on vectors $u,v\in \R^{n+m}$. Specifically, given  $u = (u^x,u^y),\ v =(v^x,v^y)$, we compute.
\begin{equation}\label{eq:def_cr_l1}
c_i=\frac{\left|u_i^x-v^x_i\right|}{\norm{u^x - v^x}_1},\quad r_j = \frac{\left|u_j^y - v_j^y\right|}{\norm{u^y-v^y}_1},
\end{equation}
and $Q_{u,v}((i,j))=c_ir_j$.
In view of Assumption \ref{ass:Oracle_Bregman}, we can compute the Lipschitz-constant $L_s$ via
\begin{align*}
    \Ex_{\xi \sim Q_{u,v}} \left[\norm{\opV_s^{\xi}(u) - \opV_s^{\xi}(v)}^2_\ast\right] &\leq 2 \beta_s^2 \Ex_{\xi \sim Q{u,v}} \left[\norm{\opF_1^{\xi}(u) - \opF_1^{\xi}(v)}^2_\ast\right] + 2 \Ex_{\xi \sim Q_{u,v}}  \left[\norm{F_2^{\xi}(u) - F_2^{\xi}(v)}_\ast^2\right]\\
    &\leq 2\beta_s^2 \left(\underset{i \sim c}{\Ex} \left[\norm{\frac{1}{c_i}e^n_i(u^x_i - v^x_i)}_{\infty}^2\right] + \underset{j \sim r}{\Ex} \left[\norm{\frac{1}{r_j}e^m_j(u^y_j-v^y_j)}_{\infty}^2 \right) + 2 \norm{M}_{\max}^2 \norm{u-v}^2\right]\\
    &=2\beta_s^2 \left( \sum_{i\in[n]} \frac{1}{c_i}(u^x_i - v^x_i)^2 + \sum_j^{m} \frac{1}{r_j}(u^y_j-v^y_j)^2 \right)+ 2 \norm{M}_{\max}^2 \norm{u-v}^2\\
    &=2 \beta_s^2 \left( \sum_{i\in[n]} \norm{u^x-v^x}_1 \abs{u^x_i - v^x_i}+ \sum_j^{m} \norm{u^y-v^y}_1 \abs{u^y_i - v^y_i} \right)+ 2 \norm{M}_{\max}^2 \norm{u-v}^2\\
    &\leq {2}\beta_s^2 \norm{u-v}^2+ 2 \norm{M}_{\max}^2 \norm{u-v}^2\\
    &=({2}\beta_s^2+2\norm{M}_{\max}^2)\norm{u-v}^2,
\end{align*}
so $L_s=\sqrt{2(\beta_s^2+\norm{M}^2_{\max})}.$

\paragraph{$\ell_2$-norm.}
We also look at the standard Euclidean geometry, i.e., with $d(z)=\norm{z}_2$.
In this case, 
\begin{equation}\label{eq:alg_l2_prox} T(u,v,w)=\prox_{G_k}(\alpha v+(1-\alpha)w - u),\end{equation}
where the proximal operator is as in \eqref{eq:prox_euclid}. 
Moreover, similarly to the $\ell_1$ geometry we define $Q_{u,v}$ dependent on  vectors $u$ and $v$, but change the definition of $c$ and $r$ to be
\begin{equation}\label{eq:def_cr_l2}
c_i=\frac{\left(u_i^x-v^x_i\right)^2}{\norm{u^x - v^x}_1}_2^2,\quad r_j = \frac{\left(u_j^y - v_j^y\right)^2}{\norm{u^y-v^y}_2^2}.
\end{equation}
Thus, under this geometry we have
\begin{align*}
    \Ex_{\xi \sim Q_{u,v}} \left[\norm{\opV_s^{\xi}(u) - \opV_s^{\xi}(v)}^2_2\right] &\leq 2 \beta_s^2 \Ex_{\xi \sim Q{u,v}} \left[\norm{\opF_1^{\xi}(u) - \opF_1^{\xi}(v)}^2_2\right] + 2 \Ex_{\xi \sim Q_{u,v}}  \left[\norm{\opF_2^{\xi}(u) - \opF_2^{\xi}(v)}_2^2\right]\\
    &\leq 2\beta_s^2 \left(\underset{i \sim c}{\Ex} \left[\norm{\frac{1}{c_i}e^n_i(u^x_i - v^x_i)}_{2}^2\right] + \underset{j \sim r}{\Ex} \left[\norm{\frac{1}{r_j}e^m_j(u^y_j-v^y_j)}_2^2 \right]\right)\\
    &+ 2 \underset{i \sim c}{\Ex} \left[\norm{\frac{1}{c_i}M_i(u^x_i - v^x_i)}_{2}^2\right] + \underset{j \sim r}{\Ex} \left[\norm{\frac{1}{r_j}M^j(u^y_j-v^y_j)}_2^2\right]\\
    &=2\beta_s^2 \left( \sum_{i\in[n]} \frac{1}{c_i}(u^x_i - v^x_i)^2 + \sum_j^{m} \frac{1}{r_j}(u^y_j-v^y_j)^2 \right)\\ 
    &+ 2\left(\sum_{i\in[n]} \frac{1}{c_i}\norm{M_i}_2^2(u^x_i - v^x_i)^2 + \sum_j^{m} \frac{1}{r_j}\norm{M^j}_2^2(u^y_j-v^y_j)^2 \right)\\
    &=2  \left( \sum_{i\in[n]} \norm{u^x-v^y}_2^2(\norm{M_i}_2^2+\beta_s^2) + \sum_j^{m} \norm{u^y-v^y}_2^2(\norm{M^j}_2^2+\beta_s^2)\right)\\
    &\leq {2}(\beta_s^2+\norm{M}_F^2) \norm{u-v}^2,
\end{align*}
where the last inequality is due to the definition of $c$ and $r$.
Similarly to the previous setting, we have $L_s=\sqrt{2(\beta_s^2+\norm{M}_F^2)}$ satisfies Assumption~\ref{ass:Oracle_Bregman}.

\subsubsection{Experiment setup and results.}\label{sec:parameters_ex1}

Note that for our choice of matrix $M$ in \eqref{eq:M_choice}, the equilibrium selection problem~\eqref{eq:eq_selection}  has a unique solution given by 
$x^*=\frac{1}{n}{\bf 1}$,
and $y^*=\frac{1}{m}{\bf 1}$.
In our experiment, we set $\nu=100$ so $n=m=2l=200$.

We compare the deterministic EG algorithms,
Algorithm~\ref{alg:Euclid}, and Algorithm~\ref{alg:MirrorProx} with both geometries. 
We chose $\delta=0.1$ for all algorithms.
For all setting, as well as the deterministic EG we use $\delta=0.1$. In Algorithm~\ref{alg:Euclid} we use probability $\theta=0.1$, interpolation parameter $\alpha=1-\theta$,  regularization sequence $\beta_k \eqdef 1/(k+1)^{\delta}$ and step size sequence  $\tau_k \eqdef \sqrt{\theta}/(2\scrL_k)$ where $\scrL_k$ is as defined above. In both setting of Algorithm~\ref{alg:MirrorProx}, we use interpolation parameter  $\alpha=1-\theta$ and number of inner iterations $K=1/\theta$, where $\theta=0.1$, as well as regularization sequence $\beta_s \eqdef 1/(K(s+1))^{\delta}$ and step size sequence  $\tau_s \eqdef \sqrt{\theta}/(2L_s)$, where $L_s$ is as defined above.

We compare two measures, the first is the feasibility gap, given by
\begin{equation}\label{eq:feasibility_gap}\text{Feas Gap}(z)=\max_{j\in[m]} (M^\top x)_j-\min_{i\in[n]} (My)_i,\end{equation}
and the second is the distance from the unique optimal solution $z^*=(x^*,y^*)$ given by $\norm{z-z^*}^2_2.$

To compare between the methods, we take into account the number of operator evaluation needed by each method. We follow the comparison method outlined in \cite[Appendix E]{malitsky2020golden}, and define an epoch by the number of operation done to compute a full operator $\opV_k$. Thus, defining $c$ to be the ratio between the number of operations required to compute the stochastic operator $\opV_k^{\xi}$ and the number of operations required to compute the full operator $\opV_k$, we can compute $\text{epoch}_k$ the number of epochs used by iteration $k$ for each of the methods. Specifically, each deterministic EG iteration is equivalent to an epoch, each iteration of Algorithm~\ref{alg:Euclid} is equivalent to an $1+2c$ epochs if a full update is made and $2c$ epochs if a full update is not made, and each outer iteration in Algorithm~\ref{alg:MirrorProx} requires $2K+1$ epochs. Thus, we give the results of the algorithms as a function of the number of epochs. 

Figure~\ref{fig:eq_select_comparison} presents the performance measures with respect to the ergodic sequence $\bar{y}_k$  (equivilantly $\bar{y}_s$ in Algorithm~\ref{alg:MirrorProx}), given by the solid line. The figure also present the performance of outer iterate $w^s$, given by the dashed lines, where for Algorithm~\ref{alg:Euclid}, $w^s$ is the subset of $w_k$ for which a full update step is performed.

Observe that under this choice of parameters, the performance of Algorithm~\ref{alg:Euclid} and Algorithm~\ref{alg:MirrorProx} with $\ell_2$ geometry are similar, with a slight advantage to Algorithm~\ref{alg:MirrorProx}, both out performing the deterministic EG and Algorithm~\ref{alg:MirrorProx} with $\ell_1$ geometry. Algorithm~\ref{alg:MirrorProx} with $\ell_1$ geometry ergodic average shows inferior performance to that of the EG. Surprisingly, all stochastic methods exhibit almost linear convergence of ``last iterate'' ($w$). This results corresponds with the results reported by \cite[Appendix E]{malitsky2020golden} for the Bregman case.

\subsection{Linearly Constrained Equilibrium}\label{app:Num_Eq_lin_const}
This problem addresses finding an equilibrium, which additionally satisfies a set of linear constraints, that is, we are interesting in finding a vector $z=(x,y)$ such that \begin{equation}\label{eq:linearly_constrained}\begin{pmatrix}x\\y\end{pmatrix}\in \begin{pmatrix}\argmin_{\tilde{x}\in\Delta_n}\tilde{x}^\top M y\\
\argmax_{\tilde{y}\in\Delta_m}x^\top M \tilde{y}
\end{pmatrix}, \quad Bx+Cy=d,\end{equation}
for some given matrices $B\in \R^{\nu \times n},\ C \in \R^{\nu \times m}$ and vector $d\in \R^{\nu}$ and matrix $M$ as defined in \ref{eq:M_choice}.
The problem's linear constraints can therefore be cast as the following optimization problem
$$\min_{x\in \Delta_n, y\in \Delta_m} \frac{1}{2}\norm{Bx+Cy-d}^2. $$
Thus, using standard optimality conditions, problem~\eqref{eq:linearly_constrained} can be formulated as hierarchical VI of form \eqref{eq:P}, with $g_1=0$, $g_2=\tilde{g}$
$$F_1=\tilde{F},\quad F_2(z)=P(Bx+Cy-d)= P\left( \sum_{i\in[n]} B^ix_i+\sum_{j\in[m]} C^jy_j -d\right),\quad  P= \begin{pmatrix}B^\top\\
C^\top\end{pmatrix}.$$
This leads to $G_k$ and $\prox_{G_K}$ being as in the previous example.

Using the definition of random variable $\xi$ as in the last example, taking values as pairs $(i,j)$, and  distributed according to $Q$, where $Q((i,j))=c_ir_j$ where $c\in\Delta_n$ and $r\in \Delta_m$. Specifically, we define the random operators:
\begin{align*}
    \opF_1^\xi(z) = \begin{pmatrix}
        \frac{1}{r_j} M^jy_j\\
        -\frac{1}{c_i} M_ix_i
    \end{pmatrix},\;\opF_2^\xi(z) = 
    P( \frac{1}{c_i} B^ix_i+\frac{1}{r_j} C^jy_j -d)
\end{align*}
Thus, while computing $\opF_1$ and $\opF_2$ requires $O(2mn+2\nu(n+m))$ operations, computing $\opF_1^{\xi}$ and $\opF_2^{\xi}$ requires $O(2(n+m))$ operations (provided that matrices $PB$ and $PC$ are precomputed).
\subsubsection{Algorithm Setting}

For this example we can not use the same probability distribution as in the equilibrium-selection example, since the matrices $B$ and $C$ might have $0$-rows or $0$-columns. Thus, the iteration independent importance sampling in Algorithm~\ref{alg:Euclid} would not work in this case and so we only run Algorithm~\ref{alg:MirrorProx} with the same geometries, definition of operator $T(u,v,w)$, and iteration dependent sampling scheme as in the previous example.
Recalling that 
 $\opV^{\xi}_s=\opF^{\xi}_2+\beta_s \opF^{\xi}_1$, we are only left to compute $L_s$ satisfying Assumption~\ref{ass:Oracle_Bregman} for each of the geometries. 

\paragraph{$\ell_1$-norm.}
Recall that $Q_{u,v}((i,j))=c_ir_j$ where $c$ and $r$ are given by \eqref{eq:def_cr_l1}.
 Denoting $z\eqdef(x,y)=u-v$, we compute the Lipschitz constant $L_s$ satisfying Assumption~\ref{ass:Oracle_Bregman} via
\begin{align*}
    \Ex_{\xi\sim Q_{u,v}}\left[ \norm{V_s^{\xi}(u) - V_s^{\xi}(v)}^2_\ast\right] &\leq 3\beta_s^2 \Ex_{Q_{u,v}} \left[\norm{\begin{pmatrix}\frac{1}{r_j}M^jy_j\\
\frac{1}{c_i}M_ix_i\end{pmatrix}}^2_* \right] + 3\Ex_{Q_{u,v}}\left[\norm{\frac{1}{c_i}PB^ix_i}_*^2  \right]+3
\Ex_{Q_{u,v}}\left[\norm{\frac{1}{r_i}PC^jy_j}_*^2  \right]\\
    &=3\beta_s^2 \left( \sum_{i\in[n]} \frac{1}{c_i}x_i^2\norm{M_i}_{\max}^2 + \sum_j^{m} \frac{1}{r_j}y_j^2\norm{M^j}_{\max}^2\right)\\
    &\quad + 3\sum_{i\in[n]} \frac{1}{c_i}x_i^2\norm{PB^i}_{\max}^2+3\sum_{j\in[m]} \frac{1}{r_i}y_j^2\norm{PC^j}_{\max}^2 \\
    &\leq  3(\beta_s^2\norm{M}_{\max}^2+\max\{\norm{PB}_{\max},\norm{PC}_{\max}\}^2)\norm{z}^2
\end{align*}
where the last inequality follows from the definition of $r$ and $c$ and the norms.
Thus, 
we obtain $L_s=\sqrt{3(\max \left\{\norm{PB}^2_{\max},\norm{PC}^2_{\max}\right\}+\beta_s^2\norm{M}_{\max}^2)}$.
 
\paragraph{$\ell_2$-norm.}

In the definition of $Q_{u,v}$ the vectors $c$ and $r$ are defined as in \eqref{eq:def_cr_l2}.
Denoting $z=(x,y)=u-v$, we compute the Lipschitz constant satisfying Assumption~\ref{ass:Oracle_Bregman} via
\begin{align*}
    &\Ex_{Q_{u,v}}\left[\norm{V_k^{\xi}(u)-V_k^{\xi}(v)}^2_2\right]\leq 3\Ex_{Q_{u,v}}\left[\norm{\frac{1}{c_i}PB^ix_i}_2^2  \right]+3
\Ex_{Q_{u,v}}\left[\norm{\frac{1}{r_i}PC^jy_j}_2^2\right]+3\beta_s^2 {\Ex}_{Q_{u,v}}\left[\norm{\begin{pmatrix}\frac{\1}{r_j}M^jy_j\\
\frac{1}{c_i}M_ix_i\end{pmatrix}}^2_2 \right]  \\
& \leq 3\left(\sum_{i\in[n]}  (x_i)^2\frac{\norm{PB^i}^2_2}{c_i}+
\sum_{j\in[m]}  (y_j)^2\frac{\norm{PC^j}^2_2}{r_j}\right)+3\beta_s^2\left(
\sum_{i\in[n]} \frac{\norm{ M_i}^2_2(x_i)^2}{c_i}+\sum_{j\in[m]} \frac{\norm{M^j}^2_2(y_j)^2}{r_j}\right)\\
& \leq 3\left(\sum_{i\in[n]} \norm{u^x-v^x}_2^2\norm{PB^i}^2_2+
\sum_{j\in[m]}  \norm{u^y-v^y}_2^2\norm{PC^j}^2_2\right)+3\beta_s^2\left(
\sum_{i\in[n]} \norm{u^x-v^x}_2^2\norm{ M_i}^2_2+\sum_{j\in[m]} \norm{u^y-v^y}_2^2\norm{M^j}^2_2\right)\\
& \leq 3\left(\norm{u^x-v^x}_2^2\norm{PB}^2_F+
 \norm{u^y-v^y}_2^2\norm{PC}^2_F\right)+3\beta_s^2\left(
\norm{u^x-v^x}_2^2\norm{ M}^2_F+\norm{u^y-v^y}_2^2\norm{M}^2_F\right)\\
&\leq 3 \left( \max\left\{ \norm{PB}_F^2,\norm{PC}_F^2\right\} + \beta_s^2 \norm{M}_F^2\right)\norm{u-v}_2^2,
\end{align*}
and so $L_s =  \sqrt{3\left( \max\left\{ \norm{PB}_F^2,\norm{PC}_F^2\right\} + \beta_s^2 \norm{M}_F^2\right)}$.

\subsubsection{Experiment setup and results.}

For this experiment, we chose $\nu=100$ so $m=n=2l=200$ and data
$$
B,C =  I_l \otimes (1, 0) \in 
\R^{\nu \times n}, \text{ and } d = {\bf 1} \frac{1}{\nu} \in \R^\nu,
$$
so that for odd indices $i$, $x_i = \frac{1}{\nu} -y_i $.

We compare the deterministic EG algorithms
and Algorithm~\ref{alg:MirrorProx} with both geometries. We take $\delta=0.01$ for all algorithms.
All other parameters were selected as specified in \eqref{sec:parameters_ex1}.
For this experiment, we compare two measures: the Optimality gap, given by
\begin{equation}\label{eq:optimality_gap_Ex2}\text{Opt Gap}(z)=\max_{j\in[m]} (M^\top x)_j-\min_{i\in[n]} (My)_i,\end{equation}
and the feasibility gap given by
\begin{equation}\label{eq:optimality_gap_Ex2}\text{Feas Gap}(z)=\norm{Bx+Cy-d}^2_2.\end{equation}
To ensure a fair comparison comparison between the algorithms, we adopt the same epoch definition as in the previous experiment, and adapt the computation to the complexities of these example. Figure \ref{fig:minmax_all} reports the performance of EG and Algorithm \ref{alg:MirrorProx} for both geometries, similar to Figure~\ref{fig:eq_select_comparison}. 
As in the previous example, 
Algorithm \ref{alg:MirrorProx} with $\ell_2$ geometry has  superior performance to that of EG, while 
Algorithm \ref{alg:MirrorProx} with $\ell_1$ geometry performs worse than EG.
Additionally, we observe again the almost linear convergence of the last iterates. 
\begin{figure}[H]
    \centering
    \includegraphics[scale=0.45]{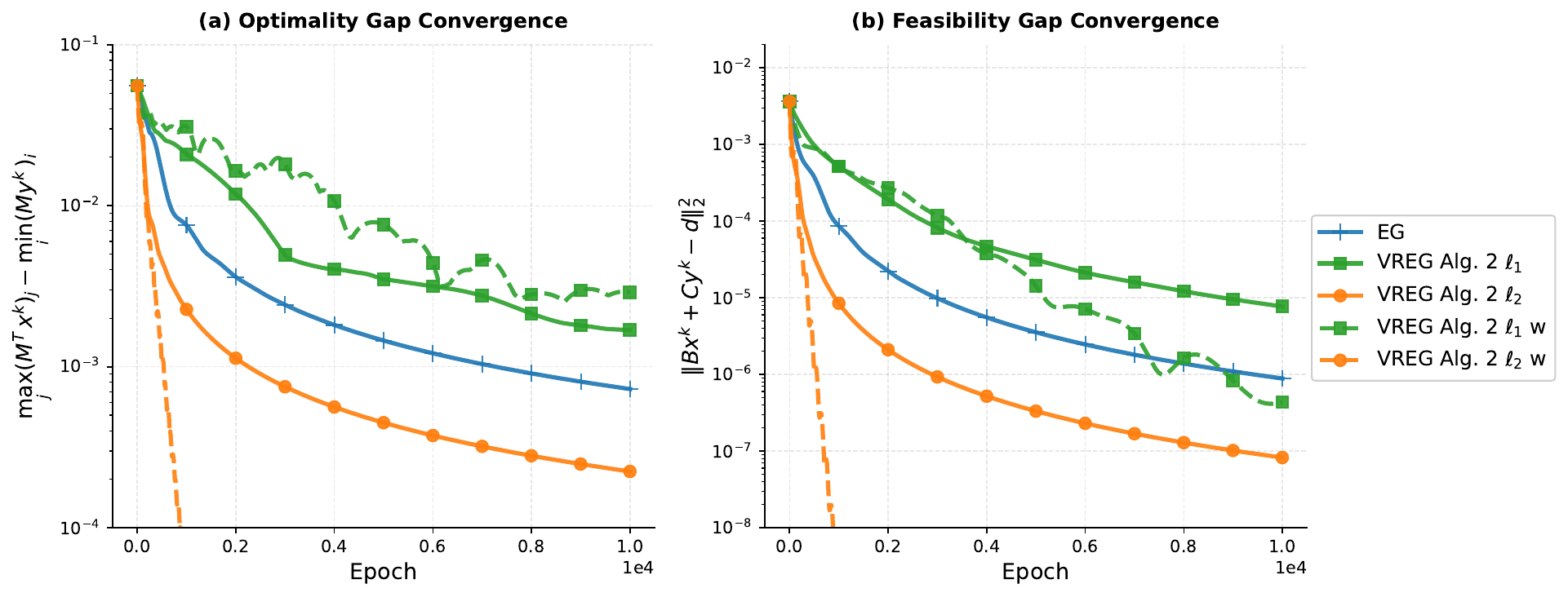}
    \caption{Ergodic average performance for $\delta=0.01$}
    \label{fig:minmax_all}
\end{figure}